\documentclass[11pt,reqno,a4paper,psamsfonts]{amsart}
\usepackage[utf8]{inputenc}
\usepackage[T1]{fontenc}
\usepackage[english]{babel}
\usepackage{amssymb,amsmath,amscd,enumerate,amsthm}
\usepackage[psamsfonts]{amsfonts}
\usepackage{latexsym}
\usepackage{amsbsy}
\usepackage{ulem}
\usepackage{pdfsync}
\usepackage{enumerate}
\usepackage{verbatim}
\usepackage{slashbox}
\usepackage{color}
\usepackage[all,color ]{xy}
\usepackage{graphicx}
\newcommand{\C}{{\mathbb{C}}}

\newcommand{\F}{{\mathcal{F}}}
\newcommand{\R}{{\mathbb{R}}}%
\newcommand{\Z}{{\mathbb{Z}}}

\let\CAL=\mathcal%
\def\mathcal#1{{\CAL#1}}%
\newcommand{\N}{\mathbb{N}}
\newcommand{\Ve}{{\mathsf{Ve}}}
\newcommand{\Ed}{{\mathsf{Ed}}}
\newcommand{\AG}{{\mathsf{A}}}

\newcommand{\MG}{{\mathsf{M}}}
\let\GA\AG
\let\GV\VG
\let\GE\EG
\let\GM\MG
\let\EF=\FE
\newcommand{\A}{{\mathsf{A}}}

\newcommand{\E}{{\mathsf{Ed}_{\A}}}
\newcommand{\AR}{{\mathsf{R}}}

\newcommand{\ER}{{\mathsf{Ed}_{\mathsf{R}}}}
\newcommand{\ARZ}{{\mathsf{R}_0}}

\newcommand{\ARU}{{\mathsf{R}_1}}

\newcommand{\AZ}{{\mathsf{Z}}}

\newcommand{\AZa}{\mathsf{Z}^{\alpha}}

\newcommand{\VZa}{\mathsf{Ve}_{\AZa}}
\newcommand{\wAZa}{\wt \AZ{}^{\alpha}}

\newcommand{\Dis}{\mathrm{Dis}^{\F^\diamond}}
\newcommand{\Fe}{\msl{Fol}}
\newcommand{\MFe}{\msl{MFol}}
\newcommand{\TFe}{\mr{Mod}_{\msl{tr}}}%{\msl{TFol_{tr}}}
\let\TFol\TFe

\def\ftrait{\;{{}^{\underline{\phantom{\longleftrightarrow}}}}\;}%
\let\sing\Sing

\newcommand{\G}{\mathcal{G}}

\renewcommand{\emph}[1]{{\textsl{#1}}}

\newcommand{\comp}{{\mr {Comp}}}

\newcommand{\msl}[1]{\mathsf{#1}}

\newcommand{\mf}[1]{{\mathfrak{#1}}}
\newcommand{\mr}[1]{{{\mathrm{#1}}}}
\newcommand{\mc}[1]{{\mathcal{#1}}}
\newcommand{\mb}[1]{{\mathbb{#1}}}
\newcommand{\wt}[1]{{\widetilde{#1}}}
\newcommand{\mbf}[1]{\mathbf{#1}}

\newcommand{\br}[1]{{\breve{#1}}}

\newcommand{\wh}[1]{{\widehat{#1}}}

\newcommand{\iso}{{\overset{\sim}{\longrightarrow}}}

\newcommand{\inte}[1]{\overset{\circ}{{#1}}}
\newcommand{\DD}[1]{\frac{\partial\phantom{ #1}}{\partial #1}}
\newtheorem{teo}{Theorem}[section]
\newtheorem{thmm}{Theorem}

\newtheorem{lema}[teo]{Lemma}

\newtheorem{prop}[teo]{Proposition}
\newtheorem{defin}[teo]{Definition}
\newtheorem{cor}[teo]{Corollary}

\newtheorem{obs2}[teo]{Remark}
\newtheorem{recap2}[teo]{Recapitulation}
\newtheorem{ex2}[teo]{Exemple}

\newenvironment{obs}{\begin{obs2}\rm}{\hfill\qed\end{obs2}}

\newenvironment{dem}{\begin{proof}[Proof]}{\end{proof}}
\newenvironment{dem2}[1]{\begin{proof}[Proof #1]}{\end{proof}}
\def\bibartp#1#2#3#4#5#6#7#8
{\bibitem{#1} {\sc #2}, {\it #3}, {#4}, {\bf t. #5}, pages { #6} to
{ #7}, {({#8})}}
\def\bibart#1#2#3#4#5#6
{\bibitem{#1} {\sc #2}, {\it #3}, {#4}, {\bf #5}, {({#6})}}
\def\bibliv#1#2#3#4#5
{\bibitem{#1} {\sc #2}, {\it #3}, {#4}, {({#5})}}
\def\bibaart#1#2#3#4
{\bibitem{#1} {\sc #2}, {\it #3}, {#4}}
\usepackage[dvipsnames]{xcolor}
\newcommand{\ED}{\mc E^{\diamond}}
\newcommand{\ME}{\MFe_{\msl{tr}}(\mc E^{\diamond})}
\newcommand{\CE}{C_{\ED}}

\title[Topological moduli]{Topological moduli space\\
for  germs of holomorphic foliations}
\date{\today}
\author{David Mar\'{\i}n, Jean-Fran\c{c}ois Mattei and \'{E}liane Salem}
\thanks{This work was partially supported by grants 
MTM2011-26674-C02-01 and MTM2015-66165-P of  Ministerio de  Econom\'{\i}a y Competitividad of Spain/FEDER}
\address{Departament de Matem\`{a}tiques \\ Universitat Aut\`{o}noma de Barcelona \\ E-08193 Bellaterra (Barcelona)\\ Spain} \email{davidmp@mat.uab.es}
\address{Institut de Math\'{e}matiques de Toulouse\\ Universit\'{e} Paul Sabatier\\ 118, Route de Narbonne\\ F-31062 Toulouse Cedex 9, France} \email{jean-francois.mattei@math.univ-toulouse.fr}
\address{Institut de Math\'{e}matiques de Jussieu-Paris Rive Gauche\\ Universit\'e Pierre et Marie Curie\\ 4 Place Jussieu, F-75252 Paris  cedex 05, France} \email{eliane.salem@imj-prg.fr}



\usepackage{tikz}

\begin{document}
\begin{abstract}
This work deals with the topological classification of germs of singular foliations on $(\C^{2},0)$. Working in a  suitable class of foliations we fix the topological invariants given by the separatrix set, the Camacho-Sad indices and the projective holonomy representations and we compute the moduli space of topological classes in terms of the cohomology of a new algebraic object that we call group-graph. 
This moduli space may be an infinite dimensional functional space but under  generic conditions we prove 
that it has finite dimension and we describe its algebraic and topological structures.
\end{abstract}
\maketitle
\tableofcontents

\section{Introduction}
This work deals with the \emph{topological classification} of germs of  \emph{singular foliations} on $(\C^{2},0)$.
To every (possibly dicritical) foliation $\F$ we can associate 
the separatrix set ${Sep}_{\F}$, that is the collection of all germs at $0\in \C^2$  of invariant irreducible analytic curves, called \emph{separatrices}, its \emph{canonical reduction map} $E_{\F}:(M_{\F},\mc E_{\F})\to(\C^{2},0)$, cf.~\cite{CCD}, and the \emph{marked exceptional divisor} 
\[\ED_{\F}=(\mc E_{\F},\Sigma_{\F}, \imath_\F),\] 
where   $\Sigma_{\F}:=\mr{Sing}(\F^\sharp)$ is the finite set consisting of the singular points of $\F^\sharp$:=  $E_{\F}^{*}\F$ 
and $\imath_\F$ is the intersection pairing of $\mc E_{\F}$  in $M_\F$. 
The topological class of ${Sep}_{\F}$ is clearly a topological invariant of $\F$. In this paper we will assume that $\F$ is a \emph{generalized curve}, i.e. $\F^\sharp$ has no saddle-node singularities.
The topological class   $[\mc E_{\F}^{\diamond}]$ of $\mc E_{\F}^{\diamond}$ (as a marked  intrinsic curve) is then a topological invariant of $\F$ because in this situation $E_{\F}$  is  also the minimal desingularization map of ${Sep}_{\F}$, cf. \cite{CLNS}. \\

We know \cite{MM3} that under some assumptions the Camacho-Sad indices of $\F^\sharp$ at the points of $\Sigma_\F$ and the holonomy representations  (up to inner automorphisms) of every component of $\mc E_\F$  are also topological invariants of the germ $\F$ at $0\in \C^2$. 
Our purpose in this work is to describe the set of all other topological invariants and highlight its geometric and algebraic structure.\\ 

\noindent\textbf{MAIN RESULT.}
\large\it
Under generic conditions, 
\begin{enumerate}[(a)]
\item
\textit{there exists an analytic family of foliations parametrized by a finite dimensional space  which gives all the topological types once we fix the topological class of the marked exceptional divisor, the Camacho-Sad indices  and the holonomy representations;}\vspace{1mm}
\item\textit{\large the quotient of this complete family by the topological equivalence relation is naturally isomorphic to the abelian group
\[
\left.\left(F\oplus B \oplus_{j=1}^{\lambda}(\C^{*}/\alpha_{j}^{\Z})\oplus(\C^{*})^{\nu}
\right)\right/ Z\,,
\]
where $\alpha_{j}\in\C^{*}$,  $F$ is a finite abelian group, $Z$ is a finite subgroup, $B$ is a direct sum of $\beta$ 
 totally disconnected subgroups of $\mb U(1)$ and 
$\lambda$, $\nu$ and $\beta$  only
depend on the  (combinatorics of the)
local types of the singularities inside the exceptional divisor.}
\end{enumerate}
\normalsize\rm

\vspace{3mm}

We will also give an explicit characterization of those foliations
satisfying Assertion~(a) in the main result above, that we will call finite type foliations.

\section{Statement of results}

\subsection{Marking of a foliation}
To give a precise sense to our problem let us call \emph{marked divisor} any collection  $\mc E^{\diamond}=(\mc E,\Sigma,\imath)$ consisting of  a compact curve with normal crossings $\mc E$ whose irreducible components  are biholomorphic to $\mb P^1$, a finite subset $\Sigma$ of $\mc E$ and a symmetric map $\imath: \mr{Comp}(\mc E)^2\to \mb Z$, $\mr{Comp}(\mc E)$ denoting the set of irreducible components of~$\mc E$. We will denote by $\mc E^d\subset \mc E$ the union of  the irreducible components of $\mc E$ that do not contain any point of $\Sigma$; we call  them \emph{dicritical components}.\\

A \emph{marking} of a foliation $\F$ by $\mc E^\diamond$ will be  a homeomorphism $f:\mc E\to\mc E_{\F}$ sending $\Sigma$ onto $\Sigma_{\F}$ compatible with the intersection pairing: 
$$\imath_\F(f(D),f(D'))=\imath(D,D').$$ 
In this way, the holonomy representations and the  Camacho-Sad indices of all  pairs $\F^\diamond:=(\F,f)$
can be now associated to  two common sets of indices: the set 
$\CE:=\mr{Comp}(\overline{\mc E\setminus\mc E^d})$ of irreducible components of $\overline{\mc E\setminus\mc E^d}$ and
 \begin{equation*}\label{ChechIndices}
I_{\mc E^\diamond}:= \{  (s,D)\in \Sigma\times\CE 
\;|\; s\in D\}\,,
\end{equation*}
by defining
\[
 \mr{CS}^{\F^\diamond}:=(\mr{CS}(\F^\diamond, D, s))_{(s,D)\in I_{\ED}},\quad
 \mr{CS}(\F^\diamond, D, s) := \mr{CS}(\F^\sharp, f(D), f(s))\,,
\]
\[
\dot{\mc H}^{\F^\diamond}:=(\dot{\mc H}^{\F^\diamond}_D)_{D\in\CE}
,\quad
\mc H_D^{\F^\diamond} := \mc H_{f(D)}^{\F^\sharp}\circ f_\ast \;:\; \pi_1(D\setminus \Sigma, \cdot) \longrightarrow\mr{Diff}(\C,0)\,,
\]
where 
$ \mc H_{f(D)}^{\F^\sharp}$  is the $\F^\sharp$-holonomy representation of $ \pi_1(f(D)\setminus \Sigma_\F, \cdot)$ in the group $\mr{Diff}(\C,0) $ of germs of holomorphic automorphisms of $(\C,0)$,   $\dot{\mc H}^{\F^\diamond}_D$ is its class  up to inner automorphisms,
$f_\ast$ is the isomorphism induced by $f$ at the fundamental groups level and $\mr{CS}(\F^\sharp, f(D), f(s))$  is the  Camacho-Sad index of $\F^\sharp$ along $f(D)$ at $f(s)$.
\\

Let us denote by $\Fe(\mc E^{\diamond})$  the set  of germs of generalized curves $\F$ at $0\in\C^2$   for which there exists a marking $f:\mc E\to\mc E_{\F}$ of $\F$ by $\mc E^{\diamond}$.
Our general goal is to describe  a generic subset of the quotient set 
\[
[\Fe(\mc E^{\diamond})]_{\mc C^0}
\]
of the set $\Fe(\mc E^{\diamond})$   by  the equivalence relation: 
\begin{itemize}
\item
\it $\F\sim_{\mc C^0}\G$  if $\F$ and $\G$  are topologically equivalent as germs at $0\in \C$.
\end{itemize}
\subsection{Globalization of topological equivalences} 
Consider now the equivalence relation:
\begin{itemize}
\item
$\F\sim_{\mc E}\G$ if  $\F^\sharp$ and $\mc G^\sharp$ are topologically conjugated, as germs along the exceptional divisors,  by a germ of a homeomorphism  $(M_\F,\mc E_\F)\to (M_{\mc G}, \mc E_{\mc G})$ which is  holomorphic 
at each point of $\Sigma_\F\setminus \mc N_\F$, 
\end{itemize}
$\mc N_\F$ denoting the subset of the singular points  of $\mc E_\F$, called \emph{nodal corners}, where the Camacho-Sad index of $\F^\sharp$ is a strictly positive real number.
Clearly relation $\sim_{\mc E}$ is stronger than $\sim_{\mc C^0}$, but they will coincide on a generic class of foliations when $\mc E^\diamond$ fulfills the following condition
\begin{enumerate}
\item[(TC)] \it The closure of each connected component of $\mc E\setminus\mc E^{d}$ contains an irreducible component $D$ with $\mr{card}(D\cap \Sigma)\neq 2$.\\
\end{enumerate}
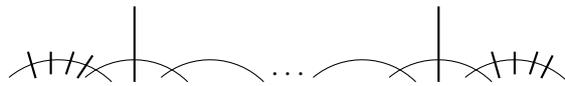
\begin{figure}[h]
\begin{tikzpicture}
\draw [thick] (5,0.05) to (4.9,.4); 
\draw [thick] (5.2,0.1) to (5.2,.4);
\draw [thick] (5.4,0.1) to (5.5,.4);
\draw [thick] (5.55,0.05) to (5.75,.35);
\draw [thick] (6.3,0) to (6.3,1);
\node at (8.35,0.1) {$\cdots$};
\draw [thick] (10.3,0) to (10.3,1);
\draw [thick] (11.1,0.05) to (11,.4); 
\draw [thick] (11.3,0.1) to (11.3,.4);
\draw [thick] (11.5,0.1) to (11.6,.4);
\draw [thick] (11.65,0.05) to (11.8,.35);
\draw  (6,0) arc [radius=1, start angle=45, end angle= 130];
\draw  (7,0) arc [radius=1, start angle=45, end angle= 130];
\draw  (8,0) arc [radius=1, start angle=45, end angle= 130];
\draw  (10,0) arc [radius=1, start angle=45, end angle= 130];
\draw  (11,0) arc [radius=1, start angle=45, end angle= 130];
\draw  (12,0) arc [radius=1, start angle=45, end angle= 130];
\end{tikzpicture}
\caption{The only situation excluded by Condition (TC), the extremity divisors being dicritical components.}
\end{figure}
To specify the notion of genericity
let us call \emph{cut-component of $\mc E_\F$} any closure $\mc C$  of  a connected component of  $\mc E_\F\setminus (\mc E_\F^d\cup\mc N_\F)$; if $\mr{card}(D\cap\Sigma)\leq 2$ for each  
 $D\in \mr{Comp}(\mc C)$ we will say that $\mc C$ is \emph{exceptional}.
Now consider the following transverse rigidity condition:
\begin{itemize}
\item[(TR)] \it Any non exceptional  cut-component of $\mc E_\F$ contains an irreducible component with topologically rigid\footnote{We recall that a subgroup $G$ of the group $\mr{Diff}(\C,0)$ of germs of biholomorphisms of $\C$ at $0$ is called \emph{topologically rigid} if every topological conjugation  between $G$ and another subgroup $G'\subset\mr{Diff}(\C,0)$ is necessarily conformal. This class contains the non-solvable groups \cite{nakai} and the non-abelian groups with dense linear part \cite[Th\'eor\`eme 2]{CerveauSad}.} holonomy group.
\end{itemize}
The Krull-open density  in $\Fe(\mc E^\diamond)$  of the  subset $\Fe_\msl{tr}(\mc E^\diamond)$  consisting of   the foliations $\F$ fulfilling  Condition (TR)  is proven in \cite{LeFloch}.
An extended version of  Main Theorem of \cite{MM3} is given in Appendix (Theorem \ref{newliftingconjugacies}). It asserts that 
\begin{thmm}\label{star}
If $\mc E^\diamond$ satisfies condition (TC) then the relations  $\sim_{\mc E}$ and $\sim_{\mc C^0}$ are equal on  $\Fe_\msl{tr}(\mc E^\diamond)$.
\end{thmm}
\noindent In other words 
\[
\left[ \Fe_{\msl{tr}}(\mc E^\diamond)\right]_{\sim_{\mc E}}=
\left[ \Fe_{\msl{tr}}(\mc E^\diamond)\right]_{\mc C^0} \;\subset\;\left[ \Fe(\mc E^\diamond)\right]_{\mc C^0}\,,
\]
 $\left[ X \right]_{\sim_{\mc E}}$ and   $\left[ X\right]_{\sim_{\mc C^0}}$ denoting the quotient of a subset   $ X\subset  \Fe(\mc E^\diamond) $ by the relations  ${\sim_{\mc E}}$ and ${\sim_{\mc C^0}}$ respectively.

\begin{obs}
This result implies that under the hypothesis (TR) and (TC) the collection of Camacho-Sad indices at the singular points of $\F^\sharp$ is a topological invariant of the germ of $\F$ at $0$. The topological classification of logarithmic foliations obtained by E. Paul shows  \cite[Th\'eor\`eme~3.5]{Paul} that Condition (TR) is necessary for this. On the other hand when Condition~(TC) is not satisfied,  it is easy to construct topologically conjugated foliations with same separatrices but different Camacho-Sad indices.
\end{obs}

\subsection{Topological moduli space of a marked foliation}\label{SubsectTopModSpace}
In order to describe $\left[ \Fe_{\msl{tr}}(\mc E^\diamond)\right]_{\sim{\mc E}}$ let us consider the set 
$\MFe_{\msl{tr}}(\mc E^{\diamond})$ 
of \emph{marked by $\mc E^\diamond$ foliations} $\F^\diamond=(\F,f)$  with $\F\in \Fe_{\msl{tr}}(\mc E^\diamond)$. We proceed in the following way:

\begin{enumerate}[A)]
\item We adapt the equivalence relation $\sim_{\mc E}$ in $\Fe_{\msl{tr}}(\mc E^\diamond)$ to $\MFe_{\msl{tr}}(\mc E^\diamond)$ by means of
\begin{itemize}
\item
\it $(\F,f)\sim_{\diamond}(\mc G,g)$ if there is a germ of homeomorphism 
$\Phi:(M_{\F},\mc E_{\F})\to( M_{\mc G},\mc E_{\mc G})$ that conjugates  $\F^\sharp$ and $\mc G^\sharp$, is holomorphic at each point of $\Sigma_\F\setminus \mc N_\F$ and its restriction to $\mc E_{\F}$ is isotopic to $g\circ f^{-1}$ by an isotopy fixing $\Sigma_\F$.\\
\end{itemize}
\item We define the \emph{Topological Teichmuller Space} as the quotient set 
 \[  \TFol (\mc E^{\diamond}) := \MFe_{\msl{tr}}(\mc E^{\diamond})/\!\! \sim_\diamond\] 
 so that the  \emph{Forgetful map} 
\[\TFol(\mc E^\diamond)
\longrightarrow 
\left[ \Fe_{\msl{tr}}(\mc E^\diamond)\right]_{\mc C^0}\,,
\quad 
[\F,f]_{\sim_\diamond}\mapsto [\F]_{\mc C^0}
\]
is well defined.
\item\label{cc} We note that the fibres of this map are exactly  the orbits of the action 
\[\dot\varphi\star\left[\F,f\right]:= \left[\F,f\circ\varphi^{-1}\right]\,,
\quad
\dot\varphi\in \mr{Mcg}(\mc E^\diamond)\,,
\quad
(\F,f)\in \MFe_{\msl{tr}}(\mc E^{\diamond})\,
,\]
on $\TFol (\mc E^{\diamond}) $ of the 
\emph{Mapping Class Group} $\mr{Mcg}(\mc E^\diamond)$ of $\mc E^\diamond$, that is the product of braid groups defined as the group of isotopy classes of $\mc C^0$-auto\-mor\-phisms of $\mc E$ fixing 
$\Sigma$ and leaving the intersection form $\imath$ invariant. Thus 
\[ \left[ \Fe_{\msl{tr}}(\mc E^\diamond)\right]_{\mc C^0}\simeq 
 \TFe(\mc E^{\diamond})
  /\mr{Mcg}(\mc E^\diamond)\,.\]
\item\label{dd} The description of  $\TFol(\mc E^\diamond)$ is obtained by fixing the Camacho-Sad indices and the holonomy representations (that are topological invariants). In other words we give a description of each nonempty fiber of the map 
\[\wt{\mc H} :
 \TFol (\mc E^{\diamond}) 
\longrightarrow
\C^{I_{\mc E^\diamond}}\times \dot{\mc R}
\,,\quad
\F^\diamond\mapsto (\mr{CS}^{\F^\diamond}, \dot{\mc H}^{\F^\diamond})\,,
\]
$\dot{\mc R}$ being  the set of  conjugacy classes of group morphisms
from the free product of the groups $\pi_1(D\setminus\Sigma,\cdot)$ for all $D\in\CE$ with values in $\mr{Diff}(\C,0)$. 
\end{enumerate}

\begin{defin}
We call \emph{topological moduli space} of $[\F^\diamond]\in \TFol(\mc E^{\diamond})$ the  fiber 
of $\wt{\mc H}$ above $\wt{\mc H}([\F^\diamond])$ that is the set
\[
\mr{Mod}([\F^\diamond]) := \left\{\left. [{\mc G}^\diamond]\in \TFol(\mc E^{\diamond})
\; \right|\;   
\mr{CS}^{\F^\diamond}=\mr{CS}^{{\mc G}^\diamond}, 
\dot{\mc H}^{\F^\diamond}=\dot{\mc H}^{{\mc G}^\diamond}
\right\}\,.
\] 
\end{defin}
We consider in Section \ref{secGroupGraph} a new algebraic notion, which we call  \emph{group-graph}, that will be the key tool in the whole paper.  It allows us,  by combining Theorems  \ref{p1} and \ref{isoH1symH1aut}, to obtain a bijection between this moduli space $\mr{Mod}([\F^\diamond])$ and the cohomology of a suitable group-graph defined in Section~\ref{Symmetry-tree-group}, namely the \emph{symmetry group-graph }$\mr{Sym}^{\F^\diamond}$.
If we call \emph{$\F^\diamond$-cut-component} of $\mc E$ any counter-image $f^{-1}(\mc C)$ of a cut-component $\mc C$ of $\mc E_\F$ and if we denote by $\A_{\F^\diamond}$ the dual graph of  the disjoint union of all the \emph{$\F^\diamond$-cut-components} of $\mc E$, then one can prove:
\begin{thmm}\label{thmB}
 If $\mc E^\diamond$ satisfies assumption (TC) and $[\F^\diamond]\in  \TFol(\mc E^{\diamond})$ we have a natural bijection 
\[\mr{Mod}([\F^\diamond])\iso H^1(\A_{\F^\diamond}, \mr{Sym}^{\F^\diamond})\,.\]
\end{thmm}
Without  any other assumption the computation of this cohomological space is difficult and the usefulness of this result is essentially theoretical. However it will  allow us in Section \ref{SectionExamples} to  construct examples for which   $\mr{Mod}([\F^\diamond])$ is an infinite dimensional functional space. 
 To get finiteness we shall  need to restrict  to some Krull open dense subsets  of $\Fe_{\msl{tr}}(\ED)$ by requiring conditions on $\mc H^{\F^\diamond}$  depending only on a   finite jet of a 1-differential form defining $\F$.

\subsection{The generic case: non-degenerate foliations}

Let us call \emph{singular chain}\footnote{Notice that a singular chain may not correspond to a chain of the dual graph of $\mc E$, in the usual sense. Indeed the interior vertices $D_i$, $0<i<\ell$,  may  meet dicritical components and the number of their adjacent edges  can be greater than two, and also  $D_0$ or $D_{\ell}$ may have only two adjacent edges. Conversely a chain of the dual graph of $\mc E$ may not be a singular chain because there may exist points of $\Sigma$ outside the singular locus of $\mc E$.}
of the dual graph   $\mc E_{\F}$ 
any sequence $D_{0},\ldots,D_{\ell}$, $\ell\ge 1$, of invariant irreducible components of $\mc E_{\F}$ 
such that: 
\begin{enumerate}[a)]
\item $D_{0}$ and $D_{\ell}$ contain at least $3$ singular points of $\F^\sharp$,
\item $D_{i}\cap \Sigma_{\F}=\{s_{i},s_{i+1}\}$ with $s_{i}=D_{i-1}\cap D_{i}$, if $1\le i\le \ell-1$.
\end{enumerate}
At all the points $s_i$, $1\leq i\leq \ell-1$, $\F^\sharp$ has the same property of normalization  and we  will 
say that the chain is  linearizable, resonant normalizable or non-normalizable,  non-resonant, if $\F^\sharp$ fulfills this property at these points~$s_i$.
\begin{defin}\label{NDfoliationsDef}
A germ of a  foliation $\F$ is called  \emph{non-degenerate} if it satisfies the following properties:
\begin{enumerate}[(i)] 
\item\label{conditiontrgener} $\F$ fulfills  condition (TR);
\item\label{holonomienonabel} the holonomy group $\mr{Im}(\mc H_D^{\F^\sharp})$ of any invariant irreducible component $D$ of $\mc E_{\F}$ with ${\rm{card}}(D\cap \Sigma_\F)\geq 3$,  is non-abelian;
\item\label{nonpreiodsingchaines} for any singular chain $D_{0},\ldots,D_{\ell}$ in $\mc E_\F$, the local holonomies of $\F^\sharp$ at the singular points $s_{i}=D_{i-1}\cap D_{i}$, $i=1,\ldots,\ell$, are non-periodic. 
\end{enumerate}
The  subset of  $\Fe_\msl{tr}(\mc E^\diamond)$ of all  non-degenerate  foliations 
will be denoted by $\Fe_{\msl{nd}}(\ED)$. 
\end{defin}
The Krull-open-density of $\Fe_{\msl{nd}}(\ED)$ in $\Fe_{\msl{tr}}(\ED)$ is given by Theorem \ref{genericiteNonDeg}.

\begin{thmm}\label{B0}
Suppose that $\mc E^\diamond$ satisfies condition (TC) and let $\F^{\diamond}=(\F,f) \in\ME$ be a   marked foliation with $\F$ non-degenerate. Then we have an identification:
\[
\mr{Mod}([\F^\diamond])\simeq\left.\left(F\oplus B \oplus_{j=1}^{\lambda}(\C^{*}/\alpha_{j}^{\Z})\oplus(\C^{*})^{\nu}
\right)\right/ Z\,,
\]
where $\alpha_{j}\in\C^{*}$,  $F$ is a finite abelian group, $Z$ is a finite subgroup, $B$ is a direct sum of $\beta$ 
 totally disconnected subgroups of $\mb U(1)$ and 
$\lambda$, $\nu$ and $\beta$ are respectively the number of linearizable, resonant normalizable 
and non-resonant non-linearizable
singular chains contained in cut-components of $\mc E_\F$, the factor $F$ corresponding to resonant non-normalizable  chains. In particular,  $\lambda+\nu$ is equal to the codimension $\tau_\F$ of $\F$ defined in \ref{defindetau}.
\end{thmm}
\noindent The naturality of this identification will be explained by Assertion~(\ref{famcomplthm}) in Theorem \ref{thm} below.

\subsection{Foliations of finite topological type}
We have seen that  $\mr{Mod}([\F^\diamond])$  is  endowed with a very specific structure of topological group of finite dimension if $\F\in \Fe_{\msl{nd}}(\ED)$. However this finiteness  property continues to be valid for a larger class of foliations that we shall call \emph{finite type foliations}. The set  $\Fe_{\mr{ft}}(\ED)\supset \Fe_{\msl{nd}}(\ED)$ of these foliations is defined in Section~\ref{sectionfoliationsTF} and it is optimal for finiteness as Example 3 in  Section~\ref{SectionExamples} shows. Furthermore we  obtain \emph{complete families} of marked foliations parametrized by  finite dimensional spaces, completeness meaning that the family contains all the topological types
of marked foliations by  $\mc E^\diamond$ with  prescribed Camacho-Sad and $\dot{\mc H}$ invariants.

\begin{thmm}\label{thm} 
Suppose that $\mc E^\diamond$ satisfies condition (TC) and let $\F^{\diamond}=(\F,f) \in\ME$ be a 
marked foliation with $\F$ of finite type.
Then $\mr{Mod}([\F^{\diamond}])$ admits an abelian group structure
with identity element $[\F^{\diamond}]$ 
such that:
\begin{enumerate}[(a)]
\item\label{strgroupetf} there is an exact sequence
\begin{equation}\label{exseqmainThm}
\Z^{p}\to \C^{\tau_{\F}}\stackrel{\Lambda}{\to} \mr{Mod}([\F^{\diamond}])\stackrel{\Gamma}{\to} {\bf D}\to 0
\end{equation}
where $\bf D$ is a totally disconnected topological abelian  group and $\tau_\F$ is the codimension of $\F$ given by Definition~\ref{defindetau};
\item\label{famcomplthm}  given a section  $i\mapsto [\F_{i},f_{i}]\in \Gamma^{-1}(i)$ of $\Gamma$, 
there is a family parametrized by $i\in \bf D$ of SL-equisingular deformations  
$(\F^{U}_{i,t})_{t\in \C^{\tau_{\F}}}$ of $\F_i$ such that for all $t\in\C^{\tau_{\F}}$ we have 
$[\F^{U}_{i,t},f_{i,t}^U]=\Lambda(t)\cdot [\F_{i},f_{i}]$,  $f_{i,t}^U$ being the marking induced by $f_{i}$ and the dot $\cdot$ denoting the operation in the group $\mr{Mod}([\F^{\diamond}])$.
\end{enumerate}
\end{thmm}
\noindent The group $\bf D$ will be specified in the proof  (Step (i) of Section~\ref{proof}): it is a quotient of a product of a finite family of totally discontinuous subgroups of $\mb U(1):=\{z\in \C\;|\; |z|=1\}$,  that can be uncountable.
However let us highlight that $\bf D$ is ``generically finite'' in the following sense: 

\begin{enumerate}[-]
\item \it there is a  subset $Z$  of zero measure in 
the algebraic subset 
$\overline{{\mr{CS}(\MFe_{\msl{tr}}(\mc E^{\diamond}))}}$ of $\C^{I_{\ED}}$
such that if $\mr{CS}([\F^{\diamond}])\not\in Z$, the formally linearizable singularities of $\F^\sharp$ are holomorphically linearizable, and in this case we can prove that $\mbf D$ is finite.
\end{enumerate}
The notion of \emph{SL-equisingular deformation} roughly means equireductibility and constancy of Camacho-Sad indices and holonomy, the markings $f_i$  extending continuously without ambiguity.
All this is made more precise in {Definition~\ref{def-SL}}, Step (\ref{recallSL}) of Section~\ref{proof}.

As a direct consequence of the proof we can see that if $\wt{\mc H}([\F^{\diamond}])=\wt{\mc H}([\G^{\diamond}])$ then
the sets $\mr{Mod}([\F^{\diamond}])$ and $\mr{Mod}([\G^{\diamond}])$ coincide. However their respective abelian group structures are related by the map $\mu\mapsto \gamma \mu$ where $\gamma=[\G^{\diamond}]\in\mr{Mod}([\F^{\diamond}])$.\\

The paper is organized as follows.
Theorem~\ref{thmB} is proven in Sections~\ref{secGroupGraph}, \ref{sec4} and~\ref{Symmetry-tree-group}. 
In Section~\ref{sectionfoliationsTF} we prove the technical theorem~\ref{isoassymarsym} that will be used throughout the following sections. The group structure on the moduli space is given by Theorem~\ref{strutgaaut}. Theorem~\ref{B0} is proven in Section~\ref{ModulusNonDeg}.
Some applications of Theorems~\ref{thmB} and~\ref{B0} are discussed in Section~\ref{SectionExamples}. 
Sections~\ref{sec9} and~\ref{proof} are devoted to the proof of Theorem~\ref{thm}.
Finally, Theorem~\ref{star}  is a direct consequence of Theorem~\ref{newliftingconjugacies} proven in Appendix.

\section{Group-graphs}\label{secGroupGraph}

In this section we will introduce and study the algebraic notion of group-graph which differs in an essential way from the notion of graph of groups introduced by Serre in \cite{Serre} and that will be a key tool of this work.\\

Let $\GA$ be a finite graph with vertex set $\GV$ and edge set $\GE$.
Denote by 
$$I_{\AG}:=\{(v,e)\in \GV\times \GE\,|\, v\in \partial e\}$$ the set of oriented edges of $\GA$.

\begin{defin} 
A \emph{group-graph $G$ over  $\GA$} is the data of groups $G_{v}$ and $G_{e}$ for each vertex $v\in \GV$ and each edge $e\in \GE$, and of group morphisms $\rho_{v}^{e}:G_{v}\to G_{e}$ for each $(v,e)\in I_{\GA}$ which are called \emph{restriction maps}. A \emph{morphim} $\alpha:F\to G$ between group-graphs over the same graph $\GA$ is given by group morphisms $\alpha_{v}:F_{v}\to G_{v}$ and $\alpha_{e}:F_{e}\to G_{e}$ such that the diagram 
$$\begin{array}{rcl}
F_{v} & \stackrel{\alpha_{v}}{\longrightarrow} & G_{v}\\
\rho_{v}^{e}\downarrow
& & \downarrow {\rho_{v}^{e}}\\
F_{e} & \stackrel{\alpha_{e}}{\longrightarrow} & G_{e}
\end{array}$$
commutes for each $(v,e)\in I_{\GA}$. A group-graph $G$ is called \emph{abelian} if all the groups $G_{v}$ and $G_{e}$ are abelian.
\end{defin}

%\begin{figure}[ht]
%$$\xymatrix{\bullet_{v_{1}}\ar@{-}[r]^{e_{1}} & {\color{white}v}\bullet_{v_{0}}
%\ar@{-}[r]^{e_{2}}\ar@{-}[d]^{e_{3}} & \bullet_{v_{2}}\\ & {\color{white}v!}\bullet_{v_{3}}  &}\qquad
%\xymatrix{G_{v_{1}}\ar[r]&G_{e_{1}}& G_{v_{0}}\ar[l]\ar[d]\ar[r] & G_{e_{2}}& G_{v_{2}}\ar[l]\\
%&& G_{e_{3}}&&\\ &&G_{v_{3}}\ar[u]&&}$$
%\caption{The dual tree of the simple cusp $y^2-x^3=0$ and a group-graph $G$ over this tree.}
%\end{figure}

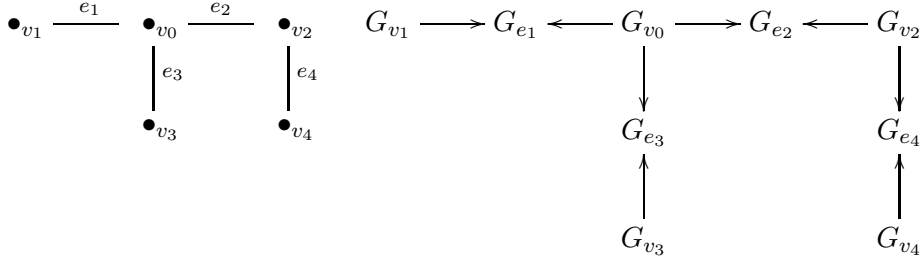
\begin{figure}[ht]
$$\hspace{-3mm}\xymatrix{\bullet_{v_{1}}\ar@{-}[r]^{e_{1}} & {\color{white}v}\bullet_{v_{0}}
\ar@{-}[r]^{e_{2}}\ar@{-}[d]^{e_{3}} &  {\color{white}v}\bullet_{v_{2}}\ar@{-}[d]^{e_{4}}\\ & {\color{white}v}\bullet_{v_{3}}  &  {\color{white}v}\bullet_{v_4}}\quad
\xymatrix{G_{v_{1}}\ar[r]&G_{e_{1}}& G_{v_{0}}\ar[l]\ar[d]\ar[r] & G_{e_{2}}& G_{v_{2}}\ar[d]\ar[l]\\
&& G_{e_{3}}&& G_{e_4}\\ &&G_{v_{3}}\ar[u]&&G_{v_4}\ar[u]}$$
\caption{Dual tree of $(y^2 - x^3)^2 - x^2 y^3=0$ and a group-graph $G$ over it.}
\end{figure}

\begin{obs} One can define in a natural way the notions of image and  kernel of a group-graph morphism $\alpha:F\to G$, which are group-graphs over the same graph.
In the abelian case the cokernel can  also be defined as a group-graph. We also have an obvious notion of restriction of a group-graph over a graph to a subgraph.
\end{obs}

\begin{defin}
Let $G$ be a group-graph over a graph $\AG$. The \emph{cochain complex}
of $G$ consists of 
$$C^{0}(\GA,G):=\prod\limits_{v\in \GV}G_{v}\qquad \hbox{and} \qquad
C^{1}(\GA,G):=\prod\limits_{(v,e)\in I_{\GA}}G_{v,e}\,, \quad G_{v,e}:=G_e\,,$$
jointly with the right action $C^{0}(\GA,G)\times C^{1}(\GA,G)\to C^{1}(\GA,G)$ given by
$$(g_{v})\,\star_G\,(g_{v,e}):=(\rho_{v}^{e}(g_{v})^{-1}g_{v,e}\rho_{v'}^{e}(g_{v'}))$$
where $\partial e=\{v,v'\}$. 
\end{defin}
The set of \emph{0-cocycles} is the subset $H^{0}(\GA,G)$ of $C^{0}(\GA,G)$ of all elements $(g_{v})$ satisfying the relations
$\rho_{v}^{e}(g_{v})=\rho_{v'}^{e}(g_{v'})$ whenever $\partial e=\{v,v'\}$.
Let us consider the set of \emph{$1$-cocycles}
$$Z^{1}(\GA,G):=\{(g_{v,e})\in C^{1}(\GA,G)\,|\, g_{v,e}g_{v',e}=1 \hbox{  when $\partial e=\{v,v'\}$}
\}
\subset C^{1}(\GA,G)$$  which is invariant by the action of $C^{0}(\GA,G)$ and its quotient, the 1-cohomology set: $$H^{1}(\GA,G):=Z^{1}(\GA,G)/C^{0}(\GA,G).$$
\begin{obs}\label{obsab} The cochains 
$C^{0}(\GA,G)$ and $C^{1}(\GA,G)$ are groups but in general $H^{0}(\GA,G)$ and $H^{1}(\GA,G)$ are merely sets (although $Z^{1}(\GA,G)$ is in bijection with $\prod\limits_{e\in \E}G_{e}$ which is a group). 
However, if the group-graph $G$ is abelian then we can consider the group $C^{2}(\GA,G):=\prod\limits_{e\in\GE}G_e$ 
and the morphisms $\partial^{0}:C^{0}(\GA,G)\to C^{1}(\GA,G)$ and $\partial^{1}:C^{1}(\GA,G)\to C^{2}(\GA,G)$ given by 
\[
\partial^{0}(g_{v}):=(g_{v})\star_G(1)=(\rho_{v}^{e}(g_{v})^{-1}\rho_{v'}^{e}(g_{v'}))\quad\text{and}\quad\partial^{1}(g_{v,e})=(g_{v,e}g_{v',e}) 
\]
whenever $\partial e=\{v,v'\}$.
It turns out that $\partial^{1}\circ\partial^{0}=1$ and we obtain a complex 
$$C^{*}(\GA,G):C^{0}(\GA,G)\stackrel{\partial^{0}}{\to}C^{1}(\GA,G)\stackrel{\partial^{1}}{\to} C^{2}(\GA,G)\to 1$$
whose cohomology is $H^{0}(C^{*}(\GA,G))=H^{0}(\GA,G)=\ker\partial^{0}$, 
$H^{1}(C^{*}(\GA,G))=H^{1}(\GA,G)=Z^{1}(\GA,G)/\partial^{0} C^{0}(\GA,G)$ and $H^{2}(C^{*}(\GA,G))=1$ because $Z^{1}(\GA,G)=\ker\partial^{1}$ and $\mr{coker}\,\partial^{1}=1$. 
\end{obs}

The following result is straightforward.

\begin{lema}[Functoriality]
Every morphism $\alpha:F\to G$ of group-graphs over the same graph $\GA$ induces well-defined maps
$\alpha_{i}:H^{i}(\GA,F)\to H^{i}(\GA,G)$, for $i=0,1$, given by
$$\alpha_{0}(g_{v})=(\alpha_{v}(g_{v}))\quad\text{and}\quad\alpha_{1}([g_{v,e}])=[\alpha_{e}(g_{v,e})].$$
Moreover, if $F$ and $G$ are abelian then $\alpha_{i}$ are morphisms.
\end{lema}

A short sequence $1\to F\stackrel{\alpha}{\to}G\stackrel{\beta}{\to}J\to 1$ of morphisms of group-graphs over the same graph $\GA$ is exact if for all $a\in\GV\cup\GE$ the corresponding short sequence of groups $1\to F_{a}\stackrel{\alpha_{a}}{\to}G_{a}\stackrel{\beta_{a}}{\to}J_{a}\to 1$ is exact.
In  the abelian case the complexes of abelian groups considered in Remark~\ref{obsab} fit into a short exact sequence $1
\to C^{*}(\GA,F)\to C^{*}(\GA,G)\to C^{*}(\GA,J)\to 1$. The following results are classical.

\begin{lema}[Long exact sequence]\label{LES}
If $1\to F\to G\to J\to 1$ is a short  exact sequence of abelian group-graphs over the same graph $\GA$, then 
there is a long exact sequence $1\to H^{0}(\GA,F)\to H^{0}(\GA,G)\to H^{0}(\GA,J)\to H^{1}(\GA,F)\to H^{1}(\GA,G)\to H^{1}(\GA,J)\to 1$.
\end{lema}

\begin{lema}[Mayer-Vietoris]\label{MV}
Let $G$ be an abelian group-graph over a graph~$\GA$. If $\GA_{0}$ and $\GA_{1}$ are subgraphs of $\GA$ such that $\GA=\GA_{0}\cup \GA_{1}$ then there is an exact sequence $1\to H^{0}(\GA,G)\to H^{0}(\GA_{0},G)\oplus H^{0}(\GA_{1},G)\to H^{0}(\GA_{0}\cap \GA_{1},G)\to H^{1}(\GA,G)\to H^{1}(\GA_{0},G)\oplus H^{1}(\GA_{1},G)\to H^{1}(\GA_{0}\cap \GA_{1},G)\to 1$.
\end{lema}

\begin{proof}
We have a short exact sequence of complexes of abelian groups
$$1\to C^{*}(\GA,G)\to C^{*}(\GA_{0},G)\oplus C^{*}(\GA_{1},G)\to C^{*}(\GA_{0}\cap\GA_{1},G)\to 1$$
and we consider the long exact sequence of cohomology.
\end{proof}

\begin{defin}\label{Avalence}
The \emph{valency in $\GA$} or the \emph{$\GA$-valency}  of a vertex $v$ of $\GA$ is the cardinality $\mr{val}_{\GA}(v)$
of the set $\{e\in \E \;;\; v\in \partial e\}$. We say that $v$ is an \emph{extremity of} $\GA$ if $\mr{val}_{\GA}(v)=1$; in that case we will write $v\in\partial\GA$.
\end{defin}

A \emph{partial dead branch $(\MG, v_0)$ of $\GA$} is the data of a vertex $v_0$ of $\GA$ called \emph{attaching point} and a connected subgraph $\GM$ of $\GA$ such that:
\begin{itemize}
\item $\GM$ contains an extremity $v'_0$ of $\GA$,
\item all its vertices are of valency 2 in $\GA$, except possibly its extremities that are $v'_0$ and $v_0$. 
\end{itemize} 
Notice that $\GM$ is always a chain. When 
$\GM\neq \GA$ and $\mr{val}_{\GA}(v_0)\geq 3$ one says that $\GM$ is a \emph{dead branch of $\GA$}. 
We define a \emph{total order $<_{{}_\GM}$} on the sets of its vertices $\Ve_{\GM}:=\{v_0,\ldots, v_\ell :=v'_0\}$ and of its edges $\Ed_{\GM}:=\{e_1,\ldots,e_\ell\}$ with $\partial e_j=\{v_{j-1},v_j\}$, by setting $v_0<_{{}_\GM}\cdots <_{{}_\GM}v_\ell$ and $e_1<_{{}_\GM}\cdots <_{{}_\GM}e_\ell$ $j=1,\ldots,\ell$.
\begin{defin}
For a group-graph $G$, we say that a partial dead branch $\GM$ is $G$-\emph{repulsive}  if the morphisms $\rho_v^{e} : G_v\to G_e$ are surjective for all $e\in \Ed_{\GM}$ such that $\partial e=\{v,v'\}$ and $v'<_{{}_\GM} v$.
\end{defin}
Now we will give a process that will allow us to restrict  a group-graph to a subgraph  without changing its cohomology.
\begin{defin}\label{pruning}
 If $(\GM,v_0)$ is a  partial dead branch of $\GA$, the \emph{pruning $\br \GA$ of $\GM$ in $\GA$}, at the attaching point $v_0$, is the subgraph $\br \GA = (\GA \setminus \GM) \cup \{v_0\}$.
 \end{defin}
\begin{teo}[Pruning]\label{elagage}
Let $G$ be a (not necessarily abelian) group-graph over $\GA$ and $\GM$ a $G$-repulsive partial dead branch of $\GA$ then there is a natural  bijection $H^{1}(\GA,G)\iso H^{1}(\br\GA,\br G)$ where $\br\GA$ is the pruning of $\GM$ in $\GA$ and $\br G$ is the restriction of $G$ to $\br\GA$. Moreover, if $G$ is abelian, this bijection is an isomorphism of groups.
\end{teo}
Before giving the proof let us notice that the natural projections $\mathrm{pr}^{i}: C^i(\GA, G)\to C^i(\br\GA, \br G)$, $i=0,1$, are group morphisms commuting with the actions $\star_G$ and $\star_{\br G}$ and inducing a natural map $\mathrm{pr}^1_\ast : H^1(\GA, G) \to H^1(\br \GA, \br G)$.

On the other hand, we have an ‘‘extension by 1'' map $\mathrm{ext} : C^1(\br\GA, \br G) \to C^1(\GA, G)$
 such that  $\mathrm{pr}^1\circ \mathrm{ext}$ is the identity map, $\mathrm{ext}(Z^1(\br \GA, \br G))\subset Z^1(\GA, G)$ and 
$$
\mathrm{ext}\circ \mathrm{pr}^1((g_{v,e}))=(g'_{v,e})\,,\quad \hbox{with}\quad g'_{v,e}=\left\{
\begin{array}{ccl}
g_{v,e}& \hbox{ if } & (v,e)\in I_\GA\setminus I_{\GM}\\
1& \hbox{ if } & (v,e)\in  I_{\GM}\,.
\end{array}
\right.
$$
\begin{dem2}{of Theorem~\ref{elagage}} We will see that the $G$-repulsivity of $\GM$ implies that: \\
\indent\textit{a)}  $\mathrm{ext}$ induces a map $\mathrm{ext}_\ast : H^1(\br \GA , \br G) \to H^1(\GA , G)$ such that  $\mathrm{pr}^1_\ast\circ \mathrm{ext}_\ast$ is the identity of $H^1(\br \GA, \br G)$,\\ 
\indent\textit{b)} 
$\mathrm{ext}_\ast\circ \mathrm{pr}^1_\ast$ is the identity of $H^1( \GA,  G)$. When $G$ is abelian the maps $\mathrm{pr}^1_\ast$ and $\mathrm{ext}_\ast$ are group morphisms,  one the  inverse  of the other, and the map $H^{1}(\GA,G)\iso H^{1}(\br\GA,\br G)$ induced by $\mathrm{pr}^1_\ast$ is trivially an isomorphism.\\

\indent \textit{a)} By $G$-repulsivity we have the following diagram of groups and morphisms
\[G_{v_0}\stackrel{\rho_{v_0}^{e_1}}{\rightarrow} G_{e_1}\stackrel{\rho^{e_1}_{v_1}}{\twoheadleftarrow} 
G_{v_1}\stackrel{\rho_{v_1}^{e_2}}{\rightarrow} G_{e_2}\stackrel{\rho^{e_2}_{v_2}}{\twoheadleftarrow}
G_{v_2}\rightarrow
\cdots \cdots
\twoheadleftarrow
G_{v_{\ell-1}}
\stackrel{\rho_{v_{\ell-1}}^{e_\ell}}{\rightarrow} G_{e_\ell}\stackrel{\rho^{e_\ell}_{v_\ell}}{\twoheadleftarrow} 
G_{v_\ell}\]
Let $\br h^1=(\br h_{ v, \ e})$ and $\br g^1=(\br g_{ v,  e})$ be two cohomologous elements  if $Z^1(\br \GA, \br G)$:
$$
\br h_{v,e}=\rho _{v}^{e}(\br h_v)^{-1}\, \br g_{v,e}\, \rho_{v'}^{e}(\br h_{v'})\,,\quad \br h^0:= (\br h_{v})\in C^0(\br \GA, \br G)\,.
$$
We will determine $h^0:=(h_v)\in C^0(\GA, G)$ such that $h^0\star_{G} \mathrm{ext}(\br h^1)=\mathrm{ext}(\br g^1)$. We define $h_v=\br h_{v}$ if $v\in \Ve_{\br A}$. For $v\notin \Ve_{\br A}$ it is sufficient to solve the system of $\ell$ equations:
$$
\left\{
\begin{array}{cclc}
1 & = & \rho^{e_1}_{v_0}(h_{v_0})^{-1}\cdot 1 \cdot \rho^{e_1}_{v_1}(h_{v_1}) \,,& h_{v_1}\in G_{v_1}\\
 & \vdots& \\
1 & = & \rho^{e_\ell}_{v_{\ell-1}}(h_{v_{\ell-1}})^{-1}\cdot 1 \cdot \rho^{e_\ell}_{v_\ell}(h_{v_\ell})\,, & h_{v_\ell}\in G_{v_\ell}
\end{array}
\right.
$$
This can be easily done using the surjectivity of the maps $\rho_{v_j}^{e_j}$, $j=1,\ldots , \ell$.\\

\indent \textit{b)} We have to prove that for each $(g_{v, e})\in Z^1(\GA, G)$ there is $(g_v)\in C^0(\GA, G)$ such that:
$$
\mathrm{ext}\circ \mathrm{pr^1}\left((g_{v,e})\right)=\left(\rho^{e}_v(g_v)^{-1} \, g_{v,e}\, \rho_{v'}^{e}(g_{v'})\right)\,.
$$
We define $g_v=1$ when $v\in \Ve_{\br \GA}$, so that in particular $g_{v_0}=1$ and therefore $\rho_{v_0}^{e_1}(g_{v_{0}})=1$. As the maps $\rho_{v_j}^{e_j}$, $1\leq j\leq \ell$, are surjective, the following system of $\ell$ equations has a solution with $g_{v_1}\in G_{v_1},\ldots, g_{v_\ell}\in G_{v_\ell}$
$$
\left\{
\begin{array}{ccl}
1 & = & \rho^{e_1}_{v_0}(g_{v_0})^{-1}\,g_{v_0,e_1} \, \rho^{e_1}_{v_1}(g_{v_1}) \\
1 & = & \rho^{e_2}_{v_1}(g_{v_1})^{-1}\,g_{v_1,e_2} \, \rho^{e_2}_{v_2}(g_{v_2}) \\
 & \vdots& \\
1 & = & \rho^{e_\ell}_{v_{\ell-1}}(g_{v_{\ell-1}})^{-1}\, g_{v_{\ell-1}, e_\ell} \,  \rho^{e_\ell}_{v_\ell}(g_{v_\ell})\,.
\end{array}
\right.
$$
\end{dem2}
\begin{obs}\label{pruningiteres}
By repeating this process we obtain  a subtree $\GA_{\mathrm{pr}}$ of $\GA$ such that the restriction $G_{\mathrm{pr}}$ of $G$ to $\GA_{\mathrm{pr}}$ has no $G_{\mathrm{pr}}$-repulsive partial dead branches and the map $\mathrm{ext} : Z^1(\GA_{\mathrm{pr}}, G_\mathrm{pr}) \to Z^1(\GA,G)$ of extension by 1, induces an isomorphism  $H^1(\GA_{\mathrm{pr}}, G_\mathrm{pr}) \iso H^1(\GA, G)$. In particular if all the morphisms $\rho^e_v : G_v\rightarrow G_e$ are surjective, the subtree $\GA_{\mathrm{pr}}$ is  reduced to a single vertex and   $H^1(\GA, G)$ is trivial.
\end{obs}

\section{Automorphism group-graph}\label{sec4}
Let us fix once for all a marked divisor $\mc E^\diamond=(\mc E, \Sigma, \imath)$ and a marked foliation $\F^{\diamond}=(\F,f)\in\EF$. We recall that $\mc E^d$ is the union of \emph{dicritical components} $D$ of $\mc E^\diamond$, i.e.  $D\cap \Sigma=\emptyset$, cf. Introduction.\\

The \emph{dual tree $\mathsf{A}_{\mc E}$ of} $\mc E$ is the graph having  $\comp(\mc E)$  and $\sing(\mc E)$ as sets $\msl{Ve}_{\msl{A}_\mc E}$ of vertices and $\msl{Ed}_{\msl{A}_\mc E}$ of edges respectively, with $\partial s=\{D,D'\}$ whenever $D\cap D'=s$.
\begin{defin}\label{CutGraphDef}
The subgraph of $\mathsf A_{\mc E}$ obtained  by removing the vertices  associated to $\mc E^d$, the edges attached with these vertices and the edges  corresponding by $f$  to nodal singularities of  $\F^\sharp$, is called \emph{cut-graph of $\F^\diamond$}. We denote it by $\A_{\F^\diamond}$, or more simply by $\A$ when there is not ambiguity.
\end{defin}
This graph is a disjoint union of trees, denoted by $\A_{\F^\diamond}^i$, or more simply by $\A^i$, that can  be considered as  the dual graphs of the $\F$-cut-components of $\mc E_{\F}$ defined  in the introduction. It only depends on the class $[\F,f]$ and in fact it is constant on the fiber $\mr{CS}^{-1}(\mr{CS}([\F^{\diamond}]))$ that contains $\mr{Mod}([\F^{\diamond}])$.
Notice that if $G$ is a group-graph on $\A$ then 
\[H^{1}(\A_{\F^\diamond},G)= \prod\limits_{i}H^{1}(\A_{\F^\diamond}^{i},G_{i})\]
where $G_{i}$ is the restriction of $G$ to $\A_{\F^\diamond}^{i}$.

\begin{defin}[The {group-graph} $\rm{Aut}^{\F^{\diamond}}$] For $s\in\msl{Ed}_{\msl{A}_{\F^\diamond}}$ and $D\in\msl{Ve}_{\msl{A}_{\F^\diamond}}$, let us denote by
\begin{enumerate}[$\bullet$]
\item $\mr{Aut}^{\F^{\diamond}}_{s}$  the group of germs at $f(s)$  of holomorphic automorphisms of $\F^\sharp$, 
\item $\mr{Aut}^{\F^{\diamond}}_{D}$ the group of germs along $f(D)$ of continuous automorphisms of $\F^\sharp$ preserving $\mc E_{\F}$, that are holomorphic at each 
singular point of $\F^\sharp$ that is not a nodal corner (cf. Introduction) and whose restriction to $f(D)\setminus \mr{Sing}(\F^\sharp)$ is homotopic to the identity.
\end{enumerate}
We define by these data the \emph{automorphism group-graph}  $\mr{Aut}^{\F^{\diamond}}$  over $\A_{\F^\diamond}$, the  morphisms $\rho_{D}^{s}$, $s\in D$,  being  just the restriction maps.
\end{defin}

\begin{obs}\label{saturation} 
If $D$ is not dicritical, the elements of $\mr{Aut}^{\F^{\diamond}}_{D}$ are transversely holomorphic at each point of 
$f(D)\setminus \Sigma_\F$, with $\Sigma_\F:=\mr{Sing}(\F^\sharp)$, because they are holomorphic on an open set whose saturation by $\F^\sharp$ is a neighborhood of $f(D)\setminus \Sigma_\F$, cf. \cite[Theorem~A]{MM4} or \cite[Theorem~2]{CamachoRosas}. 
\end{obs}

Now we will assign to each topological class $\mf g \in \mr{Mod}([\F^{\diamond}])$ a cohomo\-lo\-gy class $\mbf i_{\F^\diamond}(\mf g) \in H^{1}(\A_{\F^\diamond},\mr{Aut}^{\F^{\diamond}})$. To do that we fix a representative $(\G,g)$ of~$\mf g$.

\begin{defin}\label{goodfib}  A \emph{good fibration}  along an invariant component $g(D)$ of $\mc E_{\mc G}$ is a germ along $g(D)$ of smooth map from a neighborhood of $g(D)$ to $g(D)$, that is holomorphic at each singular point of $\mc G^\sharp$, equal to the identity on $g(D)$  and constant  on each  component adjacent to $g(D)$. 
\end{defin}
\noindent Clearly good fibrations along invariant components always exist.\\

For each vertex $D$ of $\A_{\F^\diamond}$ we fix a regular point $o_{D}\in D$ and good fibrations along $f(D)$ and $g(D)$.
Up to isotopy we can suppose that $f$ and $g$ are holomorphic at the singular points of  $\mc E$.
Since $\wt{\mc H}([\F,f])=\wt{\mc H}([\G,g])$, cf. \S\ref{SubsectTopModSpace}.\ref{dd}, 
for each $D\in\CE$  the conjugating  map of these holonomies between 
 fibers of the good fibrations over $f(o_{D})$ and $g(o_{D})$ 
 extends  by the lifting path method to a unique 
germ of a transversely holomorphic homeomorphism 
\begin{equation*}%\label{conjugsemilocales}
\psi_D : (M_{\F}, f(D))\rightarrow (M_{\mc G}, g(D))\,, \quad \psi_{D}(\F^\sharp) = \mc G^\sharp\,.
\end{equation*}
 that conjugates the fibrations and the foliations and such that the restriction to $f(D)$ is $f^{-1}\circ g$.
Classically, since the good fibrations are holomorphic near the singular points, $\psi_D$ is  also holomorphic at these points. However, a precision must be given for the extension to the nodal singularities. The good fibrations at the nodal singularities can be chosen to coincide with the projections given by linearizing holomorphic coordinates of the node. In fact, we can choose linearizing holomorphic coordinates for the nodes $f(s)$ and $g(s)$ such that $g^{-1}\circ f$ writes as the identity. Thus, on the linear model $yx^{-\lambda}$, $\lambda\in\R^{+}\setminus\mb Q$, we have an automorphism which is the identity on $|x|=1$. Since it commutes with the linear holonomy $y\mapsto ye^{2i\pi\lambda}$ we deduce that it is linear on the fibers and it extends to $|x|\le 1$ by linearity.

Thus, for each edge $s\in {\msl {Ed}}_{\msl A_{\F^\diamond}}$, the germs at $f(s)$ 
$$
\varphi_{D,s} =\psi_D^{-1}\circ\psi_{D'},\quad \varphi_{D',s} =\psi_{D'}^{-1}\circ\psi_{D}\,, \quad D\cap D'=s\,,
$$
are holomorphic automorphisms of $\F^\sharp$ and the 1-cocycle $\mbf c^1:=(\varphi_{D,s})$, with ${(D,s)\in I_{\A_{\F^\diamond}}}$, is an element of  $Z^{1}(\A_{\F^\diamond},\mr{Aut}^{\F^{\diamond}})$.
If we choose another element $(\breve{\mathcal{G}}, \breve{g})$ in  $\mf g$, and
good fibrations for $\breve{\G}$, taking in the same way homeomorphisms   
 \[
 \breve{\psi}_D : (M_{\F}, f(D))\rightarrow (M_{\mc G}, \breve{g}(D))\,,\quad 
 \breve{\psi}_{D}(\F^\sharp ) = \breve{\G}^\sharp\,,
 \]
 we obtain another  1-cocycle  
 \[
 \breve{\mbf{c}}^1=(\breve{\varphi}_{D,s})=
 (\breve{\psi}_{D}^{-1}\circ\breve{\psi}_{D'})\in H^{1}(\A_{\F^\diamond},\mr{Aut}^{\F^\diamond})\,.
 \]
Since $(\G,g)\sim_{\diamond}(\breve{\G},\breve{g})$ there is a homeomorphism $\Phi$ between neighborhoods of the exceptional divisors $\mc E_{\mc G}$ and $\mc E_{\breve{\mc G}}$ of the reductions of these foliations that conjugates  $\mathcal{G}^\sharp$ and $\breve{\mathcal{G}}^\sharp$ and is holomorphic at the singular points of $\mc G^\sharp$, except perhaps at the nodal corners; moreover,  when restricted to $\mc E_{\mc G}$, $\Phi$  is isotopic to $\breve{g} \circ g^{-1}$.

Let us denote by $\Phi_D$ the germs of $\Phi$ along the invariant components $D$ of $\mathcal{E}_{\mathcal{G}}$. One checks easily that the 0-cochain  $\mbf{c}^0 =
(\psi_{D}^{-1}\circ\Phi_{D}^{-1}\circ\breve{\psi}_{D})\in 
C^{0}(\A_{\F^\diamond},\mr{Aut}^{\F^\diamond})$ fulfills 
$\mbf c^0\,\star\, \mbf c^1=\breve{\mbf c}^1$. 
This proves that  the cohomology class of $\mbf c^1$ does not depend on the choice of the representative of the class $\mf g$ neither on the good fibrations used to define it. We denote this cohomology class by $\mbf i_{\F^{\diamond}}(\mf g)$.
\begin{teo}\label{p1}
The map
\begin{equation}\label{eq0}
\mbf i_{\F^{\diamond}}:\mr{Mod}([\F^{\diamond}])\;\iso \;H^{1}(\A_{\F^\diamond},\mr{Aut}^{\F^{\diamond}})
\end{equation}
is bijective. Moreover, $\mbf i_{\F^{\diamond}}([\F^{\diamond}])=[(\mr{id})]$.
\end{teo}

\begin{dem} 
Let us recall that $\msl{A}_{\F^\diamond}$ is the common cut-graph to all marked foliations $(\mc G,g)$ with $[\mc G,g]\in \msl{Mod}([\F^\diamond])$; in this proof we will denote it by $\msl{A}$ and by $\msl A^i$ its connected components. 
Let us show first the injectivity of $\mbf i_{\F^{\diamond}}$.
If 
$$\mbf i_{\F^{\diamond}}([\G,g])=[(\varphi_{D,s})]=[(\breve{\varphi}_{D,s})]=\mbf i_{\F^{\diamond}}([\breve{\G},\breve{g}])$$ then  for each 
$D\in\CE$ there exists $\xi_{D}\in\mr{Aut}^{\F^{\diamond}}_{D}$ such that
\[\xi_{D}^{-1}\circ\varphi_{D,s}\circ \xi_{D'}=\breve{\varphi}_{D,s}.\]
Writing $\varphi_{D,s}=\psi_{D}^{-1}\circ\psi_{D'}$ and $\breve{\varphi}_{D,s}=\breve{\psi}_{D}^{-1}\circ\breve{\psi}_{D'}$, with $s=D\cap D'$, we deduce that 
$$\breve{\psi}_{D}\circ \xi_{D}^{-1}\circ \psi_{D}^{-1}=\breve{\psi}_{D'}\circ \xi_{D'}^{-1}\circ \psi_{D'}^{-1}$$
 defines conjugations $\Phi_{i}:W_{i}\to\breve{W}_{i}$ between the foliations  
 $\G^\sharp$ and $\breve{\G}^\sharp$ restricted to some tubular neighborhoods $W_{i}$ and $\breve{W}_{i}$ of $\mc E_{\mc G}^i:=\bigcup_{D\in\Ve_{{\A^{i}}}}g(D)\subset\mc E_{\G}$ and $\bigcup_{D\in\Ve_{{\A^{i}}}}\breve{g}(D)\subset\mc E_{\breve{\G}}$ respectively.
By composing $\Phi_i$ with suitable automorphisms of $\mc G^\sharp$
isotopic to the identity along the leaves and whose supports are disjoint from the singular locus of $\mc G^\sharp$,
we can assume that $\Phi_i$ conjugates the attaching points of the adjacent dicritical components.

On the other hand, since the self-intersection of $g(D)$ and $\breve{g}(D)$ coincides for each $D\in\mr{Comp}(\mc E^d)$, there is a conjugation $\Phi_{D}$ between the foliations $\G^\sharp$ and $\breve{\G}^\sharp$ restricted to some tubular neighborhoods $W_{D}$ and $\breve{W}_{D}$ of $g(D)$ and $\breve{g}(D)$ whose restriction to $g(D)$ is $\breve{g}\circ g^{-1}$.

In order to glue the conjugations $\Phi_{i}$ and $\Phi_{D}$ we use the following trick:
\begin{itemize}
\item[]
\it 
For each $0<\varepsilon<1$, any germ of biholomorphism of $(\C^{2},0)$
preserving the fibration $(x,y)\mapsto x$ and the curve $\{y=0\}$ can be represented by a $\mc C^{1}$ diffeomorphism $F$ from $\mb D_1\times\mb D_1$ onto a neighborhood of $(0,0)$
satisfying the same properties with support in $\{|x|<\varepsilon\}$.
\end{itemize}
\rm
This implies that there is an automorphism $F_{D}$ on a neighborhood of $g(D)$ preserving $g(D)$ and $\mc G^\sharp$, which is equal to $\Phi_{D}^{-1}\circ \Phi_{i}$ in a neighborhood of the attaching point $g(D)\cap\mc E_{\G}^{i}$ with support 
a polydisk centered at this point.  
Shrinking the  neighborhood of definition of $\Phi_{D}\circ F_{D}$ 
we obtain a conjugation of pairs $(\G^\sharp,g(D))$ and $(\breve{\G}^\sharp, \breve{g}(D))$
which can be glued with $\Phi_{i}$.

The gluing of $\Phi_{i}$ and $\Phi_{j}$ at the nodal singularities is made by  using linearizing coordinates for  $(\G,g(s))$ and $(\breve{\G},\breve{g}(s))$ as in \cite[\S8.5]{MM3}. In this way we obtain a global conjugation $\Phi:M_{\G}\to M_{\breve{\G}}$ between the foliations $\G^\sharp$ and $\breve{\G}^\sharp$ which is holomorphic at the singular points.

By definition of $\mr{Aut}^{\F^{\diamond}}_{D}$, the restrictions of $\xi_{D}$ to the divisor are isotopic to the identity. Hence the restriction of $\Phi$ to $\mc E_{\mc G}^{\diamond}$ is isotopic to $\breve{g}\circ g^{-1}$. Therefore $[\G,g]=[\breve{\G},\breve{g}]$. 
    
To prove the surjectivity of $\mbf i_{\F^\diamond}$ we consider a cocycle $\mbf c=(\varphi_{D,s})$ in a given class  of $H^{1}(\A_{\F^\diamond},\mr{Aut}^{\F^\diamond})$. We define $\varphi_{D,s}=\mr{id}$ when $f(s)$ is a nodal singularity or an attaching point of a dicritical component.
By gluing open neighborhoods $U_{D}$ of $f(D)$ using the local biholomorphisms $\varphi_{D,s}$ we obtain a complex manifold $M_{\mbf c}$ endowed with a foliation $\F_{\mbf c}$, a  divisor $\mc E_{\mbf c}$ and a biholomorphism between $\mc E_{\mbf c}$ and $\mc E_{\F}$ sending the singular locus of $\F_{\mbf c}$ onto  $\Sigma_\F$. There is a composition of blow-ups $E':M'\to(\C^{2},0)$ and a biholomorphism $g:M_{\mbf c}\to M'$ sending $\mc E_{\mbf c}$ onto the exceptional divisor $E'^{-1}(0)$, see for instance \cite[p. 306]{M}. We obtain a foliation $\F'=(E'\circ g)(\F_{\mbf c})$ on $(\C^{2},0)$ and a biholomorphism $h:\mc E_{\F}\to\mc E_{\F'}$ satisfying  $h(\Sigma_\F)=\Sigma_{\F'}$.
We define $f':=h\circ f:\mc E\to\mc E_{\F'}$. By construction $\mbf i_{\F^{\diamond}}([\F',f'])=[\mbf c]$.
\end{dem}

\section{Symmetry group-graph}\label{Symmetry-tree-group} 
We keep the notations and the fixed data of the previous section.
In order to define the remaining group-graph associated to $\F$,   we moreover fix for each  
$D\in\CE$ a regular point  $o_{D}\in D$  and  a transverse section $\Delta_{D}$ to $f(D)$ passing through  $f(o_{D})$.

\begin{defin}
For $s\in\Ed_{\A_{\F^\diamond}}$ we say that $\phi\in\mr{Aut}_{s}^{\F^{\diamond}}$ \emph{fixes the leaves} of $\F^\sharp$ if for every neighborhood $V$ of $f(s)$ there is a neighborhood $V'$ of $f(s)$ such that $\phi(V')\subset V$
and for all $p\in V'$ the points $p$ and $\phi(p)$ belong to the same leaf of $\F^\sharp_{|V}$. We denote by $\mr{Fix}_{s}^{\F^{\diamond}}$ the (normal) subgroup of   
$\mr{Aut}_{s}^{\F^{\diamond}}$ of these automorphisms.
\end{defin}

\begin{obs}\label{exfix} 
It is easy to see that an example of element of $\mr{Fix}_{s}^{\F^{\diamond}}$ is provided by $\phi\in\mr{Aut}_{s}^{\F^{\diamond}}$ such that $\phi|_{f(D)}=\mr{id}_{f(D)}$ and $F_{p}\circ\phi=F_{p}$ for any local first integral $F_{p}$ at every point $p\in f(D\setminus \Sigma)$ in a neighborhood of $f(s)$. The diffeomorphisms of the flow of a vector field tangent to the foliation fulfill this property for small times and, by composition, all the diffeomorphisms of the flow are in $\mr{Fix}^{\F^{\diamond}}_s$.
\end{obs}

\begin{obs}\label{systinverses} For a fundamental system $(V_\alpha)$ of  open neighborhoods of $f(s)$ let us  denote by $\mc Q_{V_\alpha}^\F$ the leaf space of the restriction of the foliation $\F^\sharp$ to $V_\alpha\setminus \mc E_\F$. The inclusion relation on the leaves  induces an inverse system of continuous maps $\mc Q^{\F^\diamond}(s) :=(\mc Q_{V_\alpha}^\F \leftarrow \mc Q_{V_\beta}^\F)_{V_{\beta}\subset V_\alpha}$. 
Every $\psi\in\mr{Aut}_{s}^{\F^{\diamond}}$ defines an automorphism\footnote{{
The system $\mc Q^{\F^\diamond}(s)$ is an element of the categoy 
$\underleftarrow{\mr{Top}}$ of the  pro-objects associated to the category  of the topological spaces and continuous maps. The objects of  this category  are the inverse families  of topological spaces  and 
${\mr{Aut}}(\mc Q^{\F^\diamond}(s))$ is the  group of the invertible elements of 
${\underleftarrow{\mathrm{lim}}}_\beta
{\underrightarrow{\mathrm{lim}}}_\alpha\mathcal{C}^0(\mc Q_{V_\alpha}^\F,\,\mc Q_{V_\beta}^\F)$,  cf. \cite[\S2.8]{Douady} or \cite[\S3.1]{MM3}.}}
of this inverse system $\wt\psi\in\mr{Aut}(\mc Q^{\F^\diamond}(s))$ and the map $\zeta : \psi \mapsto \wt\psi$ is a group morphism. It turns out that $\wt\psi$ is the identity if and only if $\psi\in\mr{Fix}_{s}^{\F^{\diamond}}$, i.e. we have an exact sequence:
\begin{equation*}%\label{exactautfix}
1\to\mr{Fix}_{s}^{\F^{\diamond}}\longrightarrow\mr{Aut}_{s}^{\F^{\diamond}}\stackrel{\zeta}{\longrightarrow}\mr{Aut}(\mc Q^{\F^\diamond}(s))\,.
\end{equation*}
\end{obs}

\begin{defin}\label{SymEtValenceSinguliere}
For $s\in\msl{Ed}_{\A_{\F^\diamond}}$ and  $D\in\Ve_{\A_{\F^\diamond}}$ we consider the groups 
\[\mr{Sym}_{s}^{\F^{\diamond}}:=\mr{Aut}^{\F^{\diamond}}_{s}/\mr{Fix}_{s}^{\F^{\diamond}}\]
and 
 $$\mr{Sym}^{\F^{\diamond}}_{D}:=\left\{\begin{array}{lcc}
C(H_{D}) &\textrm{if}& \mr{val}_\Sigma(D)\geq 3,\\
& &\\
C(H_{D})/H_{D} & \textrm{if} & \mr{val}_{\Sigma}(D) \leq 2,
\end{array}\right.$$
where $H_{D}\subset\mr{Diff}({{\Delta}}_{D},f(o_{D}))$ is the holonomy group of $\F^\sharp$ along $f(D)$, $C(H_{D})$ is its \emph{centralizer}  inside $\mr{Diff}({{\Delta}}_{D},f(o_{D}))$ and $\mr{val}_\Sigma(D)$, called here \emph{singular valency of} $D$, is the number of elements of $D\cap \Sigma$.
\end{defin}

In order to define maps $\rho_{D}^{s}:\mr{Sym}^{\F^{\diamond}}_{D}\to\mr{Sym}^{\F^{\diamond}}_{s}$, $s\in D$ we will need the following result:

\begin{lema}\label{17}
If $\psi\in\mr{Aut}^{\F^{\diamond}}_{D}$ satisfies $\psi_{|f(D)}=\mr{id}_{f(D)}$ and $\psi_{|{{\Delta}}_{D}}=\mr{id}_{{{\Delta}}_{D}}$ then the germ of $\psi$ at $f(s)$ belongs to $\mr{Fix}_{s}^{\F^{\diamond}}$.
\end{lema}

\begin{proof}
For each $p\in f(D)\setminus\mr{Sing}(\F^\sharp)$ we choose a local holomorphic first integral $F_{p}$ of $\F$ defined in a neighborhood of $p$. 
The set 
\[\Omega:=\{p\in f(D)\setminus\mr{Sing}(\F^\sharp)\,|\,F_{p}\circ\psi=F_{p}\}\]
is open and closed in $f(D)\setminus\mr{Sing}(\F^\sharp)$ and it contains $f(o_{D})={{\Delta}}_{D}\cap f(D)$. Hence $\Omega=f(D)\setminus\mr{Sing}(\F^\sharp)$ and we conclude thanks to Remark~\ref{exfix}.
\end{proof}

Using good fibrations (cf. Definition~\ref{goodfib}), each element $\phi$ of $C(H_D)$ can be  extended to an element of $\mr{Aut}^{\F^\diamond}_{D}$. Thanks to Lemma~\ref{17},  the class modulo $\mr{Fix}_{s}^{\F^\diamond}$ of the germ at $f(s)$ of this extension does not depend on the way of extending and hence on the choice of the good fibrations. We define $\rho^s_D(\phi)$ as this class in case $\mr{val}_\Sigma(D)\ge 3$.\\

Before defining $\rho^{s}_{D}$ for $D$ with $\mr{val}_\Sigma(D)\leq 2$, we must make some preliminary considerations. 
Let us fix for each  point $s\in \Sigma\cap D$ a conformal compact disc $\mc K_s\subset \mc E$  such that $ s\in \inte{\mc K}_s$, $o_D\in \partial \mc K_s$ and  $\Sigma\cap \mc K_s=\{s\}$. The simple loops $\gamma_s$ that parametrize $\partial \mc K_s$ with the conformal orientation, form
 a generator system of the fundamental group
$\pi_1(D \setminus \Sigma, o_D)$. Hence the holonomies $h_{D,s}\in \mr{Diff}(\Delta_D, f(o_D))$ of the foliation $\F^\sharp$ along $f\circ\gamma_s$, $s\in \Sigma\cap D$,  generate $H_D$.
\begin{defin}\label{localHolonomies}
We call $h_{D,s}$ the  
\emph{local holonomies} given by the \emph{appropriate compact discs system} $(\mc K_s)_{s\in D\cap\Sigma}$.
\end{defin}
Let us denote $K_s:=f(\mc K_s)$.
If we fix  a good fibration, any element $\phi$ of the \emph{centralizer} $C(h_{D,s})$ of $h_{D,s}$ in $\mr{Diff}(\Delta_{D},f(o_{D}))$ has an unique extension to a neighborhood of $K_s$, by a homeomorphism $\phi^{\mathrm{ext}}$ that leaves invariant the foliation $\F^\sharp$ and each fiber of the fibration; moreover $\phi^{\mathrm{ext}}$ is necessarily  holomorphic at $f(s)$. Taking its germ at $f(s)$ we obtain a map
\begin{equation}\label{extension1}
 \mr{ext} : C(h_{D,s})\longrightarrow \mr{Aut}_s^{\F^{\diamond}} \,,\quad \phi\mapsto \phi^{\mr{ext}}
\end{equation}
and the property of uniqueness of the extensions imply that this map is a group morphism. 

Let us consider  the inverse system $\mc Q^{\F^\diamond}({\mc K_s}) =(\mc Q^\F_{W_\alpha} \leftarrow \mc Q^\F_{W_\beta})_{W_\beta\subset W_\alpha}$, where $(W_\alpha)_\alpha$ is the fundamental system of neighborhoods of $K_s$ and $\mc Q^\F_{W_\alpha}$ is the leaf space of the restriction of the foliation to $W_\alpha\setminus \mc E_\F$. Let  ${V}_s\subset \inte K_s$ be a small open disc centered at $f(s)$. Over $K_s\setminus  V_s$ the foliation is a collar; thus we have an isomorphism of  inverse systems (i.e.  an isomorphism of the category $\underleftarrow{\mr{Top}}$)
\begin{equation}\label{iso1}
\mc Q^{\F^\diamond}(s) \iso \mc Q^{\F^\diamond}({\mc K_s})\,.
\end{equation}

We also consider  
the orbit spaces $ \mc Q^{h_{D,s}}_\alpha$ of the pseudogroup defined by the restriction of  $h_{D,s}$ to $\Delta_D \cap W_\alpha$; they form an  inverse system $\mc Q^{h_{D,s}} =  (\mc Q^{h_{D,s}}_\alpha \leftarrow \mc Q^{h_{D,s}}_\beta )_{W_\beta\subset W_{\alpha}}$. We can choose  each $W_\alpha$  such that there are  retractions along the leaves  from $W_\alpha\setminus \mc E_\F$ on $ (W_\alpha\setminus \mc E_\F)\cap \pi_D^{-1}(\partial K_s)$, $\pi_{D}$ being the good fibration; moreover we can require that $W_\alpha\cap \pi_D^{-1}(\partial K_\alpha)$ is a set of suspension type, cf. \cite[Definition~3.1.1]{MM1}.  This property implies that the leaf space of the restriction of the foliation to this set can be  identified to the orbit space $ \mc Q^{h_{D,s}}_\alpha$ of the  restriction of  $h_{D,s}$ to $\Delta_D \cap W_\alpha$. Hence using~(\ref{iso1}),  the retractions induce  isomorphisms
\[
\tau : \mc Q^{\F^\diamond}({s})\iso \mc Q^{h_{D,s}}
\hbox{ and }
\tau_{\ast}  : \mr{Aut}(\mc Q^{\F^\diamond}({s}))\iso \mr{Aut}(\mc Q^{h_{D,s}}) \,, \;\tau_\ast(\varphi):= 
\tau\circ\varphi\circ\tau^{-1},
\]
the inverse of $\tau$ being given by the inclusion relations of the orbits of $h_{D,s}$ in the leaves of the foliation on neighborhoods of $K_s$. \\

Each element $\phi$ of $C(h_{D,s})$ induces an automorphism $\xi(\phi):=\tau_{\ast}(\zeta(\phi^\mr{ext}))$  of $\mc Q^{h_{D,s}}$ and the map $\xi : C(h_{D,s})\rightarrow \mr{Aut}(\mc Q^{h_{D,s}})$ is a group morphism.
\begin{lema}\label{noyaufix}
The kernel of the morphism $\xi $ is the cyclic group generated by~$h_{D,s}$.
\end{lema}
\begin{proof}
Let us take  $\phi\in \ker(\xi)$. This means that for each open neighborhood $U$ of $f(o_D)$ in $\Delta_D$ there is an open set  $V\supset U$ such that for each $z\in U$, $z$ and $\phi(z)$ are in the same $V$-orbit  of $h_{D,s}$. The $V$-orbit of $z$ being the set of points $z'$ of $V$ such that either there exists $n\in \N$ fulfilling either  $\phi(z),\ldots, \phi^n(z)\in V$ and $\phi^n(z)=z'$, or $\phi^{-1}(z),\ldots, \phi^{-n}(z)\in V$ and $\phi^{-n}(z)=z'$. If $|h_{D,s}'(f(o_{D}))|\neq 1$ then $h_{D,s}(z)=\lambda z$ and $\phi(z)=\mu z$. Hence 
$\phi(z)=\mu z=\lambda^{\nu(z)}z=h_{D,s}^{\nu(z)}(z)$ implies $\nu(z)$ constant.

Otherwise, let $V_{n}$ be the set of points $z\in U$ such that $h_{D,s}^{k}(z)\in V$ for each $k=0,\ldots,n$.
There is an  uncountable set $K$ invariant by $h_{D,s}$ such that for all $n\in \Z$
it is  contained in the connected component of $V_{n}$ containing $f(o_{D})$.
If $h_{D,s}$ is linearizable (conjugated to a rotation) then we can take an invariant conformal disc as $K$.
If $h_{D,s}$ is resonant non linearizable then $K$ is a union of petals contained in $U$.
If $h_{D,s}$ is non resonant non linearizable we take as $K$ the hedgehog associated to $U$, cf. \cite{Perez-Marco-Bourbaki}.

For each $z\in K$ there is  an integer $\nu(z)$ such that $\phi(z)=h_{D,s}^{\nu(z)}(z)$. Thus, there is $n\in\Z$ such that $\phi$ and $h_{D,s}^{n}$ coincide on an uncountable subset of $K$ and by isolated zeros principle 
they coincide on the connected component of $V_{n}$ containing $f(o_{D})$. Then the germs of $\phi$ and $h_{D,s}^{n}$ at $f(o_D)$ are equal, that achieves the proof.
\end{proof}

\begin{cor}\label{isosym}
The extension map (\ref{extension1}) induces an isomorphism: 
$$[\mr{ext}] : C(h_{D,s})/\langle h_{D,s}\rangle\stackrel{\sim}{\to}\mr{Sym}_{s}^{\F^\diamond}\,.$$
\end{cor}

\begin{proof}
By construction the following diagram is commutative
$$\xymatrix{1 \ar[r] & \mr{Fix}_{s}^{\F^\diamond}\ar[r] & \mr{Aut}^{\F^\diamond}_{s}
\ar[r]^\zeta & \mr{Aut}(\mc Q^{\F^\diamond}(s))\ar[d]^\tau\\ 
1\ar[r] &\langle h_{D,s}\rangle\ar[r]& C(h_{D,s})\ar[r]^\xi\ar[u]^{\mr{ext}} & \mr{Aut}(\mc Q^{h_{D,s}})}$$
Thanks to Remark (\ref{systinverses}) and Lemma (\ref{noyaufix}) the lines are  exact. Because $\tau$ is an isomorphism, $\mr{ext}$ induces an isomorphism between $C(h_{D,s})/\langle h_{D,s}\rangle$ and $\mr{Aut}^{\F^\diamond}_{s}/\mr{Fix}_{s}^{\F^\diamond}=\mr{Sym_s^{\F^\diamond}}$. 
\end{proof}

\begin{obs}\label{exemplesymlin}
Suppose that  there are local coordinates $u_1$, $u_2$ at $s$ for which $\F^\sharp$ is defined by a linear differential 1-form $\omega=\mu u_2du_1-u_1du_2$ and $D=\{u_2=0\}$, $D'=\{u_1=0\}$ are the components of $\mc E_\F$. On the transversals $\{u_i=1\}$ the local holonomies are $h_{D_1,s}(u_2)=e^{2\pi i\mu}u_2$ and $h_{D_2,s}(u_1)=e^{2\pi i\frac{1}{\mu}}u_1$ and their centralizers are formed by the linear automorphisms  in the coordinates $u_i$,  $C(h_{D_i,s})=\C^\ast u_i$. Therefore we have isomorphisms
\[
\frac{C(h_{D_1,s})}{\langle h_{D_1,s}\rangle}\simeq 
\frac{\C}{2\pi i(\Z+\mu \Z)}
 \stackrel{\tau_1}{\longrightarrow}\mr{Sym}_s^{\F^\diamond}
\stackrel{\tau_2}{\longleftarrow}
\frac{\C}{2\pi i(\Z+\frac{1}{\mu} \Z)}
\simeq
\frac{C(h_{D_2,s})}{\langle h_{D_2,s}\rangle}
\]
To describe $\tau_2^{-1}\circ \tau_1$ let us remark that  the automorphisms $(u_1,u_2)\mapsto (e^t u_1, e^{\mu t}u_2)$, $t\in \C$ are elements of $\mr{Fix}_{s}^{\F^\diamond}$. Thus the automorphisms $(u_1,u_2)\mapsto (e^t u_1, u_2)$ and $ (u_1,u_2)\mapsto  (u_1, e^{\mu t}u_2)$ in $\mr{Aut}_s^{\F^\diamond}$ that extend $h_{D_1,s}$ and $h_{D_2,s}$ respectively,  define the same element of $\mr{Sym}_s^{\F^\diamond}$. It follows: 
\[
\tau_2^{-1}\circ \tau_1: 
{\C}/{2\pi i(\Z+\mu \Z)}
\longrightarrow 
{\C}/{2\pi i(\Z+\frac{1}{\mu} \Z)}\,, 
\quad
\dot t \mapsto -\frac{1}{\mu}
\dot t\,.
\]
\end{obs}

\begin{obs}
If  $\mr{val}_{\Sigma}(D)\ge 3$,
then $[\mr{ext}]^{-1}\circ \rho^s_D$ is the quotient map 
\[\mr{Sym}^{\F^\diamond}_D=\;
C(H_D) \hookrightarrow C(h_{D,s})\rightarrow C(h_{D,s})/\langle h_{D,s}\rangle\,.
\]
\end{obs}
For $D$ containing at most two singular points of $\F^\sharp$ we  define 
\[\rho_{D}^{s}:=[\mr{ext}]\colon\mr{Sym}_{D}^{{\F^\diamond}}=C(h_{D,s})/\langle h_{D,s}\rangle\to\mr{Sym}_{s}^{{\F^\diamond}}.\]

\begin{defin}
We call \emph{symmetry group-graph} and we denote by $\mathrm{Sym}^{\F^\diamond}$  the group-graph consisting of  the groups $\mathrm{Sym}^{\F^\diamond}_D$, $\mathrm{Sym}^{\F^\diamond}_s$, with  $D\in \msl{Ve}_{\msl{A}_{\F^\diamond}}$, $s\in \msl{Ed}_{\msl{A}_{\F^\diamond}}$ and the morphisms $\rho_{D}^s$, $s\in D$.
\end{defin}

Now, we are going to define a group-graph morphism $\alpha:\mr{Aut}^{{\F^\diamond}}\to\mr{Sym}^{{\F^\diamond}}$ which will induce an isomorphism on the 1-cohomology. 
If $s\in\msl{Ed}_{{\msl A}_{\F^\diamond}}$, we define $\alpha_s$ as  the quotient map $\mr{Aut}^{{\F^\diamond}}_{s}\to\mr{Sym}^{{\F^\diamond}}_{s}$. 
If $D\in{\msl{Ve}}_{\msl{A}_{\F^\diamond}}$, we define $\alpha_{D}:\mr{Aut}^{{\F^\diamond}}_{D}\to\mr{Sym}^{{\F^\diamond}}_{D}$ as follows.
Fix $\Phi\in\mr{Aut}^{{\F^\diamond}}_{D}$ and take an homotopy $\phi_{t} : f(D\setminus \Sigma)
\rightarrow  f(D\setminus \Sigma)$, $t\in[0,1]$,
between
$\phi_{0} :=\Phi_{|f(D\setminus \Sigma)}$ 
and $\phi_{1}:=\mr{id}_{ f(D\setminus \Sigma)}$.
Consider the path 
$\beta(t)=\phi_{t}(o_{D})$ and the holonomy map $h_{\beta}:(\Phi(\Delta_{D}),\Phi(o_{D}))\to(\Delta_{D},o_{D})$ associated to it. It is easy to see that $h_{\beta}\circ\Phi|_{\Delta_{D}}$ belongs to $C(H_{D})$.
If $D$ has singular valency $\mr{val}_\Sigma(D)\ge 3$,  the group consisting of the homeomorphisms of $f(D\setminus\Sigma)$ which are homotopic to the identity is simply connected \cite{pi1homeo}; consequently $h_{\beta}$ does not depend on the chosen isotopy $\phi_{t}$ and we can put $\alpha_D(\Phi):= h_{\beta}\circ\Phi$.
Finally, if $v(D)\le 2$ then only the class $[h_{\beta}\circ \Phi_{|\Delta_{D}}]$ of $h_{\beta}\circ\Phi_{|\Delta_{D}}$ modulo $H_{D}$ is well-defined and we put $\alpha_D(\Phi):=[h_{\beta}\circ \Phi_{|\Delta_{D}}]$.

\begin{lema}
$\alpha:\mr{Aut}^{\F^\diamond}\to\mr{Sym}^{\F^\diamond}$ is a group-graph morphism.
\end{lema}

\begin{proof}
We must see that the following diagram is commutative:
$$\xymatrix{
\mr{Aut}^{{\F^\diamond}}_{D} \ar@{->}[r]^{\alpha_{D}}\ar@{->}[d]_{\breve{\rho}^{s}_{D}} & \mr{Sym}^{{\F^\diamond}}_{D}\ar[d]^{{\rho}_{D}^{s}} 
\\ \mr{Aut}^{{\F^\diamond}}_{s}\ar@{->}[r]^{\alpha_{s}} & \mr{Sym}^{{\F^\diamond}}_{s}
}$$
where  $\breve{\rho}^s_D$, resp. $\rho^s_D$, denote the ``restriction morphisms'' of the group-graphs $\mr{Aut}^{\F^\diamond}$, resp. $\mr{Sym}^{\F^\diamond}$. Let us consider two cases, depending on the singular valency $\mr{val}_\Sigma(D)$.  First let us assume $\mr{val}_\Sigma(D)\ge 3$ and let us  fix $\phi\in\mr{Aut}^{{\F^\diamond}}_{D}$.
Then $\alpha_{D}(\phi)=h_{\beta}\circ \phi_{|\Delta_{D}}$ with $h_{\beta}:\phi(\Delta_{D})\to\Delta_{D}$ the holonomy along a path~$\beta$. There exists $\phi'\in\mr{Aut}^{{\F^\diamond}}_{D}$ with compact support outside $f(D)\cap\mr{Sing}(\F^\sharp)$ whose restriction to $\Delta_{D}$ coincides with $h_{\beta}$. Indeed $\phi'$ can be constructed by composition of flows of tangent vector fields whose supports intersect the divisor $D$ in conformal disks disjoint from the singularities and which cover the image of~$\beta$. Thus $\alpha_{D}(\phi)=\phi'\circ\phi_{|\Delta_{D}}$ and ${\rho}_{D}^{s}(\alpha_{D}(\phi))$ coincides with the class modulo $\mr{Fix}^{{\F^\diamond}}_{s}$ of the germ of $\phi'\circ\phi$ at $s$ thanks to Lemma~\ref{17}. This germ is just the germ of $\phi$ at $s$ because the support of $\phi'$ does not intersect the singularities. This achieves the proof  in the case $\mr{val}_\Sigma(D)\ge 3$. 
If $\mr{val}_\Sigma(D)\le 2$, the only difference is that only the class of $h_{\beta}\circ\phi_{|\Delta_{D}}$ modulo $H_{D}=\langle h_{D,s}\rangle$ is well-defined; but we can proceed analogously choosing arbitrarily~$h_{\beta}$.
\end{proof}

\begin{prop}[Extension]\label{propext} 
Let $W$ be a neighborhood  of $f(s)$, $s\in {\E}_{\F^\diamond}$, then each germ $\phi\in\mr{Fix}^{{\F^\diamond}}_{s}$ can be extended to a germ $\Phi\in\mr{Aut}^{{\F^\diamond}}_{D}$ along $f(D)$, whose support fulfills  $\mr{supp}(\Phi)\cap f(D)\subset W$.
\end{prop}
\begin{proof} 
At the point $f(s)$ let us fix local holomorphic coordinates  $(u, v)$, $u(f(s))=v(f(s))=0$ such that the axes  are invariant by the foliation, and $v=0$ is a local equation of $f(D)$. We denote by  $Z= u\DD u + vB(u,v)\DD v$ the holomorphic vector field  tangent to the foliation.

First, we will see that each germ  of biholomorphism $\zeta : (f(D),f(s))\to (f(D),f(s))$ can be extended as  an element $g$ of $\mr{Aut}^{{\F^\diamond}}_{D}$ whose germ at $f(s)$  belongs to  $\mr{Fix}^{{\F^\diamond}}_{s}$ and whose support intersect  $f(D)$ inside $W$. This is easy to prove  when $\zeta$ is embedded in the flow $(\psi_t)_t$ of a vector field $a(u)u \DD{u}$, i.e.  $\zeta=\psi_1$. 
Indeed in this case, let us consider the real vector field $Y$ whose flow is the flow of $aZ$, but with real times. Let us take a real smooth function $\rho$ equal to $1$ on an open  neighborhood of $f(s)$, such that  $\mathrm{supp}(\rho)\cap f(D)$ is contained in  $W$ and in a definition domain of $Y$. Then $\rho Y$ extends by zero along $f(D)$ and the elements  $\Psi_t$ of its flow induce  homeomorphisms defined on  neighborhoods of $f(D)$.
Their   supports are contained in the support of $\rho$ and their germs at $f(s)$ are element of $\mr{Fix}^{{\F^\diamond}}_{s}$, cf. Example (\ref{exfix}). 
Clearly the  restriction  of $\Psi_1$ to $f(D)$ is equal to $\zeta$ near $f(s)$. Now, when $\zeta$ is not embedded in a flow, we decompose $\zeta=\zeta_1\circ \zeta_2$, with $\vert \zeta_1'(0)\vert$, $\vert \zeta_2'(0)\vert \neq 1$. Both $\zeta_1$ and $\zeta_2$ are linearizable. Thus they can both be embedded in a flow and have convenient extensions. Their composition extends $\phi$ along $f(D)$, and fulfils  the required properties.

Now, up to composition  we can suppose that the restriction of the germ $\phi$ to $f(D)$ is the identity. Let us choose $\varepsilon>0$ such that the compact disc $\overline{\mb D_{2\varepsilon}}\subset f(D)$ defined by $\vert u\vert \leq 2\varepsilon$, is contained in $W$ and in a definition domain of $\phi$. Denote by $C$ the compact annulus contained in $\overline{\mb{D}_{2\varepsilon}}$ given by $\varepsilon\leq \vert u \vert \leq 2 \varepsilon$. 
By implicit function theorem, there is a holomorphic function $\tau$ defined in an open neighborhood $\Omega$ of  $C$, that verifies:
$$(u\circ \phi)(m)=u\circ \Phi^Z_{\tau(m)}(m) \quad \hbox{and}\quad  \tau_{\vert f(D)}=0\,,$$ 
$\Phi^Z_t$ being the flow of the  previous vector field $Z$. 
Let us take a $\mc C^\infty$ function $\alpha : f(D)\to \mb R$ with  compact support  in $\Omega\cap f(D)$, that is equal to $1$ on a neighborhood of $C$. 
The map 
$$
\xi : m\mapsto \xi(m):=\Phi^Z_{\alpha(u(m))\tau(m)}(m)
$$ 
is a $\mathcal{C}^\infty$ diffeomorphism, because  its restriction to $f(D)$ is the identity and moreover it is a local diffeomorphism. Indeed  using coordinates $(u, z)$ at each point of $f(D)$, with $z$ a local first integral of the foliation,  we easily see that the jacobian matrix of $\xi$ is the identity. 
Clearly $\chi :=\phi\circ\xi^{-1}$  coincides with $\phi$ near $f(s)$, it preserves  the foliation and it  leaves invariant each fiber of $u$:
 $$
 \Delta_c:=\{u=c\},\quad \varepsilon\leq \vert  c\vert\leq 2\varepsilon\,.
 $$
 Thus the restriction $\chi_{\vert \Delta_c}$ of $\chi$ to $\Delta_c$ leaves invariant the orbits of the holonomy map of $\F^\sharp$ around $f(s)$ represented on $\Delta_c$ -which is equal to the restriction of $\Phi^Z_{2\pi}$ to $\Delta_c$. By Lemma \ref{noyaufix},  $\chi_{\vert \Delta_c}$ is an iterated of this holonomy map. We deduce of it the existence of  an integer  $p\in \mb Z$ such that near $C$,  $\chi$ coincides with $\Phi^Z_{2i\pi  p}$. 

Let us take now a function $\sigma : [0, 2\varepsilon]\to \mb R$ vanishing on $[0,\varepsilon]$ and being  equal to $1$ on $[\frac{3}{2}\varepsilon,2 \varepsilon[$.
The homeomorphism $\Theta : m\mapsto \Phi^Z_{2i\pi p\sigma(\vert u(m)\vert)}(m)$ is the identity near $f(s)$,  it coincides with $\Psi$ for $\frac{3}{2}\varepsilon\leq \vert u\vert\leq 2\varepsilon$ and it  leaves $\F$ invariant. To end the proof, we define the required  diffeomorpism $\Phi$ as the germ along $f(D)$ of the  diffeomorphism  equal to $\Theta^{-1}\circ\Psi$ when $\vert u\vert \leq 2\varepsilon$ and equal to  the identity otherwise. 
\end{proof}

\begin{teo}\label{isoH1symH1aut}
The morphism of group-graphs $\alpha : \mr{Aut}^{{\F^\diamond}}\to\mr{Sym}^{{\F^\diamond}}$  induces a natural bijection 
\begin{equation}\label{eq1}
\alpha_{1}:H^{1}({\A}_{\F^\diamond},\mr{Aut}^{{\F^\diamond}})\stackrel{\sim}{\to}H^{1}({\A}_{\F^\diamond},\mr{Sym}^{{\F^\diamond}}).
\end{equation}
\end{teo}

\begin{proof}
The surjectivity of $\alpha_{1}$ follows easily from the surjectivity of $\alpha_{s}:\mr{Aut}^{{\F^\diamond}}_{s}\to\mr{Sym}_{s}^{{\F^\diamond}}$. For this we fix an orientation $\prec$ of the tree and for $(c_{D,s})\in Z^{1}({\A}_{\F^\diamond},\mr{Sym}^{{\F^\diamond}})$, we choose for each edge $s$ with $\partial s=\{D,D'\}$, $D\prec D'$, an element $\varphi_{D,s}$ such that $\alpha_s(\varphi_{D,s})=c_{D,s}$  and we set $\varphi_{D',s}:=\varphi_{D,s}^{-1}$. Clearly the family $(\varphi_{D,s})$ is an element of $Z^{1}({\A}_{\F^\diamond},\mr{Aut}^{{\F^\diamond}})$ defining a lift of $(c_{D,s})$.

To prove the injectivity of $\alpha_{1}$ we consider
 $[\phi_{D,s}],[\wt\phi_{D,s}]\in H^{1}({\A}_{\F^\diamond},\mr{Aut}^{{\F^\diamond}})$ such that $\alpha_{1}([\phi_{D,s}])=[\alpha_{s}(\phi_{D,s})]=[\alpha_{s}(\wt\phi_{D,s})]=\alpha_{1}([\wt\phi_{D,s}])$. 
Then there is $(g_{D})\in C^{0}({\A}_{\F^\diamond},\mr{Sym}^{{\F^\diamond}})$ such that
 $$\alpha_{s}(\wt\phi_{D,s})=\rho_{D}^{s}(g_{D})^{-1}\circ\alpha_{s}(\phi_{D,s})\circ\rho_{D'}^{s}(g_{D'})\in \mr{Sym}^{{\F^\diamond}}_{s}$$
where $s=D\cap D'$. Let $\varphi_{D}\in\mr{Aut}^{{\F^\diamond}}_{D}$ be extensions of $g_{D}\in\mr{Sym}^{{\F^\diamond}}_{D}$ and let us denote by $(\varphi_{D})_{s}\in\mr{Aut}^{{\F^\diamond}}_{s}$ its germ at $s$. Then 
$$\alpha_{s}(\wt\phi_{D,s})=\alpha_{s}((\varphi_{D}^{-1})_{s})\circ\alpha_{s}(\phi_{D,s})\circ\alpha_{s}((\varphi_{D'})_{s})$$
and there is $F_{s}\in\mr{Fix}_{s}^{{\F^\diamond}}$ such that
$\wt\phi_{D,s}=(\varphi_{D}^{-1})_{s}\circ\phi_{D,s}\circ(\varphi_{D'})_{s}\circ F_{s}$. Now we choose a map $\delta:{\msl{Ed}}_{\msl{A}_{\F^\diamond}}\to{\msl{Ve}}_{\msl{A}_{\F^\diamond}}$ such that $s\in\delta(s)$ for each $s\in{\msl{Ed}}_{\msl{A}_{\F^\diamond}}$ and we define $\bar F_{D}$ as the composition over the set $\{s\in{\E}_{\F^\diamond}\,|\,\delta(s)=D\}$ of extensions of $F_{s}$ to a neighborhood of $\delta(s)$ with disjoint supports given by  Proposition~\ref{propext}. Finally putting $\bar\varphi_{D}=\varphi_{D}\circ \bar F_{D}\in\mr{Aut}^{{\F^\diamond}}_{D}$ we have that
$$\wt\phi_{D,s}=\bar\varphi_{D}^{-1}\circ\phi_{D,s}\circ\bar\varphi_{D'}^{-1}\in\mr{Aut}^{{\F^\diamond}}_{s},$$
i.e. $[\phi_{D,s}]=[\wt\phi_{D,s}]$ in $H^{1}({\A}_{\F^\diamond},\mr{Aut}^{{\F^\diamond}})$.
\end{proof}

\begin{dem2}{of Theorem~\ref{thmB}} It follows immediately from Theorem~\ref{p1} and Theorem~\ref{isoH1symH1aut}.
\end{dem2}

\section{Foliations of finite type}\label{sectionfoliationsTF}
In this section we introduce the optimal condition on a singular germ of foliation $\F$ in order to have a finite dimensional moduli space $\mr{Mod}([\F^\diamond])$.
We keep all the notations of previous sections.\\

Given a marked foliation $\F^\diamond=(\F,f)$ and a sheaf $Q$  defined on a neighborhood of $\mc E_\F$ in the ambient space $M_\F$ of $\F^\sharp$, we can associate a group-graph, denoted $Q^{\F^\diamond}$,  over the cut-graph $\A_{\F^\diamond}$  as follows: if $s\in\msl{Ed}_{\msl A_{\F^\diamond}}$ then $Q_{s}^{\F^\diamond}$ is the stalk of $Q$ at $f(s)$ and  if $D\in\msl{Ve}_{\msl A_{\F^\diamond}}$, then
\[Q_{D}^{\F^\diamond}:=H^{0}(f(D),\iota_{f(D)}^{-1}Q),\] 
$\iota_{f(D)}$ being the inclusion map of $f(D)$ in $M_\F$ and for $s\in D$ the morphism $\rho_{D}^{s}:Q_{D}^{\F^\diamond}\to Q_{s}^{\F^\diamond}$ being the canonical restrictions.
\begin{defin}\label{transverseSymetries}
We call \emph{group-graph of transverse infinitesimal symmetries} of $\F$ the group-graph $\mc T^{\F^\diamond}$ associated to the sheaf $\mc T^{\F^\sharp}:=\mc B^{\F^\sharp}/\mc X^{\F^\sharp}$ on $M_\F$ equal to the quotient of the sheaf $\mc B^{\F^\sharp}$ of $\F^\sharp$-basic\footnote{i.e. whose flow leaves $\F^\sharp$ invariant.}  
holomorphic vector fields tangent to the $\F^\sharp$-invariant components of $\mc E_\F$, 
by the sheaf $\mc X^{\F^\sharp}$ of holomorphic vector fields tangent to $\F^\sharp$.
\end{defin}

\begin{obs}\label{vfsigma}
For each $(D,s)\in I_{\mc E^\diamond}$ let us consider the local holonomies $h_{D,s}$ as in Definition~\ref{localHolonomies}. There are linear isomorphisms, depending on the choice of an appropriate disc system,
\[
\mc T^{\F^\diamond}_D \iso \mc T_{H_D}\,,\qquad \mc T_s^{\F^\diamond}\iso \mc T_{h_{D,s}}\,,
\]
where $\mc T_{H_D}$ (resp. $\mc T_{h_{D,s}}$) is the vector space of 
 germs at $f(o_{D})$ of vector fields on the transversal disc  $\Delta_{D}$  which are invariant by the holonomy group $H_{D}$ of $\F^\sharp$ along $f(D)$ (resp. invariant by $h_{D,s}$), see \cite{Loray, MS}.
Moreover, if $\mc X_{\Delta_{D}}^{0}$ denotes the set of germs of  vector fields  on $(\Delta_{D}, f(o_D))$ vanishing at $f(o_{D})$, then
 the exponential map $\exp:\mc X_{\Delta_{D}}^{0}\to\mr{Diff}(\Delta_{D},f(o_{D}))$ sends 
 $\mc T_{h_{D,s}}$ into $C(h_{D,s})$ and $\mc T_{H_D}$ into $C(H_{D})$.
\end{obs}

Now we define \emph{a coloring on $\A_{\F^\diamond}$}  by saying: 
\begin{enumerate}
\item\label{Vvert}\it $D\in \msl{Ve}_{\msl A_{\F^\diamond}}$ is \emph{green} if the holonomy group $H_{D}$ is finite,
\item\label{Evert}\it $s\in \msl{Ed}_{\msl A_{\F^\diamond}}$ is   \emph{green} if for each $D\in \partial s$  the holonomy map $h_{D,s}$ is periodic.
\item \it $D\in \msl{Ve}_{\msl A_{\F^\diamond}}$ or $s\in \msl{Ed}_{\msl A_{\F^\diamond}}$ are \emph{red} otherwise.
\end{enumerate}
Let us denote by $\mc J^{\F^\diamond}$ the \emph{group-graph of holomorphic first integrals} associated to the sheaf of germs of holomorphic first integrals of $\F^\sharp$. Because $\F^\sharp$ does not have saddle-node singularities, an element $a\in\msl{Ve}_{\msl A_{\F^\diamond}}\cup\msl{Ed}_{\msl A_{\F^\diamond}}$  is green iff $\mc J^{\F^\diamond}_{a}\neq\C$, see \cite{MatMou}.
Notice that  if an edge $s=D\cap D'$ of $\A_{\F^\diamond}$ is red then the vertices $D$ and $D'$ are also red. Hence the set of red elements of $\A_{\F^\diamond}$ is a subgraph called \emph{red graph of $\F^\diamond$} and denoted by $\AR_{\F^\diamond}$. 
\begin{prop}\label{equivrepuls}
Let  $s$ and  $D\in \partial s$ be a green edge and a green vertex of $\A_{\F^\diamond}$. Then
the following properties are equivalent:
\begin{enumerate}
 \item\label{holloc} the holonomy group $H_D$ is generated by $h_{D,s}$;
\item\label{surjintprem} $\mc J^{\F^\diamond}_D\to\mc J_s^{\F\diamond}$ is surjective;
\item\label{surjtau} $\mc T_{D}^{\F^\diamond}\to\mc T_{s}^{\F^\diamond}$
 is surjective;
\item\label{surjSym} $ \mr{Sym}^{\F^\diamond}_{D}\to\mr{Sym}_{s}^{\F^\diamond}$
 is surjective.
\end{enumerate}
\end{prop}
\begin{proof}
Let  $D$ be a green vertex of $\A_{\F^\diamond}$ and let $z:(\Delta_D,f(o_D))\to \C$ be a linearizing coordinate of the holonomy group $H_{D}\subset\mr{Diff}(\C,0)$  which is finite. For each singular point  $s$ of $D$ (necessarily a green edge of $\A_{\F^\diamond}$) there is $n_{D,s}\in\N$ such that the local holonomies $h_{D,s}$ given in Definition~\ref{localHolonomies} are $h_{D,s}(z)=\zeta_{D,s}\,z$ for some primitive $n_{D,s}$-root of unity.
Let us denote by $n_{D}\in\N$ the least common multiple of $\{n_{D,s},\ s\in D\cap\Sigma\}$ and by $\zeta_D$ a primitive $n_D$-rooth of unity. 
Because a first integral is completely determined by its restriction to the transversal $\Delta _D$, we can consider $\mc J_D^{\F^\diamond}$ as subrings of  $\C\{z\}$. In the same way, by extending the elements of $\mc J_s^{\F^\diamond}$ along the compact discs used in Definition \ref{localHolonomies} to define $h_{D,s}$, we can also  consider $\mc J_s^{\F^\diamond}$ as a subring of $\C\{z\}$.  With these identifications and using Remark \ref{vfsigma}, we have the following well known equalities and isomorphisms:
\[
\mc J^{\F^\diamond}_{D}=\C\{z^{n_{D}}\}\,,\ 
\mc T_{D}^{\F^\diamond}\simeq\mc T_{H_D}=\mc J_D^{\F^\diamond} z\partial_{z}\,,\ 
 C(H_{D})=\{z(\alpha+\mf N_D)\,|\, \alpha\in\C^{*}\}\,,
\]
\[
\mc J^{\F^\diamond}_{s}=\C\{z^{n_{D,s}}\}\,,\ 
\mc T_{s}^{\F^\diamond}\simeq\mc T_{h_{D,s}}= \mc J_s^{\F^\diamond} z\partial_{z}\,,\ 
 C(h_{D,s})=\{z(\alpha+\mf N_{s})\,|\, \alpha\in\C^{*}\}\,,
\]
with $\mf N_D\subset \mc J_D^{\F^\diamond}$ and $\mf N_{s}\subset\mc J_{s}^{\F^\diamond}$ being the maximal ideals.  Furthermore  $H_D$ is cyclic, generated by $\zeta_D\, z$. The required equivalences follow immediately.
\end{proof}

If $B$ is a nonempty connected subgraph of  a connected component $\A_{\F^\diamond}^i$ of $\A_{\F^\diamond}$, then for every vertex 
$D\notin B$ of $\A_{\F^\diamond}^i$ or $D\in \partial B$  there is a unique geodesic $[D,B]\subset\A_{\F^\diamond}^i$ joining $D$ and $\partial B$. We define the following \emph{pre-order relation} on the set of 
vertices of the closure of $\A_{\F^\diamond}^{i}\setminus B$ by means of
\begin{equation*}%\label{preordre}
D'\le D \Longleftrightarrow D'\in[D,B].
\end{equation*}

\begin{defin} We say that   $B\subset A^i_{\F^\diamond}$   is \emph{repulsive} if for each  edge $s=D\cap D'$ of $(\A^{i}_{\F^\diamond}\setminus B)$ with $D'\le D$, 
$\mr{Sym}^{\F^\diamond}_{D}\to\mr{Sym}_{s}^{\F^\diamond}$
 is surjective.
\end{defin}

\noindent A change of marking induces an isomorphism of colored graphs compatible with the repulsiveness property that gives sense to the following definition:

\begin{defin}\label{FT}
The foliation $\F$ is of \emph{finite type} if for each connected component $\A_{\F^\diamond}^i$ of $\A_{\F^\diamond}$ we have: either the subgraph $\A_{\F^\diamond}^{i}\cap \AR_{\F^\diamond}$ is nonempty, connected and repulsive, or it is empty and there exists a green repulsive vertex in $\A^i_{\F^\diamond}$.
\end{defin}

\noindent This finiteness property does not depend on the marking. In fact thanks to Theorem~ \ref{newliftingconjugacies} of Appendix, it depends only on the topological class of the germ $\F$ at $0\in\C^2$ and it is fulfilled by all the foliations $\mc G$ with $[\mc G^\diamond]\in \mr{Mod}([\F^\diamond])$ as soon as it holds for one of them.

\begin{teo} \label{isoassymarsym} If $\F$ is of finite type then  the extension by identity  map
$Z^{1}(\AR_{\F^\diamond},\mr{Sym}^{{\F^\diamond}})\to Z^{1}(\A_{\F^\diamond},\mr{Sym}^{{\F^\diamond}})$ induces a bijection :
\begin{equation}\label{eq2}
H^{1}(\A_{\F^\diamond},\mr{Sym}^{{\F^\diamond}})\simeq H^{1}(\AR_{\F^\diamond},\mr{Sym}^{{\F^\diamond}})\,.
\end{equation}
\end{teo}
\begin{proof}
It is a direct consequence of  Pruning Theorem~\ref{elagage} and Proposition~\ref{equivrepuls}.
\end{proof}

When the foliation $\F$ has  finite type let us now give the precise value of the integer $\tau_{\F}$ in the statement of Theorem~\ref{thm} in the introduction. 
Let us consider  the following subgraph $\AR^0_{\F^\diamond}$ whose
 \begin{itemize}
 \item  vertices correspond by $f$  to the invariant irreducible components of $\mc E_{\F}$ whose holonomy groups do not leave invariant  a non-trivial vector field,
 \item  edges correspond by $f$ to the 
 resonant non-normalizable or non-resonant non-linearizable
 singularities of $\F^\sharp$.
 \end{itemize} 
 As before changes of marking induce isomorphisms between the graphs $\AR^0_{\F^\diamond}$; that allows us to put:
 
\begin{defin}\label{defindetau} If $\F$ is of finite type we call \emph{codimension of $\F$} the integer
 \[\tau_{\F}:=\mr{rank}\, H_{1}({\AR}_{\F^\diamond}/\AR^0_{\F^\diamond},\,\Z),\]
 where ${\AR}_{\F^\diamond}/\AR^0_{\F^\diamond}$ is the  quotient graph obtained from ${\AR}_{\F^\diamond}$ by collapsing $\AR^0_{\F^\diamond}$ to a single vertex. 
 \end{defin}

We will highlight now a group structure on $H^{1}(\AR_{\F^\diamond},\mr{Sym}^{{\F^\diamond}})$ when $\F$ has finite type. Let us choose an arbitrary  map $s\mapsto D_s$ from $\Ed_{\AR_{\F^\diamond}}$ to $\Ve_{\AR_{\F^\diamond}}$ with $D_s\in \partial s$. 
Since $H^{1}(\AR_{\F^\diamond},\mr{Sym}^{{\F^\diamond}})$ is the quotient of
\[
Z^{1}(\AR_{\F^\diamond},\mr{Sym}^{{\F^\diamond}})\simeq
\bigoplus_{s\in\Ed_{\AR_{\F^\diamond}}}\mr{Sym}_{s}^{{\F^\diamond}}\simeq \bigoplus_{s\in\Ed_{\AR_{\F^\diamond}}}\frac{C(h_{D_s,s})}{\langle h_{D_s,s}\rangle}
\]
by $C^{0}(\AR_{\F^\diamond},\mr{Sym}^{{\F^\diamond}})=\bigoplus\limits_{D\in\Ve_{\AR_{\F^\diamond}}}\mr{Sym}_{D}^{{\F^\diamond}}$, we must pay attention to the centralizers $C(h)$ of the local holonomy transformations $h=h_{D,s}\in \mr{Diff}(\Delta_{D},f(o_{D}))$. 

Trivially $h$  is of one and only one following type:
\begin{itemize}
\item[$(P)^{\hphantom{0}}$]  periodic;
\item[$(L^{1})$] linearizable and non-periodic;
\item[$(L^{0})$] formally linearizable but non-linearizable;
\item[$(R^{1})$] resonant non-linearizable but normalizable;
\item[$(R^{0})$] resonant non-linearizable and non-normalizable.
\end{itemize}
Classically, in the first three cases there exists a (only formal, in the case~$(L^{0})$) local coordinate $u$ on $(\Delta_{D}, f(o_D))$, such that $u\circ h=\alpha u$, with $\alpha\in\C^{*}$.  In these situations  $h=\exp X$ , with $X:= \log(\alpha) u  \partial_u $.
In the resonant cases~$(R^{0})$ and~$(R^{1})$ there exists a  coordinate $u$ on $\Delta_{D}$, only formal in the case~$(R^{0})$,  such that $h=\ell^{r}\circ\exp t_0 X$, $X:=\frac{u^{p+1}}{1+\lambda u^{p}}\partial_u$, where  $p+1$ is the contact order of $h^k$ with the identity when $h'(0)^k=1$,   $\ell$ is the formal diffeomorphism defined by $u\circ \ell:=e^{2i\pi/p}u$, $h'(0)=e^{2i\pi r/p}$, $t_0\in \C^\ast$ and we can choose $t_0=1$, (remark that $\ell$ and $\exp X$ commute). In all cases $u$ is  unique up  to multiplication by an element of $\C^\ast$.
Let us denote by $\wh C(h)$ the centralizer of $h$ inside the group $\wh{\mr{Diff}}(\Delta_{D},f(o_{D}))$ of formal diffeomorphisms of $(\Delta, f(o_D))$. Clearly  $C(h)=\wh C(h)\cap\mr{Diff}(\Delta_{D}, f(o_{D}))$. As in Remark~\ref{vfsigma}
let us denote by $\mc T_h$  the space of germs of holomorphic vector fields on $(\Delta,f(o_D))$  invariant by $h$. 
The following result contains several well-known facts.
\begin{prop}\label{cent}
According to the type of $h\in\mr{Diff}(\Delta_{D},f(o_{D}))$ we have:
\begin{itemize}
\item [${(P)}$ ] $C(h)=\{g\in \mr{Diff}(\Delta_{D},f(o_{D})) \;|\; u\circ g=u(\alpha+F(u^{q})),\ \alpha\in\C^{*},\ F\in u \C\{u\}\}$; $C(h)/\langle h\rangle\simeq \mr{Diff}(\C,0)$;
$\mc T_h= \C\{u^q\}u\partial_{u}$; 
\item[${(L)}$ ] $\wh{C}(h)=\{g\in \wh{\mr{Diff}}(\Delta_{D},f(o_{D}))\;|\; u\circ g=\lambda g,\ \lambda\in\C^{*}\}$ $=\exp \C X\simeq \C^\ast$; 
\item[${(R)}$ ] $\wh{C}(h)=\{\ell^{n}\circ\exp tX,\ n\in\Z/p\Z,\ t\in\C\}=\langle \ell \rangle \oplus \exp \C X\simeq \Z/p\Z\oplus \C$.
\end{itemize}
Moreover:
\begin{itemize}
\item[${(L^1)}$ ]  $\mc T_h =\C X$ and $C(h)/\exp \mc T_h=\{1\}$;
\item[${(R^{1})}$ ] $\mc T_h =\C X$, $C(h)/\exp \mc T_h\simeq \Z/p\Z$ and 
$C(h)/\langle h,\exp \mc T_h\rangle\simeq \Z/(p,r)$;
\item[${(L^{0})}$ ] $\mc T_h=\{0\}$;  $C(h)=\{g\in \mr{Diff}(\Delta_{D},f(o_{D}))\;|\; u\circ g = \lambda u,\lambda\in \bf{D}\}\simeq \mbf D$, where
 $\mbf D:=\{\lambda\in \C^\ast\;|\; u^{-1}\circ(\lambda u) \hbox{ is convergent}\}$ 
 is a totally discontinuous subgroup of $\mb U(1)$, that can be uncountable  \cite{Perez-Marco};
\item[${(R^{0})}$ ] $\mc T_h=\{0\}$, there exists $m\in \N^\ast$ such that the  sequence 
\[
0\to \Z\stackrel{\alpha}{\longrightarrow}  C(h)\stackrel{\beta}{\longrightarrow} \Z/p\Z\,, \quad \alpha(t)=\exp\frac{t}{m} X\,,\quad \beta(g)=\frac{1}{2i\pi}\log g'(0)
\]
is exact and $C(h)/\langle h\rangle$  is finite.
\end{itemize}
\end{prop}
\begin{proof} The periodic case has been already described in the proof of Proposition~\ref{equivrepuls} unless the isomorphism $C(h)/\langle h\rangle\simeq\mr{Diff}(\C,0)$ which follows easily from the fact that every $g\in\mr{Diff}(\C,0)$ commuting with a rotation $z\mapsto e^{2i\pi/q}z$ writes as $(g^\flat(z^q))^{\frac{1}{q}}$ for a unique $g^\flat\in\mr{Diff}(\C,0)$. The description of the formal centralizers is given for instance in \cite[Proposition~1.3.2]{Loray} and \cite[p. 150]{CCD}, where it is also shown that $C(h)=\wh C(h)$ in the normalizable cases~$(L^{1})$ and~$(R^{1})$. 
In addition,
in the case $(L^{1})$, $\mc T_h=\C u\DD u$ and $\exp:\mc T_h\to C(h)$ can be canonically identified to the surjective morphism $\C\to\C^{*}$ given by $\mu\mapsto e^{\mu}$.
In the case $(R^{1})$, $\mc T_h$ is equal to $\C X$ and  
\[C(h)/\langle h,\exp \C X\rangle\simeq(\Z/p\Z\oplus \C)/\langle\dot{r}\oplus 1,0\oplus\C\rangle\simeq\Z/\langle p,r\rangle.\] 
Thanks to the description of $\wh{C}(h)$, in the case $(R^0)$  the kernel of $\beta$ consists  of the convergent  elements of the flow of $X$ and by \'Ecalle-Liverpool Theorem \cite[Corollary 2.8.2]{Loray} it is equal to $\alpha(\Z)$ for a suitable $m\in \N^\ast$.
Finally in the case $(L^{0})$, let $g$ be  an element of $C(h)=\wh{C}(h)\cap\mr{Diff}(\C,0)$. Then $u\circ g=\lambda u$ with $|\lambda|=1$. Indeed if $|\lambda|\neq 1$,  $g$ would be linearizable in a convergent coordinate and $h$, that commutes with $g$, would be also linearizable in the same coordinate, contradicting the assumption $(R^0)$. On the other hand, $\bf D$ is a totally discontinuous subgroup of $\mb U(1)$ because otherwise $\mbf{D}=\mb U(1)$ and $h$ would be linearizable.
\end{proof}
It follows from this proposition and from Remark \ref{obsab},

\begin{teo}\label{strutgaaut}
Given $a\in  \msl{Ve}_{\msl{A}_{\F^\diamond}} \cup \msl{Ed}_{\msl{A}_{\F^\diamond}}$ we have the equivalences:
\[
a \hbox{ is red} \Longleftrightarrow \mr{Sym}^{\F^\diamond}_a \hbox{ is abelian }\Longleftrightarrow \mr{dim}_\C\,
 \mc T_a^{\F^\diamond}<\infty\]
\[a\in  \msl{Ve}_{\AR^0_{\F^\diamond}}\cup\msl{Ed}_{\AR^0_{\F^\diamond}}\Longleftrightarrow \mc T^{\F^\diamond}_a =  0\,.
\]
Furthermore the abelian structure of the group-graph $\mr{Sym}^{{\F^\diamond}}$ over $\AR_{\F^\diamond}$  induces an abelian group structure on  $H^1(\AR_{\F^\diamond},\mr{Sym}^{{\F^\diamond}})$.
\end{teo}

\begin{obs}\label{point-base}
By following the natural bijections~(\ref{eq0}), (\ref{eq1}) and~(\ref{eq2}) provided by Theorems~\ref{p1}, \ref{isoH1symH1aut} and~\ref{isoassymarsym} respectively,  one can check that if ${\mc H}([\F^{\diamond}])={\mc H}([\G^{\diamond}])$ then
 $\mr{Mod}([\F^{\diamond}])$ and $\mr{Mod}([\G^{\diamond}])$ coincide as sets, but their respective abelian group structures are related by the map
$\mu\mapsto \gamma \mu$ where $\gamma=[\G^{\diamond}]\in\mr{Mod}([\F^{\diamond}])$.
\end{obs}

We end this section by proving some properties of centralizers which will be useful  in the sequel. 
\begin{lema}\label{egalcentrs}
If $g$ and $h$ are  non-periodic and  $g\in C(h)$, then $C(g)=C(h)$. 
\end{lema}
\begin{proof} The group $C(h)$ in case $(L)$ only depends on the formal coordinate $u$ that linearize $h$. Similarly in case $(R)$ all the non-periodic elements of $C(h)$ have the same invariants $p$, $\lambda$ and if we fix these invariants, the centralizers of resonant diffeomorphisms only depend on the normalizing coordinate $u$.
Thus the lemma follows from the fact 
that in  both cases all the non-periodic elements of a centralizer  can be linearized or normalized by using the same coordinate.
\end{proof}

\begin{lema}\label{centralinfini}
Let $H$ be a finitely generated subgroup of $\mr{Diff}(\C,0)$. Suppose that $H$ and its centralizer are infinite. Then $H$ is abelian, it contains a non-periodic element $h$ and  $ H\subset C(H)=C(h)$.
\end{lema}

\begin{proof} If $H$ is abelian, we can take as $h$ a non-periodic element: these can not all be periodic, otherwise $H$ would be finite. If $H$  is not abelian,  any non trivial commutator is non-periodic.
On the other hand, the description of centralizers  given by Proposition \ref{cent} allows us to see that  each infinite subgroup of $C(h)$ contains an element of infinite order. Then there is a non-periodic element $g$ in $C(H)\subset C(h)$. Because $C(g)$ is abelian and  clearly contains $H$, we have $C(g)\subset C(H)\subset C(h)$. By applying  Lemma \ref{egalcentrs} to $g\in C(h)$, we obtain  the equality $C(H)=C(h)$. The inclusion $H\subset C(H)$ follows from the  abelianity of $H$.
\end{proof}

It follows immediately:
\begin{prop}\label{memetypes}
Under the  hypothesis of the previous lemma, all the non-periodic elements of $H$ are of the same type $(L^1)$, $(L^0)$, $(R^1)$ or $(R^0)$.
\end{prop}

\section{Non-degenerate foliations}\label{ModulusNonDeg}
Before proving Theorem~\ref{B0} stated in Introduction, we are going to look at the genericity of non-degenerate foliations. In \cite{MS}  one finds results of this type but in a more restricted framework. Theorem~6.2.1 of this paper claims  the Krull open density, in the set of 1-forms defining foliations of  second kind (in particular generalized curves),  of the set of 1-forms defining foliations fulfilling conditions (\ref{holonomienonabel}) and (\ref{nonpreiodsingchaines}) of Definition \ref{NDfoliationsDef}. We can resume the proof of this theorem by using Theorem~A of \cite{MRR} instead of \cite[Lemma~6.2.7]{MS}. We obtain in this way:
\begin{teo}\label{genericiteNonDeg} If $\eta$ is germ at $0\in\C^2$ of  a holomorphic differential 1-form  with isolated singularity  defining a generalized curve foliation $\mc G$,  then for all $p\in \N$ there exists a 1-form $\eta'$ with same $p$-jet as $\eta$ and an integer $n\ge p$ such that any foliation $\F$ defined by a differential form $\omega$ with same $n$-jet at $0$ as $\eta'$,  fulfills 
\begin{enumerate}
\item the holonomy group $\mr{Im}(\mc H_D^{\F^\sharp})$ of any component $D$ of $\mc E_{\F}$ with $v:=\mr{card}(D\cap \Sigma_\F)\geq 3$,  is a free product of $v-1$ cyclic subgroups of $\mr{Diff}(\C,0)$; in particular it is non solvable hence topologically rigid;
\item for any singular chain $D_{0},\ldots,D_{\ell}$ in $\mc E_\F$, the local holonomies of $\F^\sharp$ at the singular points $s_{i}=D_{i-1}\cap D_{i}$, $i=1,\ldots,\ell$, are non-periodic; \end{enumerate}
and in particular   $\F$  is non-degenerate.
\end{teo}

\noindent It is well known  that for $p$ large enough all the foliations $\F$  have the same reduction map, the same singular points on the exceptional divisor, the same Camacho-Sad indices  and the same dicritical components as $\mc G$.\\

Let us call
\emph{$\F^\diamond$-singular chain of $\mc E$} any
sequence $D_0,\ldots,D_{\ell_{\mc C}}$ of irreducible components of $\mc E$ defining a
 connected subgraph 
\begin{equation}\label{shemadechaine}
(\mc C)\phantom{AA}
\xymatrix{
\bullet_{D_{0}}\ar@{-}[r]^{s_{1}}
& \bullet_{D_{1}}\cdots
& 
\cdots\ar@{-}[r]^{s_{\ell_{\mc C}}}& \bullet_{D_{\ell_{\mc C}}}
}
\phantom{AA\mc C}
\end{equation}
of $\A_{\F^{\diamond}}$
such that the singular valency (cf. Definition \ref{SymEtValenceSinguliere})  of its  \emph{extremities}  $D_0$ and $D_{\ell_{\mc C}}$ is at least three, and  that of the others, called \emph{interior vertices}, being exactly two. 
If $\ell_{\mc C}=1$, then $\mc C$ consists only of two adjacent divisors of singular valency at least three: $\xymatrix{
\bullet_{D_{0}}\ar@{-}[r]^{s_{1}}& 
\bullet_{D_{1}}}$.
The image by the marking map $f$ of a $\F^\diamond$-singular chain of $\mc E$ is a singular chain of $\mc E_\F$ as considered in Introduction.

\begin{prop}\label{NDimplyTF} Let $\F^\diamond$ be a non-degenerate marked foliation. Then the union $\br{\msl{R}}_{\F^\diamond}$ of all $\F^\diamond$-singular chains is a connected repulsive subgraph of $\msl{A}_{\F^\diamond}$ contained in $\msl{R}_{\F^\diamond}$;  hence $\F$ is of finite type and 
\[H^1(\msl{R}_{\F^\diamond}, \mr{Sym}^{\F^\diamond})\simeq H^1(\br{\msl{R}}_{\F^\diamond}, \mr{Sym}^{\F^\diamond})\,.\]
\end{prop}
In order to simplify the notations in the two proofs below, we will write 
$\msl R$, $\br{\msl R}$, $\mr{Sym}$, instead of  
$\msl R_{\F^\diamond}$, $\br{\msl R}_{\F^\diamond}$, 
$\mr{Sym}^{\F^\diamond}$.  

\begin{dem2}{of Proposition \ref{NDimplyTF}}
Clearly $\br{\msl{R}}$ is connected and the closure of $\msl{R}\setminus\br{\msl R}$ in $\msl R$  is exactly the union of all connected subgraphs $\mc C$ denoted as in  (\ref{shemadechaine}) but with singular valencies satisfying $\mr{val}_\Sigma (D_0)\geq 3$, $\mr{val}_\Sigma(D_j) = 2$ for $0 < j <\ell_{\mc C}$ and $\mr{val}_{\Sigma}(D_{\ell_{\mc C}})=1$ or $2$.
By definition  of the group-graph $\mr{Sym}$, for $j\geq 1$ the  morphisms $\rho_{D_{j}}^{s_{j}}: \mr{Sym}_{D_{j}}\to \mr{Sym}_{s_{j}}$  are bijective and $\br{\msl{R}}$ is repulsive in $\AR$. Because $\F$ is non degenerate all the vertices and edges of $\br{\msl{R}}$ are red; thus $\AR$ is also repulsive and connected. By using Pruning Theorem~\ref{elagage} we obtain the group isomorphism $
%H^1(\A, \mr{Sym}) \simeq 
H^1(\AR, \mr{Sym})\simeq H^1(\br\AR, \mr{Sym})$.  
\end{dem2}

\begin{dem2}{of Theorem~\ref{B0}}
To have uniqueness
of  the numbering   in the notation (\ref{shemadechaine}) of a singular chain $\mc C$,  we fix in the sequel an extremity $\br D$ of $\br\AR$ and we prescript  that $D_0$ belongs to the geodesic joining $D_{\ell_{\mc C}}$ to $\br D$.  We will say that $D_0$, resp.  $s_{\ell_{\mc C}}$, is the \emph{ initial vertex}, resp. \emph{terminal edge} of $\mc C$.

For an interior vertex $D_j$ of $\mc C$ the  morphisms $\rho_{D_{j}}^{s_{j}}
$ and $\rho_{D_{j}}^{s_{j+1}}$  are bijective and by composition they induce isomorphisms 
\[\xi_{D_j}: \mr{Sym}_{s_{1}}\simeq \mr{Sym}_{D_{j}}\,, \quad
\quad \xi_{s_j}: \mr{Sym}_{s_{1}}\simeq \mr{Sym}_{s_j}\,,\quad
0<j<\ell_{\mc C}\,.\] 
Let us consider the subgroups:
\begin{itemize}
\item $\wt{Z}^{1}(\br\AR, \mr{Sym})\subset Z^{1}(\br\AR, \mr{Sym})$ of the 1-cocycles $(\phi_{D,a})_{(D,a)\in I_{\br\AR}}$ such that $\phi_{D,a}=1$ if $a$ is not the terminal edge of some singular chain,
\item $\wt{C}^{0}(\br\AR, \mr{Sym})\subset C^{0}(\br\AR, \mr{Sym})$ of the 0-cochains $(\phi_{D})_{D\in \Ve_{\br{\AR}}}$ such that 
$\phi_{D}=
\xi_D\circ \rho^{s_1}_{D_0}(\phi_{D_{0}})$
 for all the  interior vertices $D$ of any  singular chain, $D_0$ denoting its initial vertex.
 \end{itemize}
Notice that the coboundary morphism $\partial^0$ defined in Remark \ref{obsab} maps $\wt{C}^{0}(\br\AR, \mr{Sym})$ in $\wt{Z}^{1}(\br\AR, \mr{Sym})$, allowing us to define the group
\[
\wt{H}^{1}(\br\AR, \mr{Sym}):=\mr{coker}(\partial^0 : \wt{Z}^{0}(\br\AR, \mr{Sym}) \longrightarrow \wt{Z}^{1}(\AR, \mr{Sym}))\,.
\]
We  easily see that each element of $H^{1}(\br\AR, \mr{Sym})$ can be represented by a cocycle belonging to  $\wt Z^{1}(\br\AR, \mr{Sym})$. We deduce that the morphism 
\[
\tau : \wt{H}^{1}(\br\AR, \mr{Sym}) \longrightarrow {H}^{1}(\br\AR, \mr{Sym})
\]
is surjective.
On the other hand if a cocycle $\mbf c^0:=(\phi_D)_{D\in \Ve_{\br\AR}}\in {Z}^{1}(\br\AR, \mr{Sym})$ satisfies   $\partial^0 (\mbf c^0)\in \wt{Z}^{1}(\br\AR, \mr{Sym})$, then for each singular chain $\mc C$ of lenght $\ell_{\mc C}\geq 2$ denoted as in (\ref{shemadechaine}) we have the equalities 
\[
\rho_{D_j}^{s_j}(\phi_{D_j})=\rho_{D_{j-1}}^{s_j}(\phi_{D_{j-1}})\,,\quad  j<\ell_{\mc C}\,.
\]
It follows that  $\mbf c^0\in \wt{C}^{0}(\br\AR, \mr{Sym})$. Therefore $\ker(\tau)$ is trivial and  $\tau$ is an isomorphism. 
To achieve the proof of  Theorem~\ref{B0},  let us first  notice that  the group $\wt{C}^{0}(\br\AR, \mr{Sym})$ is finite.
Indeed it is isomorphic to the product   of  all centralizers of  holonomy groups  associated to the   vertices of $\AR$ having singular valency at least three. These holonomy groups  being non-abelian by non-degenerate assumption,   thanks to Lemma \ref{centralinfini} their centralizers are finite.  
On the other hand,  Proposition \ref{cent} gives a decomposition of   $\wt{Z}^{1}(\br\AR, \mr{Sym})$ as  $F\oplus B \oplus_{j=1}^{\lambda}(\C^{*}/\alpha_{j}^{\Z})\oplus(\C^{*})^{\nu}$; that completes the proof of Theorem~\ref{B0}.
\end{dem2}

\section{Examples}\label{SectionExamples}

Before proving Theorem~\ref{thm} in full generality let us motivate its statement by computing the moduli space of some non-trivial examples using the identification $\mr{Mod}([\F^\diamond])\simeq H^1(\AR_{\F^\diamond},\mr{Sym}^{\F^\diamond})$.\\
 
 \indent $\bullet$ \emph{Example 0: a logarithmic generic multicusp}\\ 
 Let $\mc L$ be the logarithmic germ foliation at $0\in \C^2$ defined by the meromorphic form
 \[
 \omega:=\sum_{i=1}^p \alpha_i \frac{d(y^2+a_ix^3)}{y^2+a_ix^3} +
 \delta \frac{d(y-x)}{y-x}+
 \sum_{i=1}^p \beta_i \frac{d(x^2-b_iy^3)}{x^2-a_iy^3}\,,
 \]
with $a_i, b_i\in \C$ mutually distincts. 
We normalize the coefficients $\alpha_i, \beta_i,\delta\in\C^\ast$ by requiring $\delta + 2\sum_{i=1}^p(\alpha_i +\beta_i)=1$. To simplify the exposition we suppose that  $\mc E$ is equal to the exceptional divisor $\mc E_{\mc L}$, the marking being the identity map and $\mc L^\diamond = (\mc L, \mr{id}_{\mc E_{\mc L}})$. Clearly $\mc E$ is formed by five irreducible components, its dual graph is equal to $\A_{\mc L^\diamond}$
\[(\A_{\mc L^\diamond})\phantom{AAA}
\xymatrix{
\bullet_{C'}\ar@{-}[r]^{s'_{\infty}}
& \bullet_{D'}\ar@{-}[r]^{s'_{0}}
& \bullet_{D}\ar@{-}[r]^{s''_{0}}
& \bullet_{D''}\ar@{-}[r]^{s''_{\infty}}
& \bullet_{C''}
}
\phantom{(\A)\phantom{AAA}}
\]
\[
\sigma^\diamond(C')=\sigma^\diamond(C')=1\,, \quad
\sigma^\diamond(D')=\sigma^\diamond(D'')=p+2\,, \quad
\sigma^\diamond(D)=3\,,
\]
and it decomposes into one singular chain that is the red part $\AR_{\mc L^\diamond}$ of the graph 
\[
(\AR_{\mc L^\diamond})\phantom{AAA}
\xymatrix{
\bullet_{D'}\ar@{-}[r]^{s'_{0}}& 
\bullet_{D}\ar@{-}[r]^{s''_{0}}& 
\bullet_{D''}
}
\phantom{(\AR)\phantom{AAA}}
\]
and two dead branches $
\xymatrix{
\bullet_{C'}\ar@{-}[r]^{s'_{\infty}}
& 
\bullet_{D'}
}
$ 
and $
\xymatrix{
\bullet_{D''}\ar@{-}[r]^{s''_{\infty}}& 
\bullet_{C''}
}
$ necessarily  green. Thus the restriction morphisms  $\mr{Sym}^{\mc L^\diamond}_{C'}\to \mr{Sym}^{\mc L^\diamond}_{s'_\infty}$, $\mr{Sym}^{\mc L^\diamond}_{C''}\to \mr{Sym}^{\mc L^\diamond}_{s''_\infty}$ are surjective and by Pruning Theorem~\ref{elagage}, the group  $H^1(\A_{\mc L^\diamond}, \mr{Sym}^{\mc L^\diamond})$ is isomorphic to $H^1(\AR_{\mc L^\diamond}, \mr{Sym}^{\mc L^\diamond})$. On $\AR=\AR_{\mc L^\diamond}$ the morphism $\partial^0$ defined in Remark~\ref{obsab} decomposes, with additive notations on abelian groups, as 
\[
{\xymatrix{C^{0}(\AR,\mr{Sym}^{\mc L^\diamond})=\ar[d]_{\partial^0} \!\!\!\!\!\!&
 \!\!C(H_{D'})\ar[rd]_{\xi_1} \!\!&
 \!\!\oplus& C(H_{D})\ar[ld]^{\xi_2}\ar[rd]_{\xi_3} \!\!&
 \!\!\oplus& C(H_{D''})\ar[ld]_{\xi_4} \\
Z^{1}(\AR,\mr{Sym}^{\mc L^\diamond})=\!\!\!\! \!\!&
\!\! \!\!&
\!\!\mr{Sym}^{\mc L^\diamond}_{s_{0}'}
\!\!&
\!\!\oplus\!\!&
\!\! \mr{Sym}^{\mc L^\diamond}_{s_{0}''} &}}
\]
\[
\partial^0(c_1\oplus c_2\oplus c_3)=(\xi_1(c_1)+\xi_2(c_2))\oplus(\xi_3(c_2)+\xi_4(c_3))\,.
\]
On the other hand, $\mr{Sing}(\mc L^\sharp)\cap D'$, resp. $\mr{Sing}(\mc L^\sharp)\cap D''$,  is formed by 
$s'_\infty$, $s'_{0}$, resp. $s''_\infty$, $s''_{0}$, and 
the attaching points $s'_i$, resp.  $s''_i$, of the strict transforms of the curve $\{y^2+a_ix^3=0\}$, resp.  $\{x^2+b_iy^3=0\}$, $i=1,\ldots, p$; and $\mr{Sing}(\mc L^\sharp)
\cap D$ is formed by $s'_0$, $s''_0$ and the attaching  point $s_1$ of the strict transform of $\{y-x=0\}$. The Camacho-Sad indices of $\mc L^\sharp$ at these points are
\[
\mr{CS}(D',s'_{\infty})=CS(D'', s''_\infty)=1/2\,, \quad 
\mr{CS}(D, s_1)= -\delta\,,\quad
\mr{CS}(D',s'_0)=-2\wt\alpha\,,
\]
\[
\mr{CS}(D'',s''_0)=-2\wt\beta\,,
\quad
\mr{CS}(D',s'_j)=-\frac{\alpha_j}{2}\wt\alpha
\quad
\mr{CS}(D'',s''_j)=-\frac{\beta_j}{2}\wt\beta\,,
\]
with $\wt\alpha:=(1+\sum_{i=1}^p\alpha_i)$, $\wt\beta:=(1+\sum_{i=1}^p\beta_i)$. Assuming that $p\geq 3$, we choose  $\alpha_i$, $\beta_i$ and $\delta$ sufficiently generic so that no Camacho-Sad index is a real number,  except at the points $s'_\infty$ and $s''_\infty$.  All the singularities of $\mc L^\sharp$ are linearizable and  according to Proposition \ref{cent} the centralizers $C(H_{D'})$, $C(H_{D})$, $C(H_{D''})$ are isomorphic to $\C^\ast=\C/2\pi i\Z$. Using Remark \ref{exemplesymlin} we obtain:
\[
\xymatrix{
%\C/2i\pi\mb Z
\frac{\C}{2i\pi\mb Z}\ar@{>>}[rd]_{\xi_1} 
&\oplus 
& 
%\C/2i\pi\mb Z
\frac{\C}{2i\pi\mb Z}\ar@{->>}[rd]_{\xi_3}\ar@{->>}[ld]^{\xi_2} 
&\oplus
& 
%\C/2i\pi\mb Z
\frac{\C}{2i\pi\mb Z}\ar@{->>}[ld]^{\xi_4}\\ 
&
\C/2i\pi(\mb Z+2\wt\alpha\mb Z) 
&\oplus
& \C/2i\pi(\mb Z+2\wt\beta\mb Z)
&
}\]
moreover $\xi_2$ and $\xi_3$ are induced by the identity map, but $\xi_1$ is induced by $z\mapsto \frac{1}{2\wt\alpha}z$ and $\xi_4$ by $z\mapsto \frac{1}{2\wt\beta}z$.  It immediately follows that $\partial^0$ is surjective and $H^1(\AR, \mr{Sym}^{\mc L^\diamond})=0$.  We conclude that $H^1(\A,\mr{Aut}^{\mc L^\diamond})=0$ but we can not deduce from this the topological SL-rigidity of $[\mc L^\diamond]$ because $\mc L$ does not satisfy condition~(TR), see \cite[Th\'eor\'eme~3.5]{Paul}.\\
%\begin{itemize}\it
%\item[\bf -] $\mc L^\diamond$ is \emph{topologically $\mr{SL}$-rigid} i.e. its moduli space is trivial,
%\[\mr{Mod}([\mc L^\diamond])=\{[\mc L^\diamond]\}.\]
%\end{itemize}

 \indent $\bullet$ \emph{Example 1: non-degenerate  multicusps}\\ 
Let us perform a generic perturbation $\F_1$ of the previous example, provided by Theorem~\ref{genericiteNonDeg},  that does not change  the Camacho-Sad indices but such that the holonomy groups along $D'$ and $D''$ are  non abelian. In this case it is well-known that $\F_1$ satisfy condition~(TR) and therefore we can compute $\mr{Mod}([\F_1^\diamond])$ by identifying it with $H^1(\AR,\mr{Sym}^{\F_1^\diamond})$.
According to Lemma~\ref{centralinfini} their centralizers are finite groups $F_1'$, $F_1$ and $F_1''$ respectively, however  $\mr{Sym}^{{\F_1^\diamond}}_{s'_0}$ and $\mr{Sym}^{{\F_1^\diamond}}_{s''_0}$ remain isomorphic to $\mr{Sym}^{\mc L^\diamond}_{s'_0}$ and $\mr{Sym}^{\mc L^\diamond}_{s''_0}$ because the singularities of $\F_1^\sharp$ at $s'_0$ and $s''_0$ are again linearizable.
\[
\xymatrix{
F_1'\ar@{>>}[rd]_{\xi_1} 
&\oplus 
&F_1\ar@{->>}[rd]_{\xi_3}\ar@{->>}[ld]^{\xi_2} 
&\oplus
& F_1''\ar@{->>}[ld]^{\xi_4}\\ 
&
\C/2i\pi(\mb Z+2\wt\alpha\mb Z) 
&\oplus
& \C/2i\pi(\mb Z+2\wt\beta\mb Z)
&
}
\]
It follows:
\begin{itemize}
\item[\bf -]\it $\mr{Mod}([\mc F_1^\diamond])$  is a finite quotient of a product of two
 elliptic curves.
\end{itemize}

\indent $\bullet$ \emph{Example 2:  partially degenerate  multicusps}\\ 
With the induction technique used in the proof of Theorem~\ref{genericiteNonDeg}, we can perform a perturbation of $\mc L$  that provides a  foliation $\F_2$ with  same Camacho-Sad indices, with non abelian holonomy groups $F_2'$ and $F_2''$   along $D'$ and $D''$,   
but such that  there is a biholomorphism between neighborhoods of $D$ that conjugates $\mc F_2$ with $\mc L$. We have 
\[
\xymatrix{
F_2'\ar@{>>}[rd]_{\xi_1} 
&\oplus 
& \C/2i\pi\mb Z\ar@{->>}[rd]_{\xi_3}\ar@{->>}[ld]^{\xi_2} 
&\oplus
& F_2''\ar@{->>}[ld]^{\xi_4}\\ 
&
\C/2i\pi(\mb Z+\mu'\mb Z) 
&\oplus
& \C/2i\pi(\mb Z+\mu''\mb Z)
&
}\]
where again $\xi_2$ and $\xi_3$ are induced by the identity map. We easily obtain the exact sequence
\[
K\longrightarrow\C/2\pi i(\Z+2\wt\alpha \Z+2\wt\beta\Z) \longrightarrow \mr{Mod}([\mc F_2^\diamond])\longrightarrow 0\,,
\]
$K$ being finite. If $1$, $2\wt\alpha$, $2\wt\beta$ are $\Z$-independent then\begin{itemize}
\item[\bf -] \it $\mr{Mod}([\mc F_2^\diamond])$ is not a finite quotient of a product of elliptic curves; in particular,   in  the statement of Theorem~\ref{thm} we cannot replace $\Z^p$ by a finite group in the exact sequence (\ref{exseqmainThm}).
\end{itemize}

\indent $\bullet$ \emph{Example 3: infinite type  multicusps}\\ 
First  we choose the coefficients $\alpha_i$, $\beta_i$, $\delta$ in the expression of the 1-form $\omega$, so that $\mr{CS}(D',s'_0)\in \mb Z_{<0}$, the other Camacho-Sad indices being in $\C\setminus \R$, except for the points $s'_\infty$ and $s''_\infty$. 
At $s'_0$ the foliation $\mc L^\sharp$ possesses now a germ of holomorphic first integral, the local holonomy is a periodic rotation, thus  $\mr{Sym}^{\mc L^\diamond}_{s'_0}$ is isomorphic to $\mr{Diff}(\C,0)$. 
Then we perform a perturbation $\F_3$ of $\mc L$  changing only the holonomy groups $H_{D'}$ and  $H_D$ that become  non abelian, without changing $H_{D''}$ neither the local analytic types at  any singular point.
For such foliation $\F_3$ the group $\mr{Sym}^{\F^\diamond_3}_{s'_0}$ is always  isomorphic to $\mr{Sym}^{\F^\diamond_3}_{s'_0}\simeq\mr{Diff}(\C,0)$. The group-graph $\mr{Sym}^{\mc F_3^\diamond}$ is not abelian and its cohomology is no longer given by the cokernel of a morphism $\partial^0$.
However  
\[
\xi_4 : \mr{Sym}^{\F^\diamond_3}_{D''}\simeq \C/2\pi i\Z \longrightarrow \mr{Sym}^{\F^\diamond_3}_{s''_0}\simeq \C/2\pi i(\Z+\mu''\Z)
\]
 is always a submersion. Thus performing a new pruning we have an isomorphism  $H^1(\AR_{\F_3^\diamond}, \mr{Sym}^{\F^\diamond_3})\simeq H^1(\AR', \mr{Sym}^{\F^\diamond_3})$ with 
$\AR' := \xymatrix{
\bullet_{D'}\ar@{-}[r]^{s'_{0}}
& \bullet_{D}
}
$. The centralizer of $H_{D'}$ and $H_D$ being finite by Lemma \ref{centralinfini}, we obtain:
\begin{itemize}
\item[\bf -] \it $\mr{Mod}([\F_3^\diamond])$ is a quotient of $\mr{Diff}(\C,0)$ by the action of a finite group and $\AR_{\F_3^\diamond}$ is not connected.
\end{itemize}

\indent $\bullet$ \emph{Example 4:  Cremer  multicusps}\\ 
By gluing  techniques and thanks to realization Theorem \cite{Perez-Marco-Yoccoz} and P\'erez Marco results \cite{Perez-Marco} we can build a  foliation $\F_4$ with same separatrices, thus same resolution, as in previous examples, with non abelian holonomy groups $H_{D'}$, $H_D$, $H_{D''}$, but whose local holonomies at $s'_0$ and $s''_0$ are Cremer with uncountable centralizers. In this case 
\begin{itemize}\it
\item [\bf -] $\mr{Mod}([\F_4^\diamond])$ is a finite quotient of a product of two uncountable totally discontinuous subgroups of $\mb U(1):=\{z\in\C\;|\; \vert z\vert =1 \}$.
\end{itemize} 

\indent $\bullet$ \emph{Example 5: non-degenerate foliations with a single separatrix.}\\
For  such a foliation $\F_5$, after pruning all dead branches of the dual graph of $\mc E_{\F_5}$, the obtained graph is the red graph $\AR_{\F_5^\diamond}$  which is reduced to a geodesic segment
$\xymatrix{
\bullet_{D_0}\ar@{-}[r]^{s_1}
& 
\cdots\ar@{-}[r]^{s_\ell}
& 
\bullet_{D_\ell}
}$. 
All Camacho-Sad indices are rational numbers. The singular chains in $\AR_{\F_5^\diamond}$ are in two categories: the \emph{normalizable chains} whose edges $s$ correspond to normalizable resonant singularities of 
$\F_5^\sharp$ and  the {non-normalizable chains}. For the first one the group $\mr{Sym}^{\F_5^\diamond}_s$ is isomorphic to $\C^\ast$ and for non-normalizable chains it is isomorphic with $\Z/m_s\Z$ for a suitable $m_s\in \N$. It follows:
\begin{itemize}\it
\item[\bf -]  $\mr{Mod}([\F_5^\diamond])\simeq (\bigoplus_{i=1}^{\mu} \Z/m_i\Z \oplus \C^\ast{}^\nu)/ Z$, with $Z$ a finite subgroups, $\mu$, resp $\nu$,  the number of non-normalizable, resp. normalizable singular chains; furthermore $\mu+\nu$ is equal to the number of Puiseux pairs of the unique separatrix. 
\end{itemize}
Another specificity of this foliation $\F_5$ is that  the mapping class group of $\mc E_{\F_5}^\diamond$ 
is trivial because every singular point of $\mc E_{\F_5}$ is fixed\footnote{Each element of the mapping class group of $\mc E_{\F_5}^{\diamond}$ preserves dead branches so it must fix every singular point except maybe the attaching points of the two dead branches of the extremity valency $3$ divisor. But these two points have different Camacho-Sad indices, as can be easily deduced from \cite[p. 164]{LJDE}.} by $\mr{Mcg}(\mc E_{\F_5}^\diamond)$
and the pure mapping class group of $\mb P^1$ punctured at three points is trivial \cite[Proposition~2.3]{Farb}. From \S\ref{SubsectTopModSpace}.\ref{cc} we obtain that
\[\mr{Mod}([\F_5^\diamond])\subset
%[\MFe_{\msl{tr}}(\mc E_{\F_5}^{\diamond})]
\TFe(\mc E_{\F_5}^{\diamond})
\simeq[\Fe_{\msl{tr}}(\mc E_{\F_5}^{\diamond})]_{\mc C^0}.\]
 \indent $\bullet$ \emph{Example 6: some topologically $\mr{SL}$-rigid foliations.}\\ Whenever for a  marked foliation $\F^\diamond$ the red part of any cut-component of  $\A_{\F^\diamond}$  is reduced to one vertex, the moduli space   $\mr{Mod}([\F^\diamond])$ is reduced to one element. In particular this is the case for: 
\begin{enumerate}[\bf -]
\item any non dicritical foliation reduced after only one blow-up, its separatrices being smooth curves mutually transversal, or more generally any topologically quasi-homogeneous germ, see \cite{DM},
\item absolutely dicritical foliations of Cano-Corral \cite{CanoCorral},
\item dicritical foliations that are non singular after one blow-up, see \cite{Calsamiglia} and \cite{OBRGV2}.
 \end{enumerate}

\section{Exponential and disconnected group-graphs}\label{sec9}
We keep all notations used in Section \ref{sectionfoliationsTF}.
For technical reasons the last group-graphs that we must consider will be defined uniquely over the red graph $\AR_{\F^\diamond}\subset \A_{\F^\diamond}$.
Recall that $\mc X^{\F^{\diamond}}$, $\mc B^{\F^{\diamond}}$ and $\mc T^{\F^{\diamond}}$ denote the group-graphs over $\A_{\F^\diamond}$ associated to the sheafs $\mc X^{\F^\sharp}$, $\mc B^{\F^\sharp}$ and $\mc T^{\F^\sharp}=\mc B^{\F^\sharp}/\mc X^{F^\sharp}$ of tangent, basic and transverse holomorphic vector fields for $\F^\sharp$, respectively.
\begin{lema}
For $s\in\msl{Ed}_{\msl{A}_{\F^\diamond}}$ the exponential map $\exp:\mc B_{s}^{\F^\diamond}\to\mr{Aut}^{\F^\diamond}_{s}$
induces a well-defined map $\exp_{s}^{\F^\diamond}:\mc T_{s}^{\F^\diamond}\to\mr{Sym}^{{\F^\diamond}}_{s}$.
\end{lema}

\begin{dem} 
We must prove that
$\exp(Z+X)\equiv\exp(Z)$ modulo $\mr{Fix}^{{\F^\diamond}}_{s}$  if $Z\in\mc B^{\F^\diamond}_{s}$ and $X\in\mc X^{\F^\diamond}_{s}$. For that it suffices to show that for each neighborhood $V$ of $f(s)$ there is another neighborhood $U$ of $f(s)$ such that for each $p\in U$ the curve $\alpha\in[0,1]\mapsto\exp(Z+\alpha X)(p)$ is contained in a leaf of $\F_{|V}$. 
We choose $U\subset V$ such that the map $\phi:U\times \mb D_2\times\mb D_2\to V$ given by $\phi(p,t,\alpha)=\exp(t(Z+\alpha X))(p)$ is well-defined.  Fix $p\in U$ and take a local holomorphic first integral $F$ defined in a neighborhood $W$ of $p$. If $t$ is small enough then $\phi(p,t,\alpha)\in W$ and 
\begin{align*}
\frac{\partial}{\partial t}\left(\frac{\partial}{\partial\alpha}F(\phi(p,t,\alpha))\right)&=\frac{\partial}{\partial\alpha}\left(\frac{\partial}{\partial t}F(\phi(p,t,\alpha))\right)\\&=\frac{\partial}{\partial\alpha}\left((Z+\alpha X)(F)\circ\phi(p,t,\alpha)\right)\\
&=\left[X(F)\circ\phi(p,t,\alpha)\right]\frac{\partial}{\partial\alpha}\phi(p,t,\alpha)=0
\end{align*} 
because $X$ is tangent to $\F^\sharp$. Since $F(p,0,\alpha)=p$ does not depend on $\alpha$ we obtain that $\frac{\partial}{\partial\alpha}F(\phi(p,t,\alpha))=0$ for $t$ small enough. As $\phi$ is holomorphic, we conclude that the curve $\alpha\mapsto\phi(p,1,\alpha)$ is contained in a leaf of $\F^\sharp_{|V}$. 
\end{dem}

\begin{obs}\label{exps}
It can be checked that under the identifications ${\mc T}_{s}^{\F^\diamond}\simeq \mc T_{h_{D,s}}$ and $C(h_{D,s})/\langle h_{D,s}\rangle\simeq \mr{Sym}^{{\F^\diamond}}_{s}$ given by Remark~\ref{vfsigma} and Corollary~\ref{isosym}, the morphism
$\exp_{s}^{\F^\diamond}$ coincides with the composition of the restriction  $ \mc T_{h_{D,s}}\to C(h_{D,s})$ of the exponential map on the transverse section $\Delta_{D}$ and the quotient map $C(h_{D,s})\to C(h_{D,s})/\langle h_{D,s}\rangle$.
\end{obs}

Motivated by the above remark, for $D\in\msl{Ve}_{\msl{A}_{\F^\diamond}}$ we define the map   
$\exp_{D}^{\F^\diamond}:\mc T^{\F^\diamond}_{D}\to\mr{Sym}^{{\F^\diamond}}_{D}$ as the composition
$\mc T^{\F^\diamond}_{D}\simeq\mc T_{H_D}\stackrel{\exp}{\to}C(H_D)\to\mr{Sym}_{D}^{\F^\diamond}$.
From Remark~\ref{exps} it is clear that the following diagram is commutative:
\begin{equation}\label{expDs}
\xymatrix{\mc T_{D}^{\F^\diamond}\ar[d]_{\rho_{D}^{s}}
\ar[r]^{\exp_{D}^{\F^\diamond}}& \mr{Sym}^{{\F^\diamond}}_{D}\ar[d]^{\rho_{D}^{s}}
\\
\mc T_{s}^{\F^\diamond}\ar[r]^{\exp_{s}^{\F^\diamond}}& \mr{Sym}^{{\F^\diamond}}_{s}}
\end{equation}
the vertical maps being the restriction maps of the group-graphs $\mc T^{\F^\diamond}$ and $\mr{Sym}^{\F^\diamond}$, written with same notation.\\

Although the exponential map
\[\exp:\C\{z\}z\partial_{z}\to\mr{Diff}(\Delta_{D},f(o_{D}))\simeq\mr{Diff}(\C,0),\]
$z:(\Delta_D,f(o_D))\to(\C,0)$ being a germ of coordinate,
is not a morphism of groups, its restriction to a subspace of complex dimension $\le 1$ is. 
On the other hand it is well-known that $\dim_{\C}\mc T_{s}^{\F^\diamond}\le 1$ if $\mc J^{\F^\diamond}_{s}=\C$. Since $\mc T_{D}^{\F^\diamond}\subset\mc T_{s}^{\F^\diamond}$ for $s\in D$ we deduce that $\exp^{\F^\diamond}_D$ and $\exp^{\F^\diamond}_s$  define a morphism 
\[\exp^{\F^\diamond}:\mc T^{\F^\diamond}\to\mr{Sym}^{{\F^\diamond}}\] 
of abelian group-graphs over $\AR_{\F^\diamond}$.
\begin{defin}\label{arbregrpexp}
The group-graph over $\AR_{\F^\diamond}$ image of $\mathrm{exp}^{\F^\diamond}$ is called  the \emph{exponential} group-graph of $\F^\diamond$. We denote it by $\mr{Exp}^{{\F^\diamond}}$.
\end{defin}

At this point it is clear that the subset of $\AR_{\F^\diamond}$ consisting of all $a\in\msl{Ve}_{\msl{A}_{\F^\diamond}}\cup\msl{Ed}_{\msl{A}_{\F^\diamond}}$ such that $\mr{Exp}^{\F^\diamond}_{a}=0$ is just  the subgraph $\AR^0_{\F^\diamond}$  of $\AR_{\F^\diamond}$ previously defined in Section \ref{sectionfoliationsTF} and characterized by the second equivalence in Theorem~\ref{strutgaaut}.
Let us denote by  $\AR^1_{\F^\diamond}$  the \emph{completion} of $\AR_{\F^\diamond}\setminus \AR^0_{\F^\diamond}$, i.e. the minimal subgraph of $\AR_{\F^\diamond}$ containing $\AR_{\F^\diamond}\setminus \AR^0_{\F^\diamond}$.

\begin{lema}\label{surj}
If $(D,s)\in I_{\mathsf{R_{\F^\diamond}}}$ and $\mc T_{D}^{\F^\diamond}\neq 0$, then the restriction map $\rho':\mc T_{D}^{\F^\diamond}\to \mc T_{s}^{\F^\diamond}$ is an isomorphism and all the red singular points
in $D$ share the same character linearizable or resonant; we will say that $D$ is resonant or linearizable according to the case. Furthermore, the isomorphism class of the group $\mr{Exp}^{\F^\diamond}_{D}$ is given by the following table
\begin{equation}\label{tab}
\begin{tabular}{c|c|c|}
 $\mr{Exp}_{D}^{\F^\diamond}$ & {$ \mr{val}_{\Sigma}(D) \le 2$} & {$\mr{val}_{\Sigma}(D)\ge 3$}\\
\hline
$D$ \text{resonant} & $\C/\Z$ & $\C$\\
\hline
$D$ \text{linearizable} & $\C^{*}/\alpha^{\Z}$ & $\C^{*}$\\
\hline
\end{tabular}
\end{equation}
the restriction morphism $\rho_{D}^{s}:\mr{Exp}_{D}^{\F^\diamond}\to\mr{Exp}_{s}^{\F^\diamond}$ is surjective and 
$$\ker\rho_{D}^{s}\simeq\left\{
\begin{array}{lll}\mb Z & \text{if} & \mr{val}_{\Sigma}(D)\ge 3,\\
0 & \text{if} & \mr{val}_{\Sigma}(D)\le 2.
\end{array}\right.$$
\end{lema}

\begin{proof} The homogeneity of singular types in $D$ is given by Proposition \ref{memetypes}.
Notice that if a basic vector field for $\F^\sharp$ defined on a connected open set $U\subset M_\F$ is tangent to the foliation in a neighborhood of a point of $U$ then it is tangent to the foliation on the whole $U$. Using this fact 
it is easy to see that if $W$ is a connected subset of a $\F^\sharp$-invariant component of $ \mc E_\F$, the stalk maps $\mc T^{\F^\sharp}(W)\rightarrow \mc T_{m}^{\F^\sharp}$, $m\in W$ are injective. 
Taking $W=D$ we deduce that the restriction map
${\rho'}^s_D : \mc T_{D}^{\F^\diamond}\to\mc T^{\F^\diamond}_{s}$ is injective.  
Since $\dim \mc T_{D}^{\F^\diamond}=\dim\mc T_{s}^{\F^\diamond}=1$, ${\rho'}^s_D$ is an isomorphism. In fact, $\mc T_{s}^{\F^\diamond}\simeq\mc T_{D}^{\F^\diamond}$ can be identified to a line $\C X$ in the space of germs of vector fields on $(\Delta_{D}, f(o_D))$. Then Table~(\ref{tab}) follows from Proposition~\ref{cent}.

On the other hand  $\exp_{s}^{\F^\diamond}:\mc T_{s}^{\F^\diamond}\to \mr{Exp}^{\F^\diamond}_{s}$ being surjective by definition and the restriction map ${\rho'}^s_D$ being an isomorphism, it follows that
 $\rho_{D}^{s} \circ \exp_{D}^{\F^\diamond}=\exp_{s}^{\F^\diamond}\circ{\rho'}^s_D$ is surjective. Therefore $\rho_{D}^{s}$ is also surjective. We conclude thanks to Remark~\ref{exps} and the commutativity of the diagram~(\ref{expDs}).
\end{proof}
 
Thanks to Theorem~\ref{strutgaaut},
$H^{1}(\AR_{\F^\diamond},\mr{Sym}^{\F^\diamond})$ is an abelian group
and the natural inclusion $\mr{Exp}^{\F^\diamond}\to\mr{Sym}^{\F^\diamond}$ is an injective morphism of abelian group-graphs over $\AR_{\F^\diamond}$.
\begin{defin}
The quotient group-graph over $\AR_{\F^\diamond}$
\[\mr{Dis}^{\F^\diamond}:=\mr{Sym}^{\F^\diamond}/\mr{Exp}^{\F^\diamond}\] 
given by $\Dis_{a}:=\mr{Sym}^{\F}_{a}/\mr{Exp}^{\F}_{a}$ for every $a\in\msl{Ve}_{\AR_{\F^\diamond}}\cup\msl{Ed}_{\AR_{\F^\diamond}}$, is called the \emph{disconnected group-graph of $\F^\diamond$}.
\end{defin}
\noindent Obviously, we have  over $\AR_{\F^\diamond}$ the short exact sequence of abelian group-graphs:
\begin{equation}\label{seq}
0\to \mr{Exp}^{\F^\diamond}\to\mr{Sym}^{\F^\diamond}\to\Dis\to 0.
\end{equation}
The name ``disconnected'' is explained by the following proposition. 
\begin{prop}\label{Dis}
For $s\in\msl{Ed}_{\AR_{\F^\diamond}}$ and $D\in\msl{Ve}_{\AR_{\F^\diamond}}$ we have that 
\begin{enumerate}
\item\label{diss} the abelian group $\mr{Dis}^{\F^\diamond}_{s}\simeq C(h_{D,s})/\langle h_{D,s},\, \exp (\mc T_{h_{D,s}})\rangle$  is:
\begin{enumerate}[(a)]
\item trivial if $f(s)$ is a linearizable (not-periodic) singularity
\item a finite abelian group if $f(s)$ is a resonant (non-periodic) singularity, cyclic in the normalizable case;
\item \label{casdis} a cyclic quotient of a totally disconnected subgroup of $\mb U(1)$ if $f(s)$ is a non-resonant and non-linearizable singularity;
\end{enumerate}
\item\label{disD}  the abelian group $\mr{Dis}_{D}^{\F^\diamond}$ is 
\begin{enumerate}[(a)]
\item\label{disDresnnor} infinite and of finite type if $H_D$ is abelian and all the red singularities on $f(D)$ are resonant non-normalizable;
\item\label{disDnresnlin} a cyclic quotient of a totally disconnected subgroup of $\mb U(1)$ if $H_D$ is abelian and all the red singularities on $f(D)$ are  non-resonant non-linearizable;
\item\label{remaining} finite in all the remaining cases;
\end{enumerate}
\item\label{disfini0} $\mr{Dis}_s^{\F^\diamond}$ and $\mr{Dis}_D^{\F^\diamond}$ are finite if $s\in \msl{Ed}_{\AR^1_{\F^\diamond}}$ and $D\in\msl{Ve}_{\AR^1_{\F^\diamond}}$.
\end{enumerate}
\end{prop}

\begin{proof}
Assertions (\ref{diss})  result directly from Proposition \ref{cent}. 

To obtain Assertions (\ref{disD}) we can suppose that the singular valency of $D$ is at least three, otherwise $\mr{Dis}^{\F^\diamond}_D=\rm{Dis}^{\F^\diamond}_s$ for $s\in D\cap \Sigma$ and Assertions~(\ref{disD})  result again directly from Proposition~\ref{cent}.  Now let us suppose also that $\mr{Dis}^{\F^\diamond}_D$ -thus also $C(H_D)$- is infinite. Because $D$ is red, $H_D$ is infinite and it follows from Lemmas~\ref{egalcentrs} and~\ref{centralinfini}  that the set $H'$ of all non periodic elements of $H_D$ is non empty and for every $h\in H'$ we have $C(H_D)=C(h)$, therefore $\mr{Dis}^{\F^\diamond}_D=C(h)/\mr{exp}(\mc T_h)$. By Proposition \ref{cent} the only case where this group is infinite is when $h$ is resonant non normalizable or non resonant non linearizable. To see that these two possibilities correspond to the cases (\ref{disDresnnor}) and (\ref{disDnresnlin}) above, it is enough to notice that the local holonomies $h_{D,s}$, $s\in D\cap \Sigma$, that generate $H_D$, cannot be all periodic (otherwise by abelianity $H_D$ would be finite), and to use Proposition~\ref{memetypes}.

Assertion (\ref{disfini0})  follows immediately from Assertions (\ref{diss}) and (\ref{disD}) except for $\rm{Dis}^{\F^\diamond}_{D}$ when $D$ is a common vertex of $\msl{R}^1_{\F^\diamond}$ and $\msl{R}^0_{\F^\diamond}$. In this case although $\mr{Exp}^{\F^\diamond}_D=0$, at the meeting points $s$ of $D$ with components $D'$ of $\msl{R}^1_{\F^\diamond}$ we have $\mr{Exp}^{\F^\diamond}_s\neq 0$ because $\mr{Exp}^{\F^\diamond}_{D'}\neq 0$. Therefore $D$ does not correspond to case  (\ref{disDresnnor})  nor case (\ref{disDnresnlin}) and  $\rm{Dis}^{\F^\diamond}_D$ is finite
according to~(\ref{remaining}).
\end{proof}

\section{Proof of Theorem~\ref{thm}}\label{proof}
In order to simplify the notations in the proofs below, we will write again $\msl A$, $\msl R$, ${\msl R}^0$, ${\msl R}^1$, $\mr{Aut}$,  $\mr{Sym}$, $\mr{Exp}$, $\mr{Dis}$ and $\tau$,  
instead of  $\msl A_{\F^\diamond}$, $\msl R_{\F^\diamond}$, ${\msl R}_{\F^\diamond}^0$,  ${\msl R}_{\F^\diamond}^1$, $\mr{Aut}^{\F^\diamond}$,
$\mr{Sym}^{\F^\diamond}$, $\mr{Exp}^{\F^\diamond}$, $\mr{Dis}^{\F^\diamond}$ and $\tau_\F$.\\

We have already shown that there are ``natural'' bijections:
\begin{equation}\label{eq5}
\mr{Mod}([\F^{\diamond}])
\stackrel{(\ref{eq0})}{\simeq}
H^{1}(\A,\mr{Aut})
\stackrel{(\ref{eq1})}{\simeq}
H^{1}(\A,\mr{Sym})
\stackrel{(\ref{eq2})}{\simeq}
H^{1}(\AR,\mr{Sym})\,,
\end{equation}
the bijection $(\ref{eq2})$ being only valid when $\F$ has finite type.
Moreover $\mr{Sym}$ is an abelian group-graph over $\AR$ and consequently $H^{1}(\AR,\mr{Sym})$ is an abelian group, cf. Theorem~\ref{strutgaaut}.
Recall that $\AR^{0}$ is the subgraph of $\AR$ constituted by all the vertices and edges $b$ satisfying
 $\mr{Exp}^{\F}_{b}=0$ 
 and $\AR^1$ is the completion of $\AR\setminus\ARZ$.
The rest of the proof is divided in several steps:
\begin{enumerate}[ (i)]
\item\label{Rsym-Rexp}
 The abelian group $H^{1}(\AR,\mr{Sym})$ fits into an exact sequence
 \[0\to F\to H^{1}(\AR,\mr{Exp})\to H^{1}(\AR,\mr{Sym})\to \mbf D\to 0,\]
 where $F$ is a finite abelian group and $\bf D$ is a totally disconnected topological abelian group.
 \item\label{R-zones} We have  group isomorphisms
$$H^{1}(\AR,\mr{Exp})\simeq H^{1}(\AR^1,\mr{Exp})\simeq\bigoplus_{\alpha}H^{1}(\AZ^{\alpha},\mr{Exp}),$$ 
where $\AR^1=:\bigcup
\limits_{\alpha\in\pi_0(\AR\setminus\ARZ)}\AZ^{\alpha}$ where each \emph{zone} $\AZ^{\alpha}$ is the completion of a connected component of 
$\AR\setminus\AR^0=\AR^1\setminus \AR^0$.
\item \label{extendedzone} To simplify the computation of the cohomology groups  $H^{1}(\AZ^{\alpha},\mr{Exp})$  we modify each zone (not reduced to a single vertex) without changing the number of its extremities neither its cohomology, by adding a vertex and an edge, for each of its extremities. The modified zones fulfill the following property: \\

\noindent $(\ast)$ \textit{each extremity of $\AZ^{\alpha}$ is joined by its edge to a vertex of valency~$2$ in $\AZ^{\alpha}$.}\\
\item\label{zone-sommets} We  decompose each modified zone $\AZ$ as $\AZ=\AZ_{0}\cup \AZ_{1}$ where $\AZ_{0}$ is either empty or  a disjoint union of $n+1\ge 1$ segments  $\bullet_{D'_i}{\ftrait}\bullet_{D_i}$  with  $D'_i\in \ARZ$, $D_i\in \ARU$ and  $\AZ_{0}\cap \AZ_{1}=\{D_{0},\ldots,D_{n}\}$.
 We prove  that $H^{1}(\AZ,\mr{Exp})$ is trivial if $\AZ_0$ is empty, and it is a finite type quotient of $\bigoplus\limits_{i=1}^{n}\mr{Exp}_{D_{i}}$ if $\AZ_0\neq \emptyset$.
\item\label{sommets} Since 
$\mr{Exp}_{D_i}$ is isomorphic to $\C$ or $\C/\Z\simeq\C^{*}$ or $\C^{*}/\alpha^{\Z}$ by Lemma~\ref{surj}, we can construct a  morphism $\Lambda :\C^\tau\to \mr{Mod}([\F^{\diamond}])$ with totally disconnected cokernel and finite type kernel.
\item\label{recallSL} We  will  specify  the notion of semi-local-equisingularity, denoted by SL-equisingularity. This notion was introduced in \cite{MS} for germs of deformations and in this paper we adapt it to the context of a global parameter space.
\item\label{famille} We construct  SL-equisingular families of foliations $\F^{U}_{t,i}$ satisfying Theorem~\ref{thm}.
\end{enumerate}

\medskip

\paragraph{\textbf{Step (\ref{Rsym-Rexp})}}
Consider the long exact sequences 
\begin{equation}\label{fin}
\cdots\to
H^{0}(\AR, \mr{Dis})
\to H^{1}(\AR,\mr{Exp})\stackrel{\chi}{\to} H^{1}(\AR,\mr{Sym})\to H^{1}(\AR,\mr{Dis})\to 0
\end{equation}
and
\[\cdots{\to} H^{0}(\AR_1,\mr{Dis})
\to H^{1}(\AR_1,\mr{Exp})\stackrel{\chi_1}{\to} H^{1}(\AR_1,\mr{Sym})\to H^{1}(\AR_1,\mr{Dis})\to 0
\]
associated by Lemma~\ref{LES} to the short exact sequence (\ref{seq}) of abelian group-graphs.
By the first part of Proposition~\ref{Dis},  $Z^{1}(\AR,\mr{Dis})$ is a finite product of totally disconnected subgroups of $\mb U(1)$ and $H^{1}(\AR,\mr{Dis})$ is thus a totally disconnected abelian topological group. 
Moreover when all the singularities of the foliation are resonant or linearizable, the case \ref{casdis}) is excluded and $Z^{1}(\AR,\mr{Dis})$ is finite. 
\color{black} 
In order to conclude this step it only remains to prove that $\ker \chi$ is finite. 
Let us notice that $H^{0}(\AR_{0}\cap\AR_{1},\mr{Exp})=0$, 
$H^{1}(\AR_{0},\mr{Exp})=0$ and $H^{1}(\AR_{0}\cap\AR_1,\mr{Exp})=0$.
By applying Mayer-Vietoris Lemma~\ref{MV} to the union $\AR=\AR_0\cup \AR_1$ we obtain the following commutative 
diagram with exact rows
\begin{equation}\label{snake}
\xymatrix{ & 0\ar[r] & H^{1}(\AR,\mr{Exp})\ar[r]\ar[d]^{\chi}  & 0\oplus H^{1}(\AR_{1},\mr{Exp})\ar[r]\ar[d]^{0\oplus \chi_1} & 0& \\
  & \cdots \ar[r] & H^{1}(\AR,\mr{Sym})\ar[r] & H^{1}(\AR_{0},\mr{Sym})\oplus H^{1}(\AR_{1},\mr{Sym}) \ar[r] &\cdots
  }
\end{equation}
Thus $\ker \chi$ is isomorphic to a subgroup of $\ker \chi_1$  and it is sufficient to prove that $H^0(\AR_1, \mr{Dis}^\F)$ is finite. But using the second part of Proposition \ref{Dis}  we obtain that $C^0(\AR_1, \mr{Dis}^\F)$ is finite. 
\medskip

\paragraph{\textbf{Step (\ref{R-zones})}}

Diagram~(\ref{snake}), coming from the Mayer-Vietoris sequence, gives us the isomorphism
\begin{equation}\label{isorun}
H^{1}(\AR,\mr{Exp})\simeq H^{1}(\AR_1,\mr{Exp})\,.
\end{equation}
\color{black}
Clearly $\ARU$ is a finite union of zones $\AZ^{\alpha}$. Moreover, $\AZ^{\alpha}\cap \AZ^{\beta}$ is either empty or a single vertex of $\ARZ$. Therefore $$H^{0}(\AZ^{\alpha}\cap \AZ^{\beta},\mr{Exp})=0\quad\textrm{and}\quad H^{1}(\AZ^{\alpha}\cap \AZ^{\beta},\mr{Exp})=0.$$ 
By applying recursively Mayer-Vietoris Lemma~\ref{MV} we deduce that 
\begin{equation}\label{isomayerzones}
H^{i}(\ARU,\mr{Exp})\simeq 
\bigoplus\limits_{\alpha\in\pi_0(\AR\setminus\ARZ)} H^{i}(\AZ^{\alpha},\mr{Exp})\,,\quad i=0,1\,.
\end{equation}

\paragraph{\textbf{Step (\ref{extendedzone})}} 
If a zone $\AZ^\alpha$ is reduced to a  single vertex then $H^1(\AZ^\alpha, \mr{Exp})$ is clearly trivial. If this is not the case, we  
modify $\AZa$ in  the following way:   if   $v'$ is an extremity of $\AZa$ and $v''\in\VZa$ is the unique vertex joined to $v'$ by an edge $e'$,  we  replace  the segment $\bullet_{v''}\stackrel{e'}{\ftrait}\bullet_{v'}$ by 
$\bullet_{v''}\stackrel{e''}{\ftrait}
\bullet_{v}\stackrel{e}{\ftrait}\bullet_{v'}$. 
We also extend the group-graph  $\mathrm{Exp}$ to the new edges and vertices by  defining 
\begin{equation}\label{Modif}
\mathrm{Exp}_{e''}:=\mathrm{Exp}_{e}:=
\mathrm{Exp}_{v}:=\mathrm{Exp}_{e'}\,,\qquad \rho_{v}^{e''}:=\rho_{v}^{e}:=\mathrm{id}_{\mathrm{Exp}_{e'}}\,,
\end{equation}
$$ 
\rho_{v''}^{e''}=\rho_{v''}^{e'}\,,\qquad \rho_{v'}^{e}=\rho_{v'}^{e'}\,.$$
We call this operation the \emph{blow-up of the edge $e'$}. By performing these blow-ups for each extremity of $\AZa$ we get a new graph  $\wAZa$ called \emph{a modified zone}.
Clearly $\wAZa$ fulfills property $(\ast)$ of  (\ref{extendedzone}).

By doing this process on each zone, we get a \emph{modified graph} $\wt\AR$ endowed with a group-graph still denoted by $\mathrm{Exp}$.
We define now  a contraction map: 
$$Z^1(\wt\AR,\mr{Exp})\stackrel{\flat}{\rightarrow} Z^1(\AR, \mathrm{Exp})\,,\qquad c=(\phi_{ ve})\mapsto c^\flat=(\phi^\flat_{ve})$$
where $\phi_{ve}^\flat=\phi_{ve}$ if $e$ is not produced by a blow-up, and $\phi_{v''e'}:=\phi_{v''e''}\phi_{ve}=:\phi_{v'e'}^{-1}$ 
 if $\bullet_{v''}\stackrel{e''}{\ftrait}
\bullet_{v}\stackrel{e}{\ftrait}\bullet_{v'}$ is given by the blow-up of $\bullet_{v''}\stackrel{e'}{\ftrait}\bullet_{v'}$.
It is easy to see that this map induces  group isomorphisms 
\begin{equation}\label{isomod}
H^1(\wAZa, \mathrm{Exp})\simeq H^1(\AZa,\mathrm{Exp})\,,\quad  H^1(\wt\AR, \mr{Exp})\simeq H^1(\AR,\mathrm{Exp})\,.
\end{equation}

\medskip

\paragraph{\textbf{Step (\ref{zone-sommets})}} Fix $\AZ=\wAZa$ a modified zone of $\ARU $. 
Let $\AZ_1$ be the maximal subgraph ot $\AZ$ with all vertices and edges $b$ satisfying $\mr{Exp}_{b}\neq 0$. Denote by $\AZ_0$ the completion of $\AZ\setminus \AZ_1$. Clearly $\AZ=\AZ_{0}\cup \AZ_{1}$ and  $\AZ_{0}$ is either empty or a disjoint union of $n+1\ge 1$ segments $\bullet_{D'}{\ftrait}\bullet_{D}$ with $\mr{Exp}_{D'}= 0$,  $\mr{Exp}_{D}\neq 0$,  $D'$ being an extremity of $\AZ$ and $\mathrm{val}_{\msl Z}( D)=2$. Notice that $H^{1}(\AZ_{1},\, \mathrm{Exp})=0$. Indeed  the restriction morphisms of the group-graph  $\mathrm{Exp}^\F$ over $\AZ_1$ are surjective by Lemma \ref{surj}. We  apply recursively Pruning Theorem~\ref{elagage} and we conclude by Remark \ref{pruningiteres} that $H^1(\AZ, \mathrm{Exp})=0$ if $\AZ_0$ is empty.

Now suppose that $\AZ_0\neq\emptyset$. We will apply  Mayer-Vietoris Lemma~\ref{MV}  to $\AZ=\AZ_0\cup\AZ_1$. Using again Lemma~\ref{surj} we see that  $H^{1}(\AZ_{0},\mr{Exp})=0$ and $H^{0}(\AZ_{0},\mr{Exp}^{\F})=0$  by construction of the modified zones.  We obtain the exact sequence
\begin{equation}\label{seq1}
H^{0}(\AZ_{1},\mr{Exp}^{\F})\stackrel{\sigma_{\alpha}}{\longrightarrow}H^{0}(\AZ_{0}\cap\AZ_{1},\mr{Exp})\stackrel{\delta_{\alpha}}{\to} H^{1}(\AZ,\mr{Exp}^{\F})\to 0,
\end{equation}
 $\sigma_\alpha$ being   the restriction map and $\delta_\alpha$ the connecting map.  \\

In the sequel we will choose one  vertex $D_{0}$ in $\AZ_{0}\cap \AZ_{1}=\{D_{0},\ldots,D_{n}\}$ and we will call the remaining vertices  $D_{1},\ldots,D_{n}$ the \emph{active vertices of the zone $\AZ$}.

\begin{lema}\label{surjsursommet}
The projection $\pi_\alpha : H^{0}(\AZ_{1}, \mathrm{Exp})\to H^{0}(D_{0}, \mathrm{Exp})\simeq \mr{Exp}_{D_0}$ is surjective and its kernel is an abelian group of finite type.
\end{lema}
\begin{proof}
From Lemma~\ref{surj} for each $(s,D)\in I_{\AZ_{1}}$ the restriction map $\rho_{D}^{s}$ is surjective with kernel $0$ or $\Z$. 
%Let $\bullet_{D_0}\stackrel{s_i}{\ftrait}\bullet_{D_i}$, $i=1, \ldots,\ell$ be the edges passing through $D_{0}$
Let $s_i$, $i=1, \ldots,\ell$, be the edges such that $D_{0}\in \partial s_i$.
For each  $y_{0}\in\mr{Exp}_{D_{0}}$ there are $y_{i}\in\mr{Exp}_{s_{i}}$, $i=1,\ldots,\ell$, such that $\rho_{D_{i}}^{s_{0}}(y_{i})=\rho_{D_{0}}^{s_{0}}(y_{0})$, see Figure~\ref{es}. 
Moreover, the different choices of $y_{i}$ are parametrized by $\Z^{k}$ with $k\leq \ell$. By induction we easily deduce that $\pi_\alpha$ is surjective and its kernel is of finite type.
\end{proof}

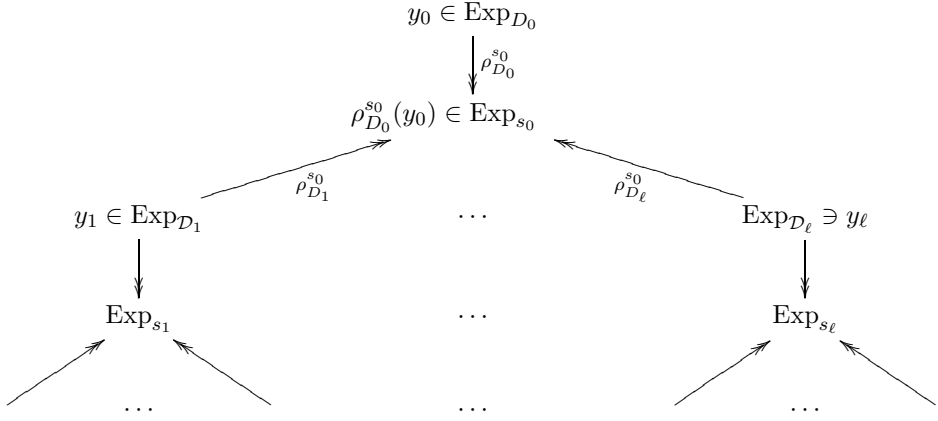
\begin{figure}
$$\scalebox{.9}{\xymatrix{& & & y_{0}\in\mr{Exp}_{D_{0}}\ar@{->>}[d]^{\rho^{s_{0}}_{D_{0}}}&&&\\ 
& & &\rho_{D_{0}}^{s_{0}}(y_{0})\in\mr{Exp}_{s_{0}}\hphantom{lllllll}\ar@{<<-}[rrd]_{\rho_{D_{\ell}}^{s_{0}}}\ar@{<<-}[dll]^{\rho^{s_{0}}_{D_{1}}} &&&\\ 
& y_{1}\in\mr{Exp}_{\mc D_{1}}\ar@{->>}[d]   && \cdots  && \mr{Exp}_{\mc D_{\ell}}\ni y_{\ell}\ar@{->>}[d]&\\ 
&\mr{Exp}_{s_{1}}& & \cdots & & \mr{Exp}_{s_{\ell}} &\\ 
\ar@{->>}[ur]&\cdots & \ar@{->>}[ul] &\cdots &  \ar@{->>}[ur]& \cdots & \ar@{->>}[ul]}}$$
\caption{Schematic diagram for the group-graph $\mr{Exp}$ restricted to $\AZ_{1}$.}\label{es}
\end{figure}

Denote by $\rho:H^{0}(\AZ_{0}\cap \AZ_{1}, \mathrm{Exp})\to H^0(D_0, \mathrm{Exp})\simeq\mr{Exp}_{D_0} $ the projection map.
Then we have the following morphism of exact sequences
$$\xymatrix{&H^{0}(\AZ_{1},\mr{Exp})\ar[r]^{\sigma_{\alpha}\phantom{aaa}}\ar[d]_{\pi_\alpha}&  H^{0}(\AZ_{0}\cap \AZ_{1},\mr{Exp})\ar[r]^{\phantom{aaa}\delta_\alpha}\ar[d]^{\rho} & H^{1}(\AZ,\mr{Exp})\ar[r]\ar[d] &0\\
 0\ar[r]&\mr{Exp}_{D_{0}}\ar[r]^{\mr{id}}& \mr{Exp}_{D_{0}}\ar[r]& 0&&}$$
Since $\AZ_{0}\cap \AZ_{1}$ only contains the vertices $D_{i}$ we have that $H^{0}(\AZ_{0}\cap Z_{1}, \mathrm{Exp})=\bigoplus_{i=0}^{n}\mr{Exp}_{D_{i}}$ and $\ker\rho$ is just $\bigoplus_{i=1}^{n}\mr{Exp}_{D_{i}}$.
Because $\pi_\alpha$ is surjective, by applying Snake Lemma  we obtain the following exact sequence:
\begin{equation}\label{exkerpi}
\ker(\pi_\alpha)\to \bigoplus_{i=1}^{n}\mr{Exp}_{D_{i}}\stackrel{\delta_\alpha}{\rightarrow} H^{1}(\AZ,\mr{Exp})\to \mathrm{coker}(\pi_\alpha)\,.
\end{equation}
Since $\pi_\alpha$ is surjective and $\ker(\pi_\alpha)$ is of finite type, thanks to  Lemma~\ref{surjsursommet} we get that $H^{1}(\AZ,\mr{Exp})$  is a quotient of $\C^n$ by a finite type subgroup.

\medskip

\paragraph{\textbf{Step (\ref{sommets})}}
First  we notice that  the number of active vertices of a modified zone $\AZ=\wt{\AZ}^\alpha $ is equal to the rank of the homology groups  $H_1(\AZ/\AZ^0, \Z) \simeq H_1(\AZ^\alpha/(\AZ^\alpha\cap \ARZ),\Z)$ 
of the corresponding quotient graphs. 
We easily deduce that the number of all active vertices $a_r$ for all the modified zones
is equal to the rank
$\tau:=\mr{rank}\, H_{1}(\AR/\mathsf{R}_{0},\Z)$ introduced in Definition \ref{defindetau}.

Each active vertex $a_r$, $r=1,\ldots,\tau$,  belonging to some modified  zone, is produced by the blow-up of an edge 
$\bullet_{v''_r}\stackrel{s_r}{\ftrait}\bullet_{v'_r}$, $v'_r$ being an extremity of the  zone. By construction,  $\mr{Exp}_{a_r}=\mr{Exp}_{s_r}$, cf.  (\ref{Modif}).
With this identification and thanks to the isomorphisms (\ref{isorun}),  (\ref{isomayerzones}) and  (\ref{isomod})  it can be easily checked, using the proof of  Lemma~\ref{MV}, that the  map $\delta=\oplus_{\alpha}\delta_\alpha$ given by the connecting maps $\delta_\alpha$ of the Mayer-Vietoris exact sequences (\ref{seq1}) is the surjective  morphism 
$$\delta \; : \;\bigoplus\limits_{r=1}^{\tau}\mr{Exp}_{a_r} =  \bigoplus\limits_{r=1}^{\tau}\mr{Exp}_{s_r}\;\ni\;
(\varphi_r)\longmapsto [(\phi_{ve})]\in  H^{1}(\AR,\mr{Exp})
$$
with $\phi_{v''_rs_r}:=\varphi_{r}$, $\phi_{v'_rs_r}:=\varphi_{r}^{-1}$ and $\phi_{ve}$ is trivial otherwise.
\\

Now for each $r=1,\ldots,\tau$  we choose  a local holomorphic basic vector field $X_r$ transverse to the foliation $\F^\sharp$ and defined on a neighborhood of $f(s_r)$ in the ambient space of $\F^\sharp$. 
We define the group morphism  $\Lambda:\C^\tau\to \mr{Mod}([\F^{\diamond}])$ of  Theorem~\ref{thm}, as the composition
$$\C^{\tau}\stackrel{\xi}{\twoheadrightarrow}\bigoplus\limits_{r=1}^{\tau}\mr{Exp}_{s_r}
\stackrel{\delta}{\twoheadrightarrow} H^{1}(\AR,\mr{Exp})\stackrel{\chi}{\to} H^{1}(\AR,\mr{Sym})\stackrel{(\ref{eq5})}{\simeq}\mr{Mod}([\F^{\diamond}])\,,$$
where $\xi(t)=(\overset{{}_\bullet}{\exp} t_{1}X_{1},\ldots, \overset{{}_\bullet}{\exp} t_{\tau}X_{\tau})$,
$\overset{{}_\bullet}{\exp} t_rX_r$ denotes the class of $\exp t_r X_r$ in $\mr{Aut}_{s_r}/\mr{Fix}_{s_r}=\mr{Sym}_{s_r}$ and $\chi$ is induced by the natural inclusion of group-graphs $\mr{Exp}\hookrightarrow \mr{Sym}$, see (\ref{seq}).
The last bijection  (\ref{eq5}) induces an abelian group structure on $\mr{Mod}([\F^{\diamond}])$. Moreover, if we define $\mbf D:=H^{1}(\AR,\mr{Dis}^{\F})$ and  $\Gamma : \mr{Mod}([\F^{\diamond}])\to \mbf D$ as the composition of the isomorphism~(\ref{eq5}) and  the last arrow in the sequence~(\ref{fin}),  then the sequence 
$$
\C^{\tau}\stackrel{\Lambda}{\to} \mr{Mod}([\F^{\diamond}])\stackrel{\Gamma}{\to}\mbf D\to 0
$$
is exact.
It remains to check that $\ker \Lambda =\ker(\chi\circ\delta\circ\xi)$ is of finite type. 
Since $\ker\chi$ is finite by step (i) and $\ker\delta$ is of finite type thanks to Lemma~\ref{surjsursommet} and~(\ref{exkerpi}),  it suffices to see that $\ker\xi$ is also of finite type. In fact, 
we will conclude by proving that the kernel of each group morphism $\xi_r : \C \to \mr{Sym}_{s_r}$, $t\mapsto \overset{{}_\bullet}{\exp} t_rX_r$  is of finite type. Since Remark~\ref{exps} allows us to work on a transversal, we can use Proposition~\ref{cent} to describe  $\ker (\xi_r)$ as the kernel of the morphism  $\C\to C(h_{r})/\langle h_{r} \rangle$ given by $t\mapsto [\exp tX_{r}]$, which is of finite type thanks to Lemma~\ref{surj}.

\medskip
\paragraph{\textbf{Step (\ref{recallSL})} }% \label{stepequising}
Let $P$ be a  holomorphic connected manifold and $t_0$ a point of~$P$.  A \emph{deformation
of $\F$} with parameter space the manifold $P$ pointed at $t_0$, is a germ  along all $\{0\}\times P$ of a $1$-dimensional holomorphic  singular foliation
$\F_P$ defined on a open neighborhood of  $\{0\}\times P$ in $ \C^2\times P$, which  is  locally  tangent to the  fibers of the projection $\pi_P: \C^2\times P\to P$ and such that $\F$ is equal to the restric\-tion of $\F_P$ to $\C^2\times \{t_0\}$, with the identification  $\C^2\iso \C^2\times \{t_0\}$,  $(x,y)\mapsto (x,y,t_0)$.

 We say that $\F_P$ is \emph{equireducible} if there exists a map $E_{\F_P} :\mathcal{M}\to \C^2\times P$ obtained by composition of blow-up maps $E_j : \mathcal{M}_{j+1}\to \mathcal{M}_{j}$ fulfilling:
\begin{enumerate}
 \item  each center of blow-up $C_{j}\subset \mc M_{j}$ of $E_j$ is biholomorphic to $P$ by  the map $\pi_j := \pi_P\circ E_0\circ E_1\circ\cdots\circ E_{j-1}:\mc M_j\to P$, 
 \item   the singular locus of the foliation $E_{\F_P}^{\ast}\F_P$ is smooth, contained in the exceptional divisor $\mc E_{\F_P}:=E_{\F_P}^{-1}(\{ 0 \}\times P)$  and the restriction  of $\pi_P\circ E_{\F_P}$ to each of its connected component is a biholomorphism onto $P$,
 \item the restriction of $E_{\F_P}$ to  $\mc M_t:=(\pi_P\circ E_{\F_P})^{-1}(t)$ is exactly the minimal reduction map of the  foliation $\F_t$ on $\C^2\times \{t\}$ induced by $\F_P$;
 \end{enumerate}
Notice that $\mc E_{\F_P}$ is a topological product over $P$, i.e.  there is a homeomorphism $\Phi_P : \mc M_{t_0}\times P \iso \mc M$ such that $\pi_P\circ \Phi_P$ is the second projection map. By identifying $\C^2\times \{t_0\}$ with $\C^2$, each marking
$f:\mc E\to\mc E_{\F_{t_0}}$ of $\F_{t_0}$ by $\mc E^\diamond$ extends via $\Phi_P$  to  markings $f_t : \mc E\to \mc E_{\F_{t}}\subset \mc M_t$ of  $\F_t$, $t\in P$, defining in this way a map   
\[
P\to \mathrm{Mod}([\F^{\diamond}])\,,\quad t\mapsto [\F_t,f_t]
\,.\] 
On the other hand,  given a point $t'\in P$, for each base point $o_D$ in a component $D$ of $\mc E$ introduced at the beginning of Section~\ref{Symmetry-tree-group}, let us choose a $(1+\dim P$)-dimensional submanifold  $\Delta_{P, D}$ of $ \mc M$, transverse to $f_{t'}(D)$ at the point $f_{t'}(o_D)$. 
The representation of  $\F_P$-holonomy of the leaf $f_{t'}(D\setminus \Sigma)$  defines a representation $\mc H_{P,D}^{t'}$ of the fundamental group  $\pi_1(D\setminus \Sigma, o_D)$ in the group $\mathrm{Diff}(\Delta_{P,D}, f_{t'}(o_D))$ of germs  of holomorphic automorphisms of $(\Delta_{P,D}, f_{t'}(o_{D}))$.

\begin{defin}\label{def-SL} 
We say that $\F_P$ is \emph{SL-equisingular at a point  $t'$ of $ P$} if 
\begin{enumerate}
\item $\F_P$ is equireducible,
\item for $t\in P$ sufficienty close to $t'$ and for each $(s, D)\in \msl{Ed}_{\msl{A}_{\F^\diamond}}\times \msl{Ve}_{\msl{A}_{\F^\diamond}}$, $D\in \partial s$, the Camacho-Sad indices $CS(E_t^\ast\F_t, f_t(D), f_t(s))$ do not depend on $t$;
\item  there is a germ of biholomorphism $\psi : (\Delta_{P,D}, f_{t'}(o_D)) \iso (\C\times P, (0, t'))$ such that 
\begin{enumerate}
\item the composition of $\psi$ by the second projection $\C\times P\to P$ is equal to $\pi_P\circ E_{\F_P}$ restricted to $\Delta_{P,D}$;
\item for all $\gamma\in \pi_1(D\setminus\Sigma, o_D)$ the biholomorphism $(z,t)\mapsto \psi\circ \mc H_{P,D}^{t'}(\gamma)\circ \psi^{-1}(z,t)$ does not depend on $t$.
\end{enumerate}
\end{enumerate}
We say that $\F_P$ is \emph{SL-equisingular} if it is SL-equisingular at each point of~$P$.
\end{defin}

\paragraph{\textbf{Step (\ref{famille})}} Consider the elements $[\F_{i},f_{i}]$ of $\mr{Mod}([\F^{\diamond}])$, $i\in \mbf D$, given in the statement of Theorem~\ref{thm}. By Isomorphism (\ref{eq5}) they are represented by cocycles $\mbf c^{i}=(c^{i}_{D,s})\in Z^{1}(\AR,\mr{Sym})$.
Now we fix an orientation $\prec$ of  $\A$ and as  in the proof of Theorem~\ref{isoH1symH1aut}, we lift this cocycle to a cocycle $(\varphi_{D,s}^{i})\in  Z^{1}(\AR,\mr{Aut})$ and we continue to denote by $s_1,\ldots,s_\tau$ the edges that produce the complete system $a_1,\ldots,a_\tau$ of active vertices used in step (\ref{sommets}). 
We define then $ (\varphi_{D,s}^{i,t})\in Z^{1}(\A,\mr{Aut})$ by setting
\begin{align*}
&\varphi_{D,s}^{i,t}  = 
\left\{
\begin{array}{rcl}
\mr{id} & \textrm{if} & s\in\E\setminus\ER\,,\\
\varphi_{D,s}^{i} & \textrm{if} & s\in\ER\setminus\{s_{1},\ldots,s_{\tau}\}\,,\\
\varphi_{D,s}^{i}\circ\exp t_{j}X_{j} & \textrm{if} & s=s_j\in\{s_1,\ldots,s_{\tau}\}\,,
\end{array}
\right.\\
&\varphi_{D',s}^{i,t}  =(\varphi_{D,s}^{i,t})^{-1}\,,
\end{align*}
for $s\in \E$ with $\partial s=\{D,D'\}$ and $D\prec D'$.
Using these cocycles, for each $i\in \mbf D$ we glue suitable neighborhoods $W_{D}$ of $D\times\C^{\tau}$ inside $M_{\F}\times\C^{\tau}$.  
We obtain a manifold $\mc M_{i}$ endowed with a submersion map onto $\C^{\tau}$, a flat divisor  over $\C^{\tau}$ and  a foliation by curves tangent to the fibers of the submersion and to the divisor. 
By the same arguments used in Theorem~\ref{p1} we obtain a open neighborhood of $\{0\}\times\C^{\tau}$ in $\C^{2}\times\C^{\tau}$ and on this neighborhood  an holomorphic vector field defining a one-dimensional equireducible foliation tangent to the fibers of the projection onto $\C^{\tau}$, whose  singular locus is $\{0\}\times\C^{\tau}$.
By construction, after equireduction the exceptional divisor, as intrinsic analytic space,  is holomorphicaly trivial over $\C^\tau$ and   along each of its  irreducible components   the reduced foliation is holomorphically trivial. 
Hence we have obtained a  SL-equisingular deformation $\F^{U}_{i,t}$ of $\F_{i}$, see \cite{MS2},  
and biholomorphisms $h_{i,t} : \mathcal{E}_{\F_i}\iso \mathcal{E}_{\F_{i,t}^U}$, $i\in \mbf D$. 
We define the markings $f_{i,t}^U : \mathcal{E}\iso \mathcal{E}_{\F_{i,t}^U}$ by $f_{i,t}^U:= h_{i,t}\circ f_i$. 

Notice that  by the  construction of $\Lambda$ in step (iv), $\alpha(t)$ is represented in $H^{1}(\A,\mr{Aut})$ by the following 1-cocycle with support in $\ARU$:
$$a^t_{D,s} :=
\left\{\begin{array}{rcl}
\mr{id} & \textrm{if} & s\in\E\setminus\ER\,,\\
\mathrm{id} & \textrm{if} & s\in\ER\setminus\{s_{1},\ldots,s_{\tau}\}\,,\\
\exp t_{j}X_{j} & \textrm{if} & s=s_{j}\,.
\end{array}\right.
$$
Thanks to Theorems~\ref{isoassymarsym} and~\ref{strutgaaut} we have  in $H^{1}(\AR,\mr{Sym})$ the equality
$$
\left[(\overset{{}_\bullet}{\varphi}_{D,s}^{i,t}) \right] = \left[\mbf c^{i} \right] \cdot \left[ (\overset{{}_\bullet}{a}^{t}_{D,s}) \right] \,,
$$
where $\overset{{}_\bullet}{\varphi}_{D,s}^{i,t} $ and $\overset{{}_\bullet}{a}^{t}_{D,s}$ denote the classes of $\varphi_{D,s}^{i,t}$ and $a^{t}_{D,s}$ in $\mathrm{Sym}_s$.  The abelian group structure on $\mr{Mod}([\F^{\diamond}])$ being induced by the one on $H^{1}(\AR,\mr{Sym}^{\F})$ by~(\ref{eq5}),   the previous equality  proves that 
in $\mr{Mod}([\F^{\diamond}])$ we have 
$$
[\F^{U}_{i,t},f_{i,t}^U]=\alpha(t)\cdot [\F_{i},f_{i}]\,,\quad i\in \mbf D \,.$$

\section{Appendix}
Let us denote by $B_r\subset \C^2$ the ball $\{|x|^2+|y|^2\leq r\}$ and for a curve $S\ni 0$ in $\C^2$ let us call  \emph{Milnor ball} any ball $B=B_R$ such that  $S\cap B\setminus \{0\}$ is regular and meets transversely each sphere  $\partial B_r$, $0<r\leq R$. We fix a germ $\F$ at $0\in \C^2$  of a singular holomorphic foliation.
\begin{defin}\label{apropriateSeparatices} A germ of an invariant curve $S$ at $0\in\C^2$ will be be called \emph{$\F$-appropriate} 
if $S$ is invariant by $\F$,  contains  all the  \emph{isolated separatrices}\footnote{i.e. their  strict transforms meet invariant components of the exceptional divisor.} and its strict transform by the reduction of $\F$ meets any  dicritical component $D$ of $\mc E_\F$-valency one, i.e.   $\mr{card}(D\cap \mr{Sing}(\mc E_\mc  \F))=1$.
\end{defin}

The following incompressibility property is proven under some additional assumptions in \cite{MM1}, \cite{MM4} and an optimal version is obtained  by L. Teyssier in  \cite{Loic1}:

\begin{teo}\label{incompressibiliteTh}
If $\F$ is a generalized curve and $S$ is an $\F$-appropriate curve in a Milnor ball $B$. Then there exists a fundamental system $\mc U=(U_n)_{n\in \N}$ of open neighborhoods of $S$ in $B$ such that  for each $n\in \N$
\begin{enumerate}
\item the inclusion map $U_n\hookrightarrow B$ induces an isomorphism between the fundamental groups of $U_n\setminus S$ and $B\setminus S$;
\item for each leaf $L$ of the foliation $\F_{| (U_n\setminus S)}$ the inclusion map  $L\hookrightarrow U_n\setminus S$ induces an injective morphism $\pi_1(L, \cdot) \hookrightarrow \pi_1(U_n\setminus S,\cdot)$;
\item there is a finite union of curves on $U_n\setminus S$ whose preimage $\Omega$ in the universal covering  $\wt{U}_n^\ast$ of $U_n\setminus S$ is a  disjoint union of embedded conformal discs $\Omega_\alpha$ such that  each leaf $L$ of the foliation $\wt\F_n$  induced by $\F$ on $\wt{U}_n^\ast$ meets $\Omega$ and  $\mr{card}(L\cap \Omega_\alpha) \leq1$ for any $\alpha$. 
\end{enumerate}
\end{teo}

\begin{obs}\label{2consequences}
Two direct consequences of this result are the simple connectedness of the leaves  of the foliation $\wt\F_n$ and a structure of (non Hausdorff) Riemann surface of its the leaf space 
$\wt{\mc Q}_{U_n}^\F$,  whose atlas is given by the transversals $\Omega_\alpha$ . 
\end{obs}
\textit{Now, in the sequel we   consider  the following situation: $\F$ and $\mc G$ are  two topologically equivalent germs  of foliations at $0\in\C^2$ and $\psi:(\C^2,0)\to(\C^2,0)$ is a germ of homeomorphism that conjugates them, $\psi^{\ast}\mc G=\F$.}\\

Previous Theorem~\ref{incompressibiliteTh} will allow us to extend Theorem~1.6 of \cite{MM3} with weaker assumptions.
\begin{teo}\label{newliftingconjugacies}
If $\F$ is a generalized curve fulfilling  Conditions (TC) and  (TR)  stated in the introduction, then 
there exists a germ of a homeomorphism $\phi:(\C^2,0)\to(\C^2,0)$ such that: 
\begin{enumerate}
\item\label{extensionrelevement} the lifting   $E_{\mc G}^{-1}\circ \phi\circ E_\F$ of $\phi$  through  the reduction maps of $\F $ and $\mc G$ 
extends to the exceptional divisor as a germ of homeomorphism $\Phi:(M_{\F}, \mc E_{\F})\to(M_{\mc G},\mc E_{\mc G})$ along the exceptional divisors;
\item\label{holosing}  $\Phi$ is holomorphic at each non-nodal singular point of $\F^\sharp$;
\item\label{transversalementholomorphe} $\Phi$ is  transversely holomorphic at each point of the exceptional divisor which is regular for $\F^\sharp$ and  not contained in a dicritical component.
\end{enumerate}
\end{teo} 
Beside the possible existence of dicritical components, the new difficulty  of this theorem lies in the fact that $\psi$ may not be transversely holomorphic  on a whole neighborhood of 0.
Indeed, let us denote by $\mc E_\F^{\mr{cut}}$ the disjoint union of the 
cut-components of $\mc E_\F$. Conditions (TC) and (TR) of the introduction do not exclude the existence of  exceptional cut-components of $\mc E_\F$, i.e.  irreducible components containing at most two singular points of $\F^\sharp$. Around such cut-components, in a meaning which will be specified later, the conjugation $\psi$ may not be transversely holomorphic. \\

Because $\F$, and therefore $\mc G$  are generalized curves \cite{CLNS} $\psi$  defines one-to-one correspondences 
\begin{equation}\label{correspond}
D\mapsto D'  \quad\hbox{ \rm and } \quad s\mapsto s'
\end{equation}
between the irreducible components of the exceptional divisors $\mc E_\F$ and $\mc E_{\mc G}$ and  between the points of $\mr{Sing}(\mc E_\F)\cup\mr{Sing}(\F^\sharp)$ and $\mr{Sing}(\mc E_\mc G)\cup\mr{Sing}(\mc G^\sharp)$. Moreover we have the equalities of intersection numbers:
\begin{equation}\label{equalityintersections}
(D'_1,D'_2)=(D_1,D_2) \quad {\rm for} \quad D_1\mapsto D_1',\quad D_2\mapsto D_2'\,.
\end{equation}
Indeed the reduction map of a foliation is also equal to the reduction map of the curve formed by all its isolated separatrices and two dicritical separatrices for each dicritical component of the exceptional divisor. Thus equality (\ref{equalityintersections}) follows from classical topological properties of germs of  curves.\\

Let us point out  that property (\ref{holosing}) in Theorem~\ref{newliftingconjugacies} implies  equality of Camacho-Sad indices  of these foliations. These equalities will be strongly used in the proof of  the above theorem and in fact we need to prove them first.

\begin{lema}\label{equalitiesCS}
Under the  assumptions of Theorem~\ref{newliftingconjugacies}, if
$s\in D\subset \mc E_\F$ and $s'\in D'\subset \mc E_{\G}$ correspond by (\ref{correspond}), then  
\begin{equation}\label{equalityCS}
\mr{CS}( \F^\sharp, D, s)=\mr{CS}(\mc G^\sharp, D',  s')\,.
\end{equation}
\end{lema} 
\begin{proof} It is enough to  prove these equalities when $s$ is an  intersection point   of the strict transform of a  separatrix $S$ of $\F$ with  an irreducible component $D$ of $\mc E_{\F}$. Indeed according to an extension of \cite{OBRGV}  given in \cite[Lemma~1.9]{MM4} or in \cite[Theorem~8]{CamachoRosas}, there is such a point on any cut-component $\mc C$ of $\mc E_\F$. Thus the induction given in 
\cite[\textsection~7.3]{MM3} will remain valid and equalities (\ref{equalityCS}) will be satisfied at every singular point of the foliation. 

We distinguish three possibilities.
\begin{enumerate}[a)]
\item \textit{$\lambda:=\mr{CS}( \F^\sharp, D, s)$ is an irrational real number.}  If $\lambda$ is positive, $s$ is a nodal singular point, and (\ref{equalityCS}) was obtained by R. Rosas in \cite[Proposition~13]{Rosas}. Another proof is given in 
 \cite[Theorem~1.12]{MM4} that remains valid for $\lambda<0$.
\item \textit{$\mc C$ is not exceptional and $s$ is not a nodal singular point.} Then by using (TR)   and thanks to an extended version \cite{MM3} of the rigidity theorem in \cite{Rebelo},  $\psi$ is transversely holomorphic on the image by $E_\F$ of a neighborhood of $\mc C$ and specifically at the points of the separatrix $S$. In this case the proof of (\ref{equalityCS}) given in \cite[Chapter 2]{MM3} remains valid. 
\item \textit{$\mc C$ is exceptional and $s$ is not a nodal point.} Then $\mc C$ is a ''chain''
%\footnote{\El{\texttt{peut \?etre ''g\'eod\'esique'' au lieu de ''chaine''?}}}  
of components of $\mc E_\F$, $\mc C=D_1\cup\cdots\cup D_\ell$, $\ell\geq 1$, 
$D_i$ meeting $D_{i+1}$ in one point $s_{i}$ and  $D_i\cap D_j=\emptyset$ if $|i-j|\neq 1$. Perhaps $\mc C$ meets several dicritical components of $\mc E_\F$, but we only have two possibilities fulfilling Assumption~(TC):
\begin{enumerate}[(i)]
\item\label{unseulpoint} $s\in D_1$ and the other singular points of $\F^\sharp$ belonging to  $\mc C$ are $s_1,\ldots,s_{\ell-1}$;
\item\label{avecnoeud} $s\in D_1$, $D_\ell$ contains a nodal singular point  $ s_\ell\neq s_{\ell-1}$, the other singular points of $\F^\sharp$ belonging to  $\mc C$ being $s_1,\ldots,s_{\ell-1}$;
%moreover $s'\neq s$ if $l=1$.
\end{enumerate}
In case (\ref{unseulpoint}) using the classical index formula, we see that   $\mr{CS}(\F^\sharp, D_1, m_1)$ is given by a continuous fraction whose coefficients are the self-intersections $(D_i,D_i)$, $i=1,\ldots,\ell$; thus (\ref{equalityCS}) follows from (\ref{equalityintersections}).
In the same way we obtain in case (\ref{avecnoeud})  that $\mr{CS}(\F^\sharp, D_1, s)$ is an irrational (negative) real number, but this case was already examined above.
\end{enumerate}
\end{proof}
\noindent Because the dicriticity of a irreducible component   $D$ can be characterized by the vanishing  of  the Camacho-Sad indices along all the adjacent components at their intersection points with $D$, we have:
\begin{cor}\label{corespdic}
By Correspondences (\ref{correspond}) the image of a dicritical component, a exceptional cut-component, a non exceptional cut-component of $\F$ is respectively   a dicritical component, an exceptional cut-component, a non exceptional cut-component of $\mc G$. 
\end{cor}
\begin{dem2}{of Theorem~\ref{newliftingconjugacies}} 
This result  extends Theorem~5.0.2 of \cite{MM3} and we will sketch an inductive proof similar to that described in Chapter 8 of \cite{MM3}.
We proceed in four steps: first we extend to our new context the notion of monodromy; then 
we construct a conjugation $\Phi_1$ between $\F^\sharp$ and $\mc G^\sharp$ on a neighborhood of the union of all  non exceptional cut-components of $\mc E_\F$; in a third step  we define  a conjugation $\Phi_2$ along  the exceptional cut-component except at the nodal singularities; finally in the fourth and last step we extend and glue $\Phi_1$ and $\Phi_2$ at the nodal singularities and along the dicritical components.\\
 
\emph{- Step 1.}
 Let us fix a  $\F$-appropriate curve $S$ and  Milnor balls $B$ and~$B'$ for $S$ and~$S':=\psi(S)$. Let   $\wt B^\ast$ and  $\wt B'{}^\ast$ be universal coverings of $B\setminus S$ and $B'\setminus S'$ respectively.  We suppose that $\psi(B)\subset B'$,  we  choose a lifting $\wt\psi: \wt B^\ast\to\wt B'{}^\ast$ of $\psi$ and we denote by $\wt\psi_\ast : \Gamma\iso\Gamma'$  the isomorphisms induced by $\wt\psi$ between the deck transformation groups of these coverings.

As in  Chapter 3 of \cite{MM3} we call \emph{monodromy of $\F$}  the natural group morphism 
\[
\mf M_S^\F : \Gamma\to \mr{Aut}_{\underleftarrow{\mr{An}}}(\wt{\mc Q}_\infty^\F)
\subset \mr{Aut}_{\underleftarrow{\mr{Top}}}(\wt{\mc Q}_\infty^\F)
\] 
where with notations of Remark \ref{2consequences},   $\wt{\mc Q}_\infty^{\F}$ is the inverse system $(\wt{\mc Q}_{U_n}^\F)_{{}_{n\in \N}}$,  ${\underleftarrow{\mr{An}}}$ is the category of pro-objects associated to the category of analytic spaces and 
$\underleftarrow{\mr{Top}}$ is the category of pro-objets associated to the category of topological spaces and continuous functions. 
The monodromy 
\[
\mf M_{S'}^{\mc G} : \Gamma'\to \mr{Aut}_{\underleftarrow{\mr{An}}}(\wt{\mc Q}_\infty^{\mc G})
\subset \mr{Aut}_{\underleftarrow{\mr{Top}}}(\wt{\mc Q}_\infty^{\mc G})
\]
of $\mc G$ is defined in the same way, after the  choice of  $S':=\psi(S)$ as $\mc G$-appropriate curve. 
The conjugation $\wt \psi$ induces an automorphism $h_{\wt\psi}: \wt{\mc Q}_\infty^\F\iso \wt{\mc Q}_\infty^{\mc G}$ in the category $\underleftarrow{\mr{Top}}$; however, because here Condition (G) of \cite[page 406]{MM3} may not be satisfied, 
$h$ may not be $\mc N$-analytic in the sense of \cite[Definition~3.4.2]{MM3}. Thus we extend the notion of conjugation  of \cite[page 416]{MM3} by calling  \emph{topological conjugation between the monodromies} $\mf M_S^\F$ and $\mf M_{S'}^{\mc G}$ any pair $(\mf g,h)$ where $\mf g:\Gamma\iso\Gamma'$ is a group isomorphism and $h:\wt{\mc Q}_\infty^{\F}\iso\wt{\mc Q}_\infty^{\mc G}$ is an isomorphism in the category $\underleftarrow{\rm{Top}}$ such that 
$h_\ast\circ\mf M_S^\F=\mf M_{S'}^{\mc G}\circ \mf g$, with $h_\ast: \mr{Aut}_{\underleftarrow{\mr{Top}}}(\wt{\mc Q}_\infty^{\F})\to \mr{Aut}_{\underleftarrow{\mr{Top}}}(\wt{\mc Q}_\infty^{\mc G})$, $\varphi\mapsto h\circ\varphi\circ h^{-1}$. The notions of \emph{geometric conjugation} and \emph{realization of geometric conjugation} given in \cite[Definitions 3.3.3 and 3.6.1]{MM3} which are already defined in the topological category, remain unchanged for topological conjugations. 
Now we highlight the fact that 
\begin{enumerate}[A.]\it%
\item\label{gc} \textit{the pair $(\wt\psi_\ast, h_{\wt\psi})$ is a geometric topological  conjugation between $\mf M_S^\F$ and $\mf M_{S'}^{\mc G}$ that is realized on germs $(\Delta,S)$ and $(\psi(\Delta), S')$ for any subset $\Delta$ of $B$ meeting $S$;}
\item \textit{Theorem~4.3.1 of \cite{MM3} which gives a relation between  conjugations of monodromies and conjugation of holonomies, remains valid for topological conjugations;}
\item \textit{Key Lemma~8.3.2 of  \cite{MM3} of extension of realizations, is also valid when the conjugation of holonomies $(\mf g,h)$ is topological, the realization $\phi_S: (T,c)\to (T',c')$ remaining biholomorphic.}
\end{enumerate}
Assertion \textit{\ref{gc}} that  extends Assertion (2) of  Invariance Theorem~5.0.1 of \cite{MM3}, is immediate; the other two follow directly from the  proofs in \cite{MM3}.\\

 \emph{- Step 2.} We will perform an induction process as in \cite[Chapter 8]{MM3}.

 \indent a) First we define  \emph{elementary pieces} as in  \cite[\textsection 8.2]{MM3}, however the unions defining the real hypersurfaces $\mc H$ and  $\mc H'$  are now indexed by the set of   all  singular points of the exceptional divisor and of the foliation. In this way we  have
three new types of elementary pieces:
\begin{enumerate}[(i)]
\item $K_D$ with $D$ an invariant component of the exceptional divisor meeting a dicritical component,
\item\label{dicnondic} $K_s$ with $s$  an intersection point of a dicritical and a non dicritical component of the exceptional divisor,
\item\label{compdic} $K_D$, with $D$ a dicritical component.
\end{enumerate}

\indent b) To start  the induction, we proceed as  in \cite[\textsection 8.4]{MM3} but with an \emph{$\F$-suitable collection of  transversals}  $(\Delta_{\mc C})_{\mc C\in\mf C}$ in the meaning that it  is obtained in the following way: $\mf C$ is the set of non exceptional cut-components of $\mc E_\F$;  for each $\mc C\in \mf C$ we choose one  separatrix $S_{\mc C}$ whose strict transform meets $\mc C$ and an embedded  conformal disc $\Delta_{\mc C}$ transversal to $S_\mc C$ at a regular point; these discs are small enough so their pairwise intersections are empty and they are tranversal to the foliation.  Up to suitable foliated isotopies  we suppose that $\psi(\Delta_{\mc C})$ are also conformal embedded discs. Thus  thanks to Corollary \ref{corespdic} the collection given by $\Delta'_{\mc C'}:=\psi(\Delta_{\mc C})$,  $\mc C\in \mf C$, where $\mc C'$ is the cut-component corresponding to $\mc C$ by (\ref{correspond}), forms a $\mc G$-suitable collection of transversals. We begin the induction with the realization $(\psi_{|\Delta}, \wt\psi_{|\wt\Delta}, h_{\wt\psi})$ of the geometric conjugation $(\psi_\ast, h_{\wt\psi})$ on $\Delta:=\cup_{\mc C\in \mf C}\Delta_{\mc C}$ and $\Delta':=\cup_{\mc C'\in \mf C'}\Delta_{\mc C'}$. Because the cut-components are non exceptional and thanks to the extended version of the main theorem of \cite{Rebelo}, Rigidity condition (TR) of the introduction implies the analyticity of the restrictions $\psi_{|\Delta_{\mc C}}$. Then we finish the base case of the induction by the construction of a geometric representation  of the conjugation $(\psi_\ast, h_{\wt\psi})$ satisfying condition (23) of \cite[Extension Lemma~8.3.2]{MM3}, as  in \cite[\textsection 8.4]{MM3}.

\indent c) The process of the induction described in \cite[\S8.1]{MM3} and started in this way,  stops when it  requires  to make an extension to an elementary piece of type (\ref{dicnondic}) or to an elementary piece containing a nodal singularity belonging to $\mr{Sing}(\mc E_\F)$. In this way we obtain the announced conjugation $\Phi_1$.\\ 

\emph{- Step 3.} Let $\mc C$ be an exceptional cut-component of $\mc E_\F$ and let us keep the notations introduced in the proof of Lemma~\ref{equalitiesCS}: $\mc C=D_1\cup\cdots\cup D_\ell$ has two possible configurations (\ref{unseulpoint}) and (\ref{avecnoeud}). 

In the first case (\ref{unseulpoint}), $D_\ell$ contains only one singular point of the foliation, hence there is a holomorphic first integral defined on a neighborhood of $\mc C$ and specifically the foliation is linearizable at each singular point. Thus by equality of Camacho-Sad indices given by Lemma \ref{equalitiesCS} the considered foliations  are locally holomorphically conjugated at the singular points corresponding by~(\ref{correspond}). The equalities (\ref{equalityintersections})  of self-intersections of the components of $\mc C$ and $\mc C'$ allow us to glue these conjugacies and to obtain a homeomorphism  defined on a neighborhood of $\mc C$. We leave to the reader the details of this construction.

The situation in case (\ref{avecnoeud}) is similar: we have again equality of Camacho-Sad indices and therefore local conjugacies, and then equality of self-inter\-sections allowing to glue and to obtain a global $\mc C^0$ conjugation.
\\

\emph{- Step 4.}  On the elementary pieces $K_s$ corresponding to a nodal singular point $s$, we perform the gluing of the homeomorphisms already constructed by the process described  in \cite[\textsection 8.5]{MM3}. It remains to extend the obtained homeomorphisms to the dicritical components. Notice that in all the above constructions the homeomorphisms can be built by respecting  the dicritical components meeting their definition domains. Finally we arrive at the following situation described in \cite[page 147]{MM4}: if we identify tubular neighborhoods of  corresponding dicritical components $D\subset \mc E_\F$ and $D'\subset\mc E_{\mc G}$ with a same\footnote{It is possible because $D$ and $D'$ have same negative self-intersection.}
tubular neighborhood of the zero section of the normal bundle of $D$, the corresponding foliations being identified with the natural normal fibration, we have to extend to the whole $D$ a continuous map $g$ from 
a union $\mc K$ of disjoint closed discs to the group $\mr{Aut}_0(\C,0)$ of germs of homeomorphisms of $(\C,0)$. This can be easily made by extending $g$ to a union of bigger discs $\mc K'$ being a constant automorphism on $\partial \mc K'$. This ends the proof of Theorem~\ref{newliftingconjugacies}.
\end{dem2}

\bibliographystyle{plain}

\end{document}